\documentclass[leqno,11pt]{amsart}

\usepackage[letterpaper,margin=1in]{geometry}
\usepackage{arydshln}
\usepackage[all]{xy} 
\usepackage{amsmath, amssymb, amsfonts, latexsym, mdwlist, amsthm, amscd,mathabx,braket}
\usepackage{subfig}
\usepackage{graphicx}
\usepackage{wrapfig}
\usepackage{etex}
\usepackage{tikz-cd}
\usepackage[colorlinks=true, citecolor=blue, urlcolor=blue, linkcolor=blue, pagebackref]{hyperref}

\usepackage{tikz}
\usetikzlibrary{calc,trees,positioning,arrows,chains,shapes.geometric,%
    decorations.pathreplacing,decorations.pathmorphing,shapes,%
    matrix,shapes.symbols}

\tikzset{
>=stealth',
  punktchain/.style={
    rectangle,
    rounded corners,
    draw=black, thick,
    minimum height=3em,
    text centered,
    on chain},
  line/.style={draw, thick, <-},
  eLement/.style={
    tape,
    top color=white,
    bottom color=blue!50!black!60!,
    minimum width=8em,
    draw=blue!40!black!90, very thick,
    text width=10em,
    minimum height=3.5em,
    text centered,
    on chain},
  every join/.style={->, thick,shorten >=1pt},
  decoration={brace},
  tuborg/.style={decorate},
  tubnode/.style={midway, right=2pt},
}

\usepackage{paralist}
\usepackage{enumitem} 
\setdefaultenum{(1)}{(a)}{(i)}{}
\setlist[enumerate,1]{label={\upshape(\arabic*)}}
\setlist[enumerate,2]{label={\upshape(\alph*)},ref=\theenumi\upshape(\alph*)}
\setlist[enumerate,3]{label={\upshape(\roman*)},ref=\theenumi\theenumii\upshape(\roman*)}
\usepackage[capitalize]{cleveref}
\crefname{Prop}{Proposition}{Propositions}
\crefname{Thm}{Theorem}{Theorems}
\crefname{Lem}{Lemma}{Lemmas}
\crefname{enumi}{Case}{Cases}

\def\dim{\mathop{\mathrm{dim}}\nolimits}

\def\min{\mathop{\mathrm{min}}\nolimits}

\def\MG13{\ensuremath{{\mathcal M}_{\Gamma_1(3)}}}
\def\tildeMG13{\ensuremath{\widetilde{\mathcal M}_{\Gamma_1(3)}}}

\newcommand\TFILTB[3]{%
\xymatrix@=1pc{
{0 = {#1}_0} \ar[rr]&&
{{#1}_1} \ar[rr]\ar[ld] &&
{{#1}_2} \ar[r]\ar[ld] &
{\cdots} \ar[r] & { {#1}_{#3-1}} \ar[rr] &&
{{#1}_{#3} = {#1}} \ar[ld]
\\
& *{{#2}_1} \ar@{.>}[ul] &&
{{#2}_2} \ar@{.>}[ul] & &&&
{{#2}_{{#3}}} \ar@{.>}[ul]
}}

\makeatletter
\newtheorem*{rep@theorem}{\rep@title}
\newcommand{\newreptheorem}[2]{%
\newenvironment{rep#1}[1]{%
 \def\rep@title{#2 \ref{##1}}%
 \begin{rep@theorem}}%
 {\end{rep@theorem}}}
\makeatother

\newtheorem{Thm}{Theorem}[section]
\newreptheorem{Thm}{Theorem}
\newtheorem{Prop}[Thm]{Proposition}

\newtheorem{Lem}[Thm]{Lemma}

\newtheorem{Cor}[Thm]{Corollary}
\newreptheorem{Cor}{Corollary}

\newreptheorem{Con}{Conjecture}

\newtheorem*{theorem*}{Theorem}
\newtheorem*{lemma*}{Lemma}
\newtheorem*{proposition*}{Proposition}
\newtheorem*{conjecture*}{Conjecture}
\newtheorem*{corollary*}{Corollary}
\newtheorem*{problem*}{Problem}

\newtheorem{Thm-int}{Theorem}

\theoremstyle{definition}
\newtheorem{Def-s}[Thm]{Definition}
\newtheorem{Def}[Thm]{Definition}
\newtheorem{Rem}[Thm]{Remark}

\newtheorem{Ex}[Thm]{Example}

\newcommand{\ignore}[1]{}
\begin{document}

\title{Multi-rigidity of Schubert classes in partial flag varieties}
\author{Yuxiang Liu ${}^{1}$, Artan Sheshmani${}^{1,2,3}$ and  Shing-Tung Yau$^{2,4}$}

\address{${}^1$ Beijing Institute of Mathematical Sciences and Applications, No. 544, Hefangkou Village, Huaibei Town, Huairou District, Beijing 101408}

\address{${}^2$  Massachusetts Institute of Technology, IAiFi Institute, 77 Massachusetts Ave, 26-555. Cambridge, MA 02139, artan@mit.edu}

\address{${}^3$ National Research University Higher School of Economics, Russian Federation, Laboratory of Mirror Symmetry, NRU HSE, 6 Usacheva str.,Moscow, Russia, 119048}
\address{${}^4$ Yau Mathematical Sciences Center, Tsinghua University, Haidian District, Beijing, China}


\begin{abstract}
In this paper, we study the multi-rigidity problem in rational homogeneous spaces. A Schubert class is called multi-rigid if every multiple of it can only be represented by a union of Schubert varieties. We prove the multi-rigidity of Schubert classes in rational homogeneous spaces. In particular, we characterize the multi-rigid Schubert classes in partial flag varieties of type A, B and D. Moreover, for a general rational homogeneous space $G/P$, we deduce the rigidity and multi-rigidity from the corresponding generalized Grassmannians (correspond to maximal parabolics). When $G$ is semisimple, we also deduce the rigidity and multi-rigidity from the simple cases.
\end{abstract}

\maketitle


\tableofcontents

\section{Introduction}
The integer cohomology of a rational homogeneous variety has a free basis in terms of Schubert classes. A Schubert class is called multi-rigid if every multiple of it can only be represented by a union of Schubert varieties. In this paper, we study the multi-rigidity problem in rational homogeneous spaces. In particular, we characterize the multi-rigid Schubert classes in partial flag varieties of type A, B and D.

There is another type of rigidity: a Schubert class is called rigid if every representative of it can only be represented by Schubert varieties. In prequel to current article, we investigated the rigidity problem in rational homogeneous spaces \cite{YL3}. Particularly, we classified the rigid Schubert classes for partial flag varieties and orthogonal partial flag varieties. The current article will focus on the multi-rigidity problem, which is a stronger property than rigidity, therefore, the results of the current article do also imply the rigidity of Schubert classes.

The classical way to deal with the multi-rigidity problem is to make use of the Schur differential systems, which were introduced by Walters \cite{Walter} and Bryant \cite{RB2000} in 2000. Based on their works, Hong \cite{Ho1}, \cite{Ho2}. Robles and The \cite{RT} characterized the multi-rigid Schubert classes in compact Hermitian symmetric spaces, including Grassmannians $G(k,n)$ and spinor varieties $OG(k,2k)$. In this paper, we will use an algebro-geometric approach and make no use of differential systems. 

Let $G$ be a connected, simply connnected and semisimple algebraic group over $\mathbb{C}$. Let $T$ be a maximal torus in $G$ and let $B$ be a Borel subgroup containing $T$. Let $P$ be a parabolic subgroup containing $B$. The homogeneous space $X=G/P$ has a structure of projective variety. 

The semi-simple Lie groups are classified by the Dynkin diagrams. We first study the cases when $G$ is a group of classical type, then we generalize the results to general rational homogeneous spaces. First we introduce partial flag varieties  explored in current article.

\subsection{Partial flag varieties of type A}\label{introductionA}
Let $V$ be a complex vector space of dimension $n$. The partial flag variety $F(d_1,...,d_k;n)$ parameterizes all $k$-step partial flags $\Lambda_1\subset...\subset \Lambda_k$ in $V$, where $\dim(\Lambda_i)=d_i$, $1\leq i\leq k$. It is a smooth projective variety of dimension $\sum_{i=1}^k d_i(d_{i+1}-d_{i})$. (Here we set $d_{k+1}=n$.) The partial flag variety $F(d_1,...,d_k;n)$ can be identified with the homogeneous space $G/P$, where $G=SL(V)$. It is explained in \S\ref{sec-prelim}.

The Schubert varieties in $F(d_1,...,d_k;n)$ can be defined using a double index $a^\alpha$, where $a$ is a strictly increasing sequence of positive integers of length $d_k$:
$$1\leq a_1<...< a_{d_k}\leq n,$$
and $\alpha$ is a sequence of the same length taking the values from $\{1,...,k\}$ such that $\#\{j|\alpha_j\leq i\}=d_i$, $1\leq i\leq k$. Set $\mu_{i,t}:=\#\{s|a_s\leq a_i,\alpha_s\leq t\}$. Given a partial flag $F_\bullet$:
$$F_{a_1}\subset...\subset F_{a_{d_k}},$$ where $\dim(F_{a_i})=a_i$, the corresponding Schubert variety is defined as the following locus:
$$\Sigma_{a^\alpha}(F_\bullet):=\{(\Lambda_1,...,\Lambda_k)|\dim(F_{a_i}\cap \Lambda_t)\geq \mu_{i,t}\}.$$

\begin{Def}
Let $\sigma_{a^\alpha}$ be a Schubert class of $F(d_1,...,d_k;n)$. A sub-index $a_i$ is called {\em essential} if either $a_{i+1}\neq a_i+1$ or $\alpha_{i+1}>\alpha_{i}$. An essential $a_i$ is called {\em multi-rigid} if for every irreducible representative $X$ of $m\sigma_{a^\alpha}$, $m\in\mathbb{Z}^+$, there exists a linear subspace $F_{a_i}$ such that 
$$\dim(F_{a_i}\cap \Lambda_t)\geq \mu_{i,t}, \ \forall (\Lambda_1,...,\Lambda_k)\in X.$$
\end{Def}
The first main result of this paper is the classification of the multi-rigid subindices in partial flag varieties:
\begin{Thm}
Let $\sigma_{a^\alpha}$ be a Schubert class of $F(d_1,...,d_k;n)$. An essential $a_i$ is multi-rigid if and only if $a_i-a_{i-1}=1$ and one of the following conditions hold:
\begin{itemize}
\item $a_{i+1}-a_i\geq 3$;
\item $a_{i+1}-a_i=2$ and $\alpha_i<\alpha_{i+1}$;
\item $a_{i+1}-a_i=1$ and $\alpha_{i-1}<\alpha_{i+1}$ and either
\begin{enumerate}
\item $a_{i+2}-a_i\geq 3$; or
\item $\alpha_i<\alpha_{i+2}<\alpha_{i+1}$.
\end{enumerate}
\end{itemize}
\end{Thm}
As a corollary, we classify the multi-rigid Schubert classes in partial flag varieties. A Schubert class is multi-rigid if and only if all essential sub-indices are multi-rigid and the corresponding flag elements are compatible. To be more precise, we define a relation `$\rightarrow$' between two sub-indices: $a_i\rightarrow a_j$ if $i<j$ and $a_j$ is essential in $(\pi_t)_*(\sigma_{a^\alpha})$ for some $t\geq \min(\alpha_i,\alpha_j)$. This relation extends to a strict partial order (which we also denote by `$\rightarrow$') on the set of essential sub-indices by transitivity.
\begin{Cor}
A Schubert class $\sigma_{a^\alpha}\in A(F(d_1,...,d_k;n))$ is multi-rigid if and only if all essential sub-indices are multi-rigid and the set of all essential sub-indices is strict totally ordered under the relation `$\rightarrow$'.
\end{Cor}

\subsection{Orthogonal Grassmannians}
Let $q$ be a non-degenerate symmetric bilinear form on $V$. Let $Q$ be the quadric hypersurface defined by $q$. A subspace $W$ is called isotropic if $q$ vanishes on it, or equivalently $\mathbb{P}(V)$ lies on $Q$. The orthogonal Grassmannian $OG(k,n)$ parametrizes all isotropic subspaces of dimension $k$, unless $n=2k$, in which case the parameter space has two irreducible components and we let $OG(k,2k)$ denote one of the components. The orthogonal Grassmannian $OG(k,n)$ can be identified with $G/P$, where $G=SO(V)$ and $P$ is a maximal parabolic subgroup when $k\neq \frac{n}{2}-1$.

Given an isotropic subspace $W$, we denote $W^\perp$ its orthogonal complement with respect to $q$. Let $F_\bullet=F_1\subset...\subset F_{\left[n/2\right]}$ be an isotropic flag in $V$. When $n$ is even, the orthogonal complement of $F_{n/2-1}$ intersects $Q$ into a union of two maximal isotropic subspaces, one is $F_{n/2}$ and the other one belongs to the different irreducible components than $F_{n/2}$. By abuse of notation, we denote also by $F_{n/2-1}^\perp$ the maximal isotropic subspace in the orthogonal complement of $F_{n/2-1}$ other than $F_{n/2}$. 
\begin{Def}
A Schubert index for $OG(k,n)$ consists of two increasing sequences of integers $1\leq a_1<...< a_s\leq \frac{n}{2}$ and $0\leq b_1<...< b_{k-s}\leq \frac{n}{2}-1$ such that $a_i\neq b_j+1$ for all $1\leq i\leq s,1\leq j\leq k-s$. When $n=2k$, we further require that $s$ and $k$ have the same parity, i.e.
$$ s \equiv k \pmod{2} $$
\end{Def}

\begin{Def}
Given a Schubert index $(a_\bullet;b_\bullet)$ and an isotropic flag $F_\bullet$, the Schubert variety $\Sigma_{a;b}(F_\bullet)$ is defined to be the closure of the following locus:
$$\Sigma_{a;b}(F_\bullet):=\{\Lambda\in OG(k,n)|\dim(\Lambda\cap F_{a_i})= i, \dim(\Lambda\cap F_{b_j}^\perp)= k-j+1, 1\leq i\leq s, 1\leq j\leq k-s\}.$$
\end{Def}

\begin{Def}
Let $\sigma_{a;b}$ be a Schubert class for $OG(k,n)$. A sub-index $a_i$ is called {\em essential} if one of the following holds:
\begin{itemize}
\item $i<s$ and $a_{i}<a_{i+1}-1$
\item $n$ is odd and $i=s$;
\item $n$ is even, $i=s$ and $a_s+b_{k-s}\neq n-2$.
\end{itemize}

A sub-index $b_j$ is called {\em essential} if either $j=1$ or $b_{j}\neq b_{j-1}+1$. 

An essential sub-index $a_i$ is called {\em multi-rigid} if for every irreducible representative $X$ of the class $m\sigma_{a;b}$, $m\in\mathbb{Z}^+$, there exists an isotropic subspace $F_{a_i}$ of dimension $a_i$ such that for every $k$-plane $\Lambda$ parametrized by $X$,
$$\dim(\Lambda\cap F_{a_i})\geq i.$$
Similarly, an essential sub-index $b_j$ is called {\em multi-rigid} if for every irreducible representative of the class $m\sigma_{a;b}$, $m\in\mathbb{Z}^+$, there exists an isotropic subspace $F_{b_j}$ of dimension $b_j$ such that for every $k$-plane $\Lambda$ parametrized by $X$, $$\dim(\Lambda\cap F_{b_j}^\perp)\geq k-j+1.$$
\end{Def}
In \S \ref{Type BD}, we obtain the following characterization of multi-rigid sub-indices in $OG(k,n)$:
\begin{Thm}\label{intror}
Let $\sigma_{a;b}$ be a Schubert class for $OG(k,n)$. An essential sub-index $a_i$ is multi-rigid if $a_i$ is rigid and one of the following conditions hold:
\begin{enumerate}
\item $i<s$ and $a_{i-1}+1=a_i\leq a_{i+1}-3$;
\item $i<s$, $a_i=b_j$ for some $1\leq j\leq k-s$, $a_i\leq a_{i+1}-3$ and $b_j\leq b_{j-1}+3$.
\item $n$ is even, $i=s$, $a_s\leq \frac{n}{2}-3$ and $a_s=a_{s-1}+1$;
\item $n$ is even, $i=s$, $a_s= \frac{n}{2}-1$, $a_s=a_{s-1}+1$ and $b_{k-s}\neq\frac{n}{2}-1$;
\item $n$ is even, $i=s$, $a_s= \frac{n}{2}-2$, $a_s=a_{s-1}+1$ and $b_{k-s}\neq \frac{n}{2}-1$;
\item $n$ is even, $i=s$, $a_s=\frac{n}{2}$ and $b_{k-s}\leq \frac{n}{2}-4$;
\item $n$ is odd, $i=s$, $a_s\neq\left[\frac{n}{2}\right]-1$ and $a_s=a_{s-1}+1$;
\item $n$ is odd, $i=s$, $a_s=\left[\frac{n}{2}\right]-1$, $a_s=a_{s-1}+1$ and $b_{k-s}\neq\left[\frac{n}{2}\right]-1$.
\end{enumerate}
An essential sub-index $b_j$ is multi-rigid if $b_j$ is rigid and one of the following conditions hold:
\begin{enumerate}
\item $b_j<\frac{n}{2}-2$, $b_j\neq a_i$ for all $1\leq i\leq s$ and $b_{j+1}-1=b_j\leq b_{j-1}+3$;
\item $b_j<\frac{n}{2}-2$, $b_j=a_i$ for some $1\leq i\leq s$, $a_i\leq a_{i+1}-3$ and $b_j\leq b_{j-1}+3$;
\item $n$ is even, $j=k-s$, $b_{k-s}=a_s=\frac{n}{2}-2$ and $a_{s-1}=\frac{n}{2}-3$;
\item $n$ is even, $j=k-s$, $b_{k-s}=\frac{n}{2}-1$ and $b_{k-s-1}\leq\frac{n}{2}-4$;
\end{enumerate}
\end{Thm}

As an application, we obtain a classification of multi-rigid Schubert classes in $OG(k,n)$:
\begin{Cor}
Let $\sigma_{a;b}$ be a Schubert class for $OG(k;n)$. The Schubert class $\sigma_{a;b}$ is multi-rigid if all essential sub-indices satisfy one of the conditions in Theorem \ref{intror}. If $n$ is odd, then the converse is also true.
\end{Cor}

\subsection{Orthogonal partial flag varieties}Let $V$ and $q$ be defined as before. The orthogonal partial flag variety $OF(d_1,...,d_k;n)$ parametrizes all $k$-step partial flags $(\Lambda_1,...,\Lambda_k)$, where $\Lambda_i\subset \Lambda_{i+1}$ and $\Lambda_i$ are isotropic subspaces of dimension $d_i$, unless $n=2d_k$, in which case the parameter space has two irreducible components and we let $OF(d_1,...,d_k;n)$ denote one of the components. It generalizes the orthogonal Grassmannian varieties and can be identified with $G/P$, where $G=SO(V)$ and $P$ is a parabolic subgroup.

\begin{Def}
A Schubert index for $OF(d_1,...,d_k;n)$ consists of two increasing sequences of integers $$1\leq a_1^{\alpha_1}<...<a^{\alpha_s}_s\leq \frac{n}{2}$$ and $$0\leq b^{\beta_1}_1<...<b^{\beta_{k-s}}_{k-s}\leq \frac{n}{2}-1,$$ of length $s$ and $k-s$ respectively, $0\leq s\leq k$, such that 
\begin{enumerate}
\item $a_i\neq b_j+1$ for all $1\leq i\leq s,1\leq j\leq k-s$, and 
\item $d_t-d_{t-1}=\#\{i|\alpha_i=t\}+\#\{j|\beta_j=t\}$ for every $1\leq t\leq k$. (Here we set $d_0=0$.)
\end{enumerate}
When $n=2d_k$, we further require $s$ and $d_k$ have the same parity, i.e.
$$ s \equiv d_k \pmod{2}.$$
\end{Def}

Given a Schubert index $(a^\alpha;b^\beta)$ and a flag of isotropic subspaces $F_1\subset ...\subset F_{\left[n/2\right]}$ in $V$, set
$$\mu_{i,t}:=\#\{c|a_c\leq a_i,\alpha_c\leq t\},$$
$$\nu_{j,t}:=\#\{d|\alpha_d\leq t\}+\#\{e|b_e\geq b_j,\beta_e\leq t\}.$$
The Schubert variety $\Sigma_{a^\alpha;b^\beta}$ is then defined to be the closure of the following locus:
\begin{eqnarray}
\Sigma_{a^\alpha;b^\beta}:=\{(\Lambda_1,...,\Lambda_k)\in OF&|&\dim(F_{a_i}\cap \Lambda_t)\geq \mu_{i,t}, \alpha_i\leq t,\nonumber\\
& &\dim(F^\perp_{b_j}\cap \Lambda_t)\geq \nu_{j,t},\beta_j\leq t\}.\nonumber
\end{eqnarray}
\begin{Def}\label{defofrigid}
Let $(a^\alpha,b^\beta)$ be a Schubert index in $OF(d_1,...,d_k;n)$. A sub-index $a_i$ or $b_j$ is called {\em essential} if it is essential with respect to the class $(\pi_t)_*(\sigma_{a^\alpha;b^\beta})$ in $OG(d_t,n)$ for some $1\leq t\leq k$.

An essential sub-index $a_i$ is called {\em multi-rigid} if for every irreducible representative $X$ of $m\sigma_{a^\alpha;b^\beta}$, $m\in\mathbb{Z}$, there exists an isotropic subspace $F_{a_i}$ of dimension $a_i$ such that 
$$\dim(F_{a_i}\cap\Lambda_t)\geq \mu_{i,t}, \forall (\Lambda_1,...,\Lambda_k)\in X,\alpha_i\leq t\leq k.$$
An essential sub-index $b_j$ is called {\em multi-rigid} if for every irreducible representative $X$, there exists an isotropic subspace $F_{b_j}$ of dimension $b_j$ such that
$$\dim(F_{b_j}^\perp\cap\Lambda_t)\geq \nu_{j,t}, \forall (\Lambda_1,...,\Lambda_k)\in X,\beta_j\leq t\leq k.$$
\end{Def}
In \S \ref{Type BD}, we prove a Schubert sub-index is multi-rigid if it is multi-rigid with respect to the push-forward class:
\begin{Thm}
Let $\sigma_{a^\alpha;b^\beta}$ be a Schubert class for $OF(d_1,...,d_k;n)$. An essential $a_i$ (or $b_j$ resp.) is multi-rigid if $a_i$ (or $b_j$ resp.) is multi-rigid with respect to the Schubert class $(\pi_i)_*(\sigma_{a^\alpha;b^\beta})$ for some $i$.
\end{Thm}
As a corollary, we obtain the multi-rigidity of Schubert classes in orthogonal partial flag varieties:
\begin{Cor}
Let $\sigma_{a^\alpha;b^\beta}$ be a Schubert class of $OF(d_1,...,d_k;n)$. The Schubert class $\sigma_{a^\alpha;b^\beta}$ is multi-rigid if and only if all essential sub-indices are multi-rigid, and the corresponding isotropic subspaces form a partial flag.
\end{Cor}

\subsection{General homogeneous varieties}Let $G$, $T$, $B$, $P$ defined as before. Let $N(T)$ be the normalizer of $T$ in $G$ and let $N_P(T)$ be the normalizer of $T$ in $P$. Let $W=N(T)/T$ and $W_P=N_P(T)/T$ be the Weyl group of $G$ and $P$ respectively. Each coset in $W/W_P$ has a unique element of minimal length. Let $W^P$ be the set of minimal representatives of $W/W_P$. For every $w\in W^P$, the Schubert variety $\Sigma_{w}$ is defined to be the Zariski closure of $BwP/P$ in $G/P$ with the canonical reduced scheme structure.

By a theorem of Borel and Remmert, the homogeneous space $G/P$ with $G$ semi-simple can be decomposed into a direct product 
$$G/P=G_1/P_1\times...\times G_k/P_k,$$
where $G=G_1\times ...\times G_k$, $G_i$ are simple and $P=P_1\times...\times P_k$, $P_i$ are parabolic subgroups of $G_i$. Let $\pi_i:G/P\rightarrow G_i/P_i$ be the natural projections. In \S \ref{general case}, we prove the rigidity and multi-rigidity of Schubert classes for $G/P$ can be deduced from the multi-rigidity of Schubert classes for $G_i/P_i$:
\begin{Thm}
A Schubert class $\sigma_{w}$ for $G/P$ is rigid (resp. multi-rigid) if $(\pi_i)_*(\sigma_{w})$ are rigid (resp. multirigid) for all $1\leq i\leq k$.
\end{Thm}

Let $P_m$ be a maximal parabolic subgroup containing $P$. There is a natural morphism $$\pi:G/P\rightarrow G/P_m$$ which induces a correspondence between Schubert varieties. The next result shows that the rigidity of Schubert classes for homogeneous space $G/P$ with arbitrary parabolic subgroup $P$ can be deduced from the generalized Grassmannians (i.e. homogeneous spaces corresponding to maximal parabolic subgroups):
\begin{Thm}
Let $P$ be a parabolic subgroup with associated set of simple roots $I_P$. Let $w\in W^P$. The Schubert class $\sigma_w\in A^*(G/P)$ is rigid (resp. multi-rigid) if
\begin{enumerate}
\item $(\pi_\alpha)_*(\sigma_w):=\sigma_{w_\alpha}$ are rigid (resp. multi-rigid) for all $\alpha\in S:=\Phi\backslash I_P$, where $\pi_\alpha:G/P\rightarrow G/P_{\Phi\backslash\{\alpha\}}$ are natural projections; and
\item there exists an ordering on $S=\{\alpha_1,...\alpha_t\}$ such that 
$$E(w_{\alpha_i})\subset E(\pi_{\alpha_i*}\pi_{\alpha_{i+1}}^*(w_{\alpha_{i+1}})),\ \ 1\leq i\leq t-1$$
\end{enumerate}
\end{Thm}

\subsection*{Organization of the paper} In \S \ref{sec-prelim}, we review the basic facts about the partial flag varieties. In \S \ref{Type A}, we investigate the rigidity problem for partial flag varieties of type A. In particular, we characterize the multi-rigid Schubert classes in partial flag varieties of type A. In \S \ref{Type BD}, we investigate the rigidity problem in orthogonal Grassmannians as well as orthogonal partial flag varieties. In \S \ref{general case}, we study the rigidity problem in general homogeneous spaces.

\subsection*{Acknowledgments} We would like to thank Izzet Coskun, Yuri Tschinkel and others for valuable comments and corrections to improve our manuscript.  The second author is supported by a grant from Beijing Institute of Mathematical Sciences and Applications (BIMSA). The second author would also like to thank China's National Program of Overseas High Level Talent for generous support.

\section{Preliminaries}\label{sec-prelim}
In this section, we review some basic facts from Lie theory and provide the identification of partial flag varieties of classical type with the homogeneous spaces $G/P$. For a reference, see \cite{BL}.

\subsection{Root systems of algebraic groups}
Let $G$ be an algebraic group over an algebraically closed field $k$. The radical of $G$, denoted by $Rad(G)$, is the identity component of its maximal normal solvable subgroup. $G$ is called semi-simple if $Rad(G)$ is trivial. A torus in $G$ is an algebraic subgroup that isomorphic to $(\mathbb{C}^*)^n$. A maximal torus is one which is maximal among such subgroups. A Borel subgroup of $G$ is a maximal connected solvable algebraic subgroup. A parabolic subgroup $P$ of $G$ is a Zariski closed subgroup such that $G/P$ is a complete variety. A subgroup is parabolic if and only if it contains a Borel subgroup. 

Let $T$ be a maximal torus and $B$ be a Borel subgroup containing $T$. Let $\mathfrak{g}$ be the Lie algebra of $G$. The root system $R$ of $G$ relative to $T$ consists of all non-zero homomorphisms $\chi:T\rightarrow \mathbb{G}_m$, where $\mathbb{G}_m$ is the multiplication group of $k$, such that 
$$\mathfrak{g}_\chi:=\{v\in \mathfrak{g}|Ad(t)v=\chi(t)v,\forall t\in T\}\neq\emptyset.$$
The Weyl group is a subgroup generated by the reflections of elements in $R$. It is isomorphic to $N(T)/T$ where $N(T)$ is the normalizer of $T$ in $G$. The set of positive roots $R^+$ is a subset of $R$ that consists of all the elements such that 
$$\{v\in \mathfrak{b}|Ad(t)v=\chi(t)v,\forall t\in T\}\neq\emptyset,$$
where $\mathfrak{b}$ is the Lie algebra of $B$. The simple roots are the elements in $R^+$ that cannot be expressed as a positive sum of other elements in $R^+$. The set of simple roots is denoted by $\Phi$. The reflections corresponding to the simple roots are called simple reflections. The Weyl group $W$ is generated by the simple reflections. The length of an element $w\in W$ is defined to be the minimal length of an expression of $w$ as a product of simple reflections.

Let $P$ be a parabolic subgroup containing $B$. Let $Rad(P)$ be the radical of $P$. Let $Rad_u(P)$ be the subgroup of unipotent elements of $Rad(P)$. For each $\beta\in R$, there exists a unique connected $T$-stable subgroup $U_\beta$ of $G$ with $\mathfrak{g}_\beta$ as its Lie algebra. The set of positive roots $R^+_P$ associated to $P$ is defined by
$$R^+\backslash R^+_P=\{\beta\in R^+|U_\beta\subset Rad_u(P)\}.$$
The root system $R_P$ associated to $P$ is defined as the union $-R^+_P\cup R^+_P$. The set of simple roots $I_P$ associated to $P$ is the intersection $I_P=\Phi\cap R_P$. The Weyl group $W_P$ of $P$ is isomorphic to $N_P(T)/T$, where $N_P(T)$ is the normalizer of $T$ in $P$. In each coset of $W/W_P$, there exists a unique element of minimal length. Let $W^P$ be the set of minimal representatives of $W/W_P$. For every $w\in W^P$, the Schubert variety $\Sigma_{w}$ is defined as the Zariski closure of $BwP/P$ in $G/P$ with the canonical reduced scheme structure.

\subsection{Partial flag varieties (Type $A$)}
Let $V$ be a complex vector space of dimension $n$. The partial flag variety $F(d_1,...,d_k;n)$ parametrizes all $k$-step partial flags of subspaces of dimension $(d_1,...,d_k)$ in $V$. Fix a basis $(e_1,...,e_n)$ of $V$. Let $G=SL(V)\cong SL(n;\mathbb{C})$ be the special linear group of $V$. The partial flag variety $F(d_1,...,d_k;n)$ can be identified with the homogeneous space $G/P$ as follows:

Let $T$ be the maximal torus consisting of all the diagonal matrices in $G$. Let $B$ be the Borel subgroup consisting of all the upper triangular matrices $\begin{bmatrix}
    \lambda_1 & *  & \dots  & * \\
    0 & \lambda_2  & \dots  & * \\
    \vdots & \vdots  & \ddots & \vdots \\
    0 & 0 &  \dots  & \lambda_n
\end{bmatrix}$ in $G$. Let $P$ be the parabolic subgroup consisting of all the upper triangular block matrices of size $(d_1,d_2-d_1,...,d_k-d_{k-1},n-d_k)$:
$$P=\left\{\begin{bmatrix}
    A_{d_1\times d_1} &  * & \dots  & * \\
    0 & A_{(d_2-d_1)\times (d_2-d_2)}  & \dots  & * \\
    \vdots & \vdots  & \ddots & \vdots \\
    0 & 0 &  \dots  & A_{(n-d_k)\times (n-d_k)}
\end{bmatrix}\in G\right\}.$$
It is clear that $T\subset B\subset P$.

Let $W=N(T)/T$ be the Weyl group. It is easy to check that $N(T)$ consists of all the matrices that have exactly one non-zero entry in each row and in each column. Therefore $W$ is isomorphic to the symmetric group $S_n$ and we can choose the representatives in $W$ to be the permutation matrices. Let $W_P=N_P(T)/T$ be the Weyl group of $P$. Notice that $N_P(T)$ consists of all the block diagonal matrices of the size $(d_1,d_2-d_1,...,d_k-d_{k-1},n-d_k)$ in $N(T)$ and therefore $W_P\cong S_{d_1}\times...\times S_{n-d_k}$. A minimal representative of $W/W_P$ is an element $(w_1...w_n)\in S_n$ such that 
$$w_1<...<w_{d_1}, w_{d_1+1}<...<w_{d_2},...,w_{d_{k}+1}<...<w_n$$
There is a natural action of $G$ on $\mathbb{P}(\wedge^{d_i} \mathbb{C}^n)$ for each $1\leq i\leq k$, which induces an action of $G$ on $F(d_1,...,d_k;n)$. Let $E_d=<e_1,...,e_d>$ be the linear space spanned by $e_1,...,e_d$. Then $F(d_1,...,d_k;n)$ is the orbit of $(E_{d_1},...,E_{d_k})$ under $G$ with isotropy group $P$. The differential $T(G)_e\rightarrow T(F_{d_1,...,d_k;n})_e$ is surjective and therefore $F(d_1,...,d_k;n)$ can be identified with $G/P$. 

Let $w=(w_1...w_n)\in W^P$. Then the Schubert varieties $\Sigma_{w}$ can be identified with $\Sigma_{a^\alpha}(E_\bullet)$ defined as in \S\ref{introductionA}, where $a^\alpha$ is obtained by rearranging $w_1^1,...,w_{d_1}^1, w_{d_1+1}^2,...,w_{d_2}^2,...,w_{d_{k-1}+1}^k,...,w_{d_k}^k$ so that the lower indices are in the increasing order.

\subsection{Odd orthogonal partial flag varieties (Type $B$)}
Let $V$ be a complex vector space of odd dimension $n=2m+1$. Let $q$ be a non-degenerate symmetric bilinear form on $V$. The orthogonal partial flag variety $OF(d_1,...,d_k;n)$ parametrizes all $k$-step partial flags of isotropic subspaces of dimension $(d_1,...,d_k)$. 

Choose a basis $(e_1,...,e_n)$ of $V$ such that $q$ can be identified with an anti-diagonal matrix $Q$ with 1 along the anti-diagonal except at the middle the entry is 2. Let $G=SO(V)\cong SO(n;\mathbb{C})$ be the special orthogonal group of $V$. Note that $G$ is a subgroup of $SL(V)$ and an element $M\in SL(V)$ is contained in $G$ if and only if $M=Q^{-1}(M^T)^{-1}Q$.

Let $T^o$ be the maximal torus consisting of all the diagonal matrices in $SL(V)$. Let $B^o$ be the Borel subgroup consisting of all the upper triangular matrices in $SL(V)$. Let $P^o$ be the parabolic subgroup consisting of all the upper triangular block matrices in $SL(V)$ of size 
$$(d_1,d_2-d_1,...,d_k-d_{k-1},1+2(m-d_k),d_k-d_{k-1},...,d_1)$$

Let $T=T^o\cap G$, $B=B^o\cap G$, and $P=P^o\cap G$. It is easy to see that $T$ is a maximal torus in $G$, $B$ is a Borel subgroup in $G$ and $P$ is a parabolic subgroup in $G$. 

Let $M=\begin{bmatrix}
    \lambda_1 & *  & \dots  & * \\
    0 & \lambda_2  & \dots  & * \\
    \vdots & \vdots  & \ddots & \vdots \\
    0 & 0 &  \dots  & \lambda_n
\end{bmatrix}\in B$. The condition $M=Q^{-1}(M^T)^{-1}Q$ implies $\lambda_i=\lambda_{n+1-i}^{-1}$, $1\leq i\leq n$. In particular, $\lambda_m=1$.

Let $M'=\begin{bmatrix}
    A_{1} & *  & \dots  & * \\
    0 & A_{2}  & \dots  & * \\
    \vdots & \vdots  & \ddots & \vdots \\
    0 & 0 &  \dots  & A_{2k+1}
\end{bmatrix}\in P$, where $A_1$ and $A_{2k+1}$ are square matrices of size $d_1\times d_1$, $A_i$ and $A_{2k+2-i}$ are square matrices of size $(d_i-d_{i-1})\times (d_i-d_{i-1})$, $2\leq i\leq k$, and $A_{k+1}$ is a square matrix of size $[1+2(m-d_k)]\times [1+2(m-d_k)]$. The condition $M'=Q^{-1}(M'^T)^{-1}Q$ implies $A_{i}=A_{2k+2- i}^{-1}.$ In particular, $A_{k+1}^2$ is the identity matrix.

The maximal torus $T$ consists of all the diagonal matrices diag($t_1...t_n)$ such that $t_i=t_{n+1-i}^{-1}$, $1\leq i\leq n$. The normalizer $N(T)$ of $T$ consists of all the matrices that have exactly one non-zero entry in each row and in each column and where the entry at $(i,j)$-place is non-zero if and only if the entry at $(n+1-i,n+1-j)$-place is non-zero. Therefore the Weyl group $W=N(T)/T$ can be identified with
$$\{(w_1...w_n)\in S_n|w_i=n+1-w_{n+1-i}\}.$$
Notice that $(w_1...w_n)\in W$ is determined by the first $m$ entries $w_1,...,w_{m}$.

A minimal representative of $W/W_P$ is an element $(w_1...w_n)\in W$ whose first $m$ entries satisfy
\begin{enumerate}
\item $w_1<...<w_{d_1}, w_{d_1+1}<...<w_{d_2},...,w_{d_{k}+1}<...<w_m$;
\item $w_i\neq m+1$, $1\leq i\leq m$;
\item if $x\in\{w_1,...,w_m\}$, then $n+1-x\notin\{w_1,...,w_m\}$.
\end{enumerate}

Let $E_i=<e_1,...,e_i>$ for $1\leq i\leq m$. Then $E_i$ are isotropic subspaces with respect to $q$. The orthogonal partial flag variety $OF(d_1,...,d_k;n)$ is the orbit of $(E_{d_1},...,E_{d_k})$ under $G$ with isotropy group $P$. 

Let $w=(w_1...w_n)\in W^P$. The Schubert variety $\Sigma_{w}$ can be identified with $\Sigma_{a^\alpha;b^\beta}(E_\bullet)$ in $OF(d_1,...,d_k;n)$, where $a^\alpha$ is obtained from $w_1^1,...,w_{d_1}^1, w_{d_1+1}^2,...,w_{d_2}^2,...,w_{d_{k-1}+1}^k,...,w_{d_k}^k$ by first removing all the elements whose lower index is greater than $m$ and then rearranged so that the lower indices are in the increasing order, and $b^\beta$ is obtained from $(n-w_1)^1,...,(n-w_{d_1})^1, (n-w_{d_1+1})^2,...,(n-w_{d_2})^2,...,(n-w_{d_{k-1}+1})^k,...,(n-w_{d_k})^k$ by removing all the elements whose lower index is greater than $m-1$ and then rearranged so that the lower indices are in the increasing order.

\subsection{Even orthogonal partial flag varieties (Type $D$)}Let $V$ be a complex vector space of even dimension $n=2m$. Let $q$ be a non-degenerate symmetric bilinear form on $V$. The orthogonal partial flag variety $OF(d_1,...,d_k;n)$ parametrizes all $k$-step partial flags of isotropic subspaces of dimension $d_1,...,d_k$, unless $d_k=m$, in which case the parameter space has two irreducible components and we let $OF(d_1,...,d_k;n)$ denote one of the components. 

When $d_k=m-1$, for every isotropic subspace $F_{m-1}$ of dimension $m-1$, there exists a unique maximal isotropic subspace (contained in a fixed irreducible component of the space of maximal isotropic subspaces) containing $F_{m-1}$. Therefore $$OF(d_1,...,d_{k-1},m-1;n)\cong OF(d_1,...,d_{k-1},m-1,m;n).$$
From now on, we assume $d_k\neq m-1$.

Choose a basis $(e_1,...,e_n)$ of $V$ such that $q$ can be identified with an anti-diagonal matrix $Q$ with 1 along the anti-diagonal. Let $G=SO(V)\cong SO(n;\mathbb{C})$ be the special orthogonal group of $V$. Note that $G$ is a subgroup of $SL(V)$ and an element $M\in SL(V)$ is contained in $G$ if and only if $M=Q^{-1}(M^T)^{-1}Q$.

Let $T^o$ be the maximal torus consisting of all the diagonal matrices in $SL(V)$. Let $B^o$ be the Borel subgroup consisting of all the upper triangular matrices in $SL(V)$. Let $P^o$ be the parabolic subgroup consisting of all the upper triangular block matrices in $SL(V)$ of size $$(d_1,d_2-d_1,...,d_k-d_{k-1},n-2d_k,d_k-d_{k-1},...,d_2-d_1,d_1).$$
Let $T=T^o\cap G$, $B=B^o\cap G$, and $P=P^o\cap G$. Then $T$ is a maximal torus in $G$, $B$ is a Borel subgroup in $G$ and $P$ is a parabolic subgroup in $G$. 

Let $M'=\begin{bmatrix}
    A_{1} & *  & \dots  & * \\
    0 & A_{2}  & \dots  & * \\
    \vdots & \vdots  & \ddots & \vdots \\
    0 & 0 &  \dots  & A_{2k+1}
\end{bmatrix}\in P$, where $A_1$ and $A_{2k+1}$ are square matrices of size $d_1\times d_1$, $A_i$ and $A_{2k+2-i}$ are square matrices of size $(d_i-d_{i-1})\times (d_i-d_{i-1})$, $2\leq i\leq k$, and $A_{k+1}$ is a square matrix of size $[n-2d_k]\times [n-2d_k]$. (If $n=2d_k$, then we omit $A_{k+1}$.) The condition $M'=Q^{-1}(M'^T)^{-1}Q$ implies $A_{i}=A_{2k+2- i}^{-1}.$

One can also easily check that $T$ consists of all the diagonal matrices diag($t_1...t_n)$ such that $t_i=t_{n+1-i}^{-1}$, $1\leq i\leq n$. The Weyl group $W$ of $G$ can be identified with
$$\{(w_1...w_n)\in S_n|w_i=n+1-w_{n+1-i},\#\{i|w_i>m,1\leq i\leq m\}\text{ is even}\}.$$
Notice that $(w_1...w_n)\in W$ is determined by the first $m$ entries $w_1,...,w_{m}$.

A minimal representative of $W/W_P$ is an element $(w_1...w_n)\in W$ whose first $m$ entries satisfy
\begin{enumerate}
\item $w_1<...<w_{d_1}, w_{d_1+1}<...<w_{d_2},...,w_{d_{k}+1}<...<w_m$;
\item if $x\in\{w_1,...,w_m\}$, then $n+1-x\notin\{w_1,...,w_m\}$.
\end{enumerate}

Let $E_i=<e_1,...,e_i>$ for $1\leq i\leq m$. Let $w=(w_1...w_n)\in W^P$. The Schubert variety $\Sigma_{w}$ can be identified with $\Sigma_{a^\alpha;b^\beta}(E_\bullet)$ in $OF(d_1,...,d_k;n)$, where $a^\alpha$ is obtained from the sequence $$w_1^1,...,w_{d_1}^1, w_{d_1+1}^2,...,w_{d_2}^2,...,w_{d_{k-1}+1}^k,...,w_{d_k}^k$$ by first removing all the elements whose lower index is greater than $m+1$ and then rearranged so that the lower indices are in the increasing order, and $b^\beta$ is obtained from the sequence 
$$(n-w_1)^1,...,(n-w_{d_1})^1, (n-w_{d_1+1})^2,...,(n-w_{d_2})^2,...,(n-w_{d_{k-1}+1})^k,...,(n-w_{d_k})^k$$ by removing all the elements whose lower index is greater than $m$ and then rearranged so that the lower indices are in the increasing order.

\section{Partial flag varieties of type A}\label{Type A}
In this section, we study the multi-rigidity problem in partial flag varieties that generalize Grassmannian varieties. In particular, we classify the multi-rigid Schubert classes in partial flag varieties.

\begin{Def}
A Schubert class is called {\em multi-rigid} if every multiple of it can only be represented by a union of Schubert varieties.
\end{Def}
Equivalently, a Schubert class $\sigma$ is multi-rigid if and only if $\sigma$ is rigid and $m\sigma$ cannot be represented by irreducible subvarieties for any $m\geq 2$. 

We first recall the rigidity results from Grassmannians. 

\begin{Thm}\cite{Ho2,RT}
A Schubert class $\sigma_{a}$ in $G(k,n)$ is multi-rigid if and only if for all essential $a_i<n$ (i.e. $a_{i+1}\neq a_i+1$),
$$a_{i-1}+1= a_i\leq a_{i+1}-3$$
(Here we set $a_0=0$, and $a_{k+1}=\infty$.)
\end{Thm}

To generalize this result, it would be helpful to consider the multi-rigidity of each flag element. To be more precise,
\begin{Def}
Let $\sigma_a$ be a Schubert class in $G(k,n)$. An essential sub-index $a_i$ is called {\em multi-rigid} if for any irreducible representative $X$ of $m\sigma_a$, $m\in\mathbb{Z}^+$, there exists a linear subspace $F_{a_i}$ such that $$\dim(\Lambda\cap F_{a_i})\geq i, \forall\Lambda\in X$$
\end{Def}
In \cite{YL2}, we characterized the multi-rigid sub-indices in Grassmannians:
\begin{Thm}\cite[Theorem 4.1.5]{YL2}\label{rigid in g}
Let $\sigma_a$ be a Schubert class in $G(k,n)$. Set $a_0=0$ and $a_{n+1}=\infty$. An essential sub-index $a_i$ is multi-rigid if and only if 
\begin{enumerate}
\item either $a_{i-1}+1=a_i\leq a_{i+1}-3$; or
\item $a_i=n$.
\end{enumerate}
\end{Thm}

Now we move to the case of partial flag varieties. 
\begin{Def}
Let $\sigma_{a^\alpha}$ be a Schubert class in $F(d_1,...,d_k;n)$. A sub-index $a_i$ is called {\em essential} if either $a_{i+1}\neq a_i+1$ or $\alpha_{i+1}>\alpha_{i}$. An essential $a_i$ is called {\em multi-rigid} if for every irreducible representative $X$ of $m\sigma_{a^\alpha}$, $m\in\mathbb{Z}^+$, there exists a linear subspace $F_{a_i}$ such that 
$$\dim(F_{a_i}\cap \Lambda_j)\geq \mu_{i,j}, j=1,...,k, \ \forall (\Lambda_1,...,\Lambda_k)\in X.$$
\end{Def}

We start with an example:

\begin{Ex}
Consider the partial flag variety $F(2,3;4)$. It admits two natural projections $\pi_1:F(2,3;4)\rightarrow G(2,4)$ and $\pi_2:F(2,3;4)\rightarrow G(3,4)$. Consider the Schubert class $\sigma_{1^1,2^1,4^2}$. It is taken to the Schubert class $\sigma_{1,2}$ in $G(2,4)$ under the first projection $\pi_1$, and is taken to the Schubert class $\sigma_{1,2,4}$ in $G(3,4)$ under the second projection $\pi_2$. By Theorem \ref{rigid in g}, the Schubert class $\sigma_{1,2}$ is multi-rigid. We claim that the sub-index $2$ is also multi-rigid for the Schubert class $\sigma_{1^1,2^1,4^2}$.

Let $X$ be an irreducible subvariety with class $m\sigma_{1^1,2^1,4^2}$, $m\in\mathbb{Z}$. The image $\pi_1(X)$ is also irreducible and has class $m'\sigma_{1,2}$ in $G(2,4)$ for some $m'|m$. By Theorem \ref{rigid in g}, $m'=1$ and $\pi_1(X)$ consists of a unique subspace $F_2$ of dimension 2. It is also obvious that $\dim(\Lambda_2\cap F_2)=\dim(\Lambda_1\cap F_2)=2$ for all $(\Lambda_1,\Lambda_2)\in X$. Therefore $2$ is multi-rigid.
\end{Ex}

From this example, one may expect a sub-index to be multi-rigid if it is multi-rigid in some components.

Recall that the push-forward of a Schubert class under the canonical projection $$\pi_t:F(d_1,...,d_k;n)\rightarrow G(d_t,n)$$ is well-defined and is obtained by removing all the sub-indices whose upper-index is greater than $t$:

\begin{Prop}\label{class of pushforward}
Let $X$ be a representative of $m\sigma_{a^\alpha}$, $m\in\mathbb{Z}^+$ and set $X_t:=\pi_t(X)$. Then $X_t$ has class $m'(\pi_t)_*(\sigma_{a^\alpha})$ in $G(d_t,n)$ for some $m'\in\mathbb{Z}^+$.
\end{Prop}

\begin{proof}
Use the method of undetermined coefficients. The Chow ring of $G(d_t,n)$ is freely generated by the Schubert classes and therefore we may write 
$$[X_t]=\sum_\gamma \alpha_\gamma\sigma_\gamma,$$
where $\alpha_\gamma\in\mathbb{Z}$ and $\sigma_{\gamma}$ are Schubert classes in $G(d_t,n)$. To find the coefficients $\alpha_\gamma$, we intersect $X_t$ with the Schubert varieties of the complementary dimension, which corresponds to intersecting $X$ with the pull-back of those Schubert varieties in $F(d_1,...,d_k,n)$. On the level of intersection product, we can presume $X$ to be a Schubert variety and it turns out that the intersection product is $0$ expect the dual Schubert class of $(\pi_t)_*(\sigma_{a^\alpha})$. Therefore $[X_t]=m'\gamma(\pi_t)_*(\sigma_{a^\alpha})$ for some $m'\in\mathbb{Z}^+$.
\end{proof}

Also recall that the class of a general fiber of a canonical projection can be computed as follows:
\begin{Prop}\cite[Corollary 3.11]{YL3}
Let $X$ be a representative of $\sigma_{a^\alpha}$. A general fiber of $\pi_i|_X$ has class $\sigma_{b^\beta}$ where $b_j=j$ and $\beta_j=\alpha_j$ for $1\leq j\leq d_i$; $b_{d_i+t}=a_t''+\#\{s|a_s>\alpha_t',\alpha_s\leq i\}$ and $\beta_{d_i+t}=\alpha_t'$ for $d_i+1\leq d_i+t\leq d_k$.
\end{Prop}
The same statement is true for a multiple of a Schubert class, if we allow the coefficient to be different:
\begin{Cor}\label{class of fiber}
Let $X$ be a representative of $m\sigma_{a^\alpha}$, $m\in\mathbb{Z}^+$. Then a general fiber of $\pi_i|_X$ has class $m'\sigma_{b^\beta}$ where $m'\in\mathbb{Z}^+$, $b_j=j$ and $\beta_j=\alpha_j$ for $1\leq j\leq d_i$; $b_{d_i+t}=a_t''+\#\{s|a_s>\alpha_t',\alpha_s\leq i\}$ and $\beta_{d_i+t}=\alpha_t'$ for $d_i+1\leq d_i+t\leq d_k$.
\end{Cor}

Now we are able to show that a sub-index is multi-rigid if it is multi-rigid in some components:

\begin{Prop}\label{rigid of flag a}
Let $\sigma_{a^\alpha}$ be a Schubert class for $F(d_1,...,d_k;n)$. For every $1\leq \gamma\leq n$, set $a^{-1}(\gamma)=i$ if $a_i=\gamma$ and $a^{-1}(\gamma)=d_k+1$ if $a_i\neq \gamma$ for all $1\leq i\leq d_k$. An essential sub-index $a_i$ is multi-rigid if $a_{i-1}=a_i-1$ and $\max\{\alpha_{i-1},\alpha_i\}<\min\{\alpha_{a^{-1}(a_i+1)},\alpha_{a^{-1}(a_i+2)}\}$.
\end{Prop}
\begin{proof}
Let $X$ be an irreducible subvariety with class $m\sigma_{a^\alpha}$, $m\in\mathbb{Z}^+$. Set $t=\max\{\alpha_{i-1},\alpha_i\}$. By assumption and Theorem \ref{rigid in g}, the sub-index $a_i$ is multi-rigid with respect to the Schubert class $(\pi_t)_*(\sigma_{a^\alpha})$. Since $X_t:=\pi_t(X)$ is also irreducible, there exists a linear subspace $F_{a_i}$ such that for every $\Lambda_t\in X_t$, $\dim(F_{a_i}\cap \Lambda_t)\geq \mu_{i,t}$. We claim that $\dim(F_{a_i}\cap \Lambda_r)\geq \mu_{i,r}$ for all $1\leq r\leq k$.

Suppose for a contradiction that $\dim(F_{a_i}\cap \Lambda_r)<\mu_{i,r}$ for some $r$ and $\Lambda_r\in \pi_r(X)$. By the semi-continuity of dimension, this inequality holds for a general $\Lambda_r\in \pi_r(X)$. Consider the fiber of $\pi_r$ over a general point $\Lambda_r\in \pi_r(X)$. By Corollary \ref{class of fiber}, $(\pi_r|_X)^{-1}(\Lambda_r)$ has class $m_r\sigma_{b^\beta}$ where $m_r\in\mathbb{Z}^+$, $b_{\mu_{i,r}}=\mu_{i,r}$ and $\beta_{\mu_{i,r}}=\alpha_i$. 

{\bf Case I}. If $t<r$, then $({\pi_t}\circ \pi_r|_X^{-1})(\Lambda_r)$ has class $m_c\sigma_{c}$ where $m_c\in\mathbb{Z}^+$ and $c_{\mu_{i,t}}=\mu_{i,r}$. On the other hand, set $W=F_{a_i}\cap\Lambda_r$. By assumption, $w:=\dim(W)<\mu_{i,r}$. Since $\dim(F_{a_i}\cap\Lambda_t)\geq \mu_{i,t}$, we get for all $\Lambda'_t\in \pi_t(\pi_r|_X^{-1}(\Lambda_r))$, 
$$\dim(W\cap\Lambda'_t)= \dim(F_{a_i}\cap\Lambda'_t)\geq \mu_{i,t}.$$
Since $\dim(W)<\mu_{i,r}$, it contradicts the condition that $c_{\mu_{i,t}}=\mu_{i,r}$.

{\bf Case II}. If $r<t$, by the previous case we may assume $t=k$ and $a_i$ is rigid in the $k$-th component. If there is no $l$ such that $l>i$ and $\alpha_l>t$, then $d_k-d_r=\mu_{i,k}-\mu_{i,r}$. Since $\Lambda_r$ is of codimension $d_k-d_r$ in $\Lambda_k$, we must have
$$\dim(F_{a_i}\cap \Lambda_r)\geq \dim(F_{a_i}\cap\Lambda_k)-(d_k-d_r)\geq \mu_{i,k}-(\mu_{i,k}-\mu_{i,r})=\mu_{i,r}.$$
If there exists $l$ such that $l>i$ and $\alpha_l>r$, assume $l$ is minimum among such $l$. Set $\gamma=\#\{s|s\leq l,\alpha_s>r\}$. Then $({\pi_k}\circ \pi_i|_X^{-1})(\Lambda_r)$ has class $m_c\sigma_{c}$ where $m_c\in\mathbb{Z}^+$ and $$c_{d_r+\gamma}=d_r+\gamma+a_l-l.$$
On the other hand, let $U=$span$(\Lambda_r,F_{a_i})$. Then $\dim(U)=d_r+a_i-w$, and 
\begin{eqnarray}
\dim(U\cap \Lambda_k)=d_r+i-w&>&d_i+i-\mu_{i,r}\nonumber\\
&=&d_r+\gamma-1.\nonumber
\end{eqnarray}
Let $U'$ be a general subspace of codimension $i-w-\gamma$ in $U$. Then
$$\dim(U'\cap\Lambda_k)\geq d_r+i-w-(i-w-\gamma)=d_t+\gamma.$$
Since $a_i$ is essential, $a_l-l>a_i-i$, and therefore 
$$\dim(U')=d_r+a_i-w-(i-w-\gamma)<d_r+\gamma+a_l-l.$$
This contradicts the relation $c_{d_r+\gamma}=d_r+\gamma+a_l-l.$ We conclude that $a_i$ is multi-rigid.
\end{proof}

Even if the conditions in the previous proposition fail, the sub-index may still be multi-rigid. This can be seen from the following example.
\begin{Ex}
Consider the partial flag variety $F(1,3;4)$. Let $\pi_1:F(1,3;4)\rightarrow G(1,4)$ and $\pi_2:F(1,3;4)\rightarrow G(3,4)$ be the two projections. Consider the Schubert class $\sigma_{1^2,2^1,4^2}$. It is taken to $\sigma_{2}$ through the first projection and is taken to $\sigma_{1,2,4}$ through the second projection. In either case, by Theorem \ref{rigid in g}, the sub-index $2$ is not multi-rigid. Nevertheless, we claim that the sub-index 2 is multi-rigid with respect to the Schubert class $\sigma_{1^2,2^1,4^2}$.

Let $X$ be an irreducible representative of $m\sigma_{1^2,2^1,4^2}$, $m\in\mathbb{Z}^+$. Then $\pi_1(X)$ has class $m'\sigma_2$ in $G(1,4)$, for some $m'\in\mathbb{Z}^+$. Therefore $\pi_1(X)$ is an irreducible curve $C$ in $\mathbb{P}^3$. Let $W$ be the projective linear space spanned by $C$. For a general $\Lambda_2\in\pi_2(X)$, there are two possibilities: either $\mathbb{P}(\Lambda_2)$ meets $C$ in finitely many points, or $\mathbb{P}(\Lambda_2)$ contains $W$. The first possibility cannot happen since by Corollary \ref{class of fiber}, the fiber over $\Lambda_2$ has dimension $1$ instead of $0$. The second possibility can only happen when $\dim(W)=1$, i.e. $C$ is a projective line. It is because if $\dim(W)\geq 2$ and $\Lambda$ contains $W$, then $\pi_2(X)$ should have dimension $0$ which is not case.
\end{Ex}

\begin{Ex}
Consider the Schubert classes $\sigma_{1^1,4^1,5^2}$ in the partial flag variety $F(2,3;5)$. It is taken to the Schubert class $\sigma_{1,4,5}$ under the second projection, which is multi-rigid in the Grassmannian $G(3,5)$. Nevertheless we claim that the Schubert class $\sigma_{1^1,4^1,5^2}$ is not multi-rigid in $F(2,3;5)$ since the sub-index $4$ is not multi-rigid in the first component and is not essential in the second component.

Take a partial flag of subspaces in $V$:
$$F_1\subset F_{2}\subset F_5=V,$$
where $\dim(F_{a_i})=a_i$. Let $C$ be a smooth space curve in $\mathbb{P}(V)$ of degree $d$ whose span is disjoint from $F_2$. Let $Z$ be the cone over $C$ with vertex $F_2$. Let $X$ be the subvariety defined by
$$X:=\{(\Lambda_1,\Lambda_2)\in F(2,3;5)|F_1\subset \Lambda_1\subset \Lambda_2,\mathbb{P}(\Lambda_1)\subset Z\}.$$
It is then easy to check $X$ is irreducible with class $d\sigma_{1^1,4^1,5^2}$
\end{Ex}

To obtain a complete classification, we need the following lemma:

\begin{Lem}\label{degree}
Let $\sigma_a$ be a Schubert class in $G(k,n)$. Let $X$ be an irreducible representative of $m\sigma_a$, $m\in\mathbb{Z}^+$. If $a_{i}\leq a_{i+1}-3$, then there exists an irreducible projective variety $Y$ of dimension $a_i-1$ such that for every $\Lambda\in X$, $\dim(\mathbb{P}(\Lambda)\cap Y)\geq i-1$.
\end{Lem}
\begin{proof}
We use induction on $a_i-i$ on an increasing order. If $a_i=i$, then by Theorem \ref{rigid in g}, $a_i$ is multi-rigid and the statement follows.

If $a_i-i\geq 1$, then consider the incidence correspondence
$$I:=\{(\Lambda,H)|\Lambda\subset H,\Lambda\in X, H\in G(n-1,n)\}.$$
For a general hyperplane $H\in\pi_2(I)$, the locus $X_H:=\pi_1(\pi_2^{-1}(H))$ has class $m'\sigma_{a'}$ in $G(k,n)$, where $m'\in\mathbb{Z}$ and $a'_j=a_j$ if $a_j=j$ and $a'_j=a_j-1$ if $a_j>j$. By induction, there exists an irreducible projective variety $Y_H$ of dimension $a_i-2$ such that $\dim(\mathbb{P}(\Lambda)\cap Y_H)\geq i-1$ for all $\Lambda\in X_H$. As we vary $H$, let $Y$ be the projective variety swept out by $Y_H$. Clearly $\dim(Z)\geq a_i-2$. For a general $H'$ that does not contain $Y_H$, $Y_{H'}\neq  Y_H$, and hence $\dim(Z)\geq a_i-1$. We claim that $\dim(Y)=a_i-1$. 

Suppose, for a contradiction, that $\dim{Y}=a\geq a_i$. Let $G_\bullet$ be a general complete flag. Then $\mathbb{P}(G_{n-a_i})$ will meet $Y$ in finitely many points. By the construction of $Y$, there exists a hyperplane $H$ such that $$\mathbb{P}(G_{n-a_i})\cap Y_H\neq \emptyset.$$ Since $G_{n+2-a_i}$ is a general linear space of dimension $n+2-a_i$ containing $G_{n-a_i}$, we may assume
$$\mathbb{P}(G_{n-a_i})\cap Y_H=\mathbb{P}(G_{n+2-a_i})\cap Y_H.$$ 
Let $\sigma_b$ be the dual Schubert class of $\sigma_{a'}$ and let $\Sigma_b(G_\bullet)$ be the corresponding Schubert variety. Let $\Lambda\in X_H\cap\Sigma_b(G_\bullet)$. Then $\dim(\mathbb{P}(\Lambda)\cap Y_H)=i-1$, $\dim(\Lambda\cap G_{n+2-a_{i+1}})=k-i$ and $\dim(\Lambda\cap G_{n+2-a_{i}})=k-i+1$, and therefore $$\mathbb{P}(\Lambda)\cap \mathbb{P}(G_{n+2-a_i})\cap Y_H\neq\emptyset.$$
Since $a_{i+1}\geq a_i+3$, $n-a_i>n-a_{i+1}+2$ and thus $\mathbb{P}(G_{n+2-a_{i+1}})$ does not meet $Y_H$, 
$$\dim(\Lambda\cap G_{n-a_i})\geq \dim(\Lambda\cap Y_H\cap G_{n+2-a_i})+\dim(\Lambda\cap G_{n+2-a_{i+1}})=k-i+1.$$
Hence we proved that for a general $(n-a_i)$-dimensional vector space $G_{n-a_i}$, there exists a $\Lambda\in X$ that meets $G_{n-a_i}$ in a $(k-i+1)$-dimensional subspace, which is a contradiction since
$$\sigma_a\cdot \sigma_c=0,$$
where $c_j=n-a_i+j-(k-i+1)$ for $1\leq j\leq k-i+1$ and $c_j=n+j-k$ for $k-i+2\leq j\leq k$. We conclude that $\dim(Y)=a_i-1$ as desired.

\end{proof}

\begin{Prop}\label{rigid 2}
Let $\sigma_{a^\alpha}$ be a Schubert class in $F(d_1,...,d_k;n)$. If $a_i=a_{i-1}+1=a_{i+1}-2$ and $\alpha_{i}<\alpha_{i+1}\leq \alpha_{i-1}$, then $a_i$ is multi-rigid.
\end{Prop}
\begin{proof}
Let $X$ be an irreducible representative of $m\sigma_{a^\alpha}$, $m\in\mathbb{Z}$. Let $\pi_t:F(d_1,...,d_k;n)\rightarrow G(d_t,n)$ be the natural projections. Set $X_t:=\pi_t(X)$. By Proposition \ref{class of pushforward}, $X_{\alpha_i}$ has class $m'\sigma_{a'}$ where $m'\in\mathbb{Z}^+$, $a'_{\mu_{i,\alpha_i}}=a_i$ and $a'_{\mu_{i,\alpha_i}+1}\geq a_i+3$. Apply Lemma \ref{degree} to $X_{\alpha_i}$, there exists a variety $Y$ of dimension $a_i-1$ such that $\dim(Y\cap\mathbb{P}(\Lambda_{\alpha_i}))\geq i-1$ for all $\Lambda_{\alpha_i}\in X_{\alpha_i}$. We claim that $Y$ is linear.

Let $p$ be a general point in $Y$ and consider the locus $X_{\alpha_{i-1},p}:=\{\Lambda_{\alpha_{i-1}}\in X_{\alpha_{i-1}}|p\in\mathbb{P}(\Lambda)\}$. In $G(d_{\alpha_{i-1}},n)$, $X_{\alpha_{i-1},p}$ has class $m_b\sigma_{b}$ where $b_{\mu_{i,\alpha_{i-1}}}=a_i$. Let $T_pY$ be the tangent space of $Y$ at $p$, and let $Y'=Y\cap T_pY$. If $Y$ is not linear, then $\dim(Y')<\dim(Y)=a_i-1$. However, for every $\Lambda_{\alpha_{i-1}}\in X_{\alpha_{i-1},p}$, $\Lambda_{\alpha_{i-1}}$ is contained in the tangent space $T_pY$ and therefore $$\dim(\mathbb{P}(\Lambda_{\alpha_{i-1}})\cap Y')=\dim(\mathbb{P}(\Lambda_{\alpha_{i-1}})\cap Y)\geq \mu_{i,\alpha_{i-1}}-1.$$
This contradicts the condition $b_{\mu_{i,\alpha_{i-1}}}=a_i$. We conclude that $Y$ is linear.

By a similar argument as in Proposition \ref{rigid of flag a}, we obtain $\dim(Y\cap\mathbb{P}(\Lambda_j))\geq \mu_{i,j}-1$ for all $(\Lambda_1,...,\Lambda_k)\in X$, $\alpha_i\leq j\leq k$. We conclude that $a_i$ is multi-rigid.
\end{proof}
We hereby obtain the following classification of multi-rigid sub-indices:
\begin{Thm}\label{rigid3}
Let $\sigma_{a^\alpha}$ be a Schubert class of $F(d_1,...,d_k;n)$. An essential sub-index $a_i$ is multi-rigid if and only if $a_i-a_{i-1}=1$ and one of the following conditions hold:
\begin{itemize}
\item $a_{i+1}-a_i\geq 3$;
\item $a_{i+1}-a_i=2$ and $\alpha_i<\alpha_{i+1}$;
\item $a_{i+1}-a_i=1$ and $\alpha_{i-1}<\alpha_{i+1}$ and either
\begin{enumerate}
\item $a_{i+2}-a_i\geq 3$; or
\item $\alpha_i<\alpha_{i+2}<\alpha_{i+1}$.
\end{enumerate}
\end{itemize}
\end{Thm}
\begin{proof}
\underline{Sufficiency} Follows from Proposition \ref{rigid of flag a} and Proposition \ref{rigid 2}.

\underline{Necessity}
We construct counter-examples when the conditions fail. 

First assume $a_{i-1}\neq a_i-1$. Pick a partial flag 
$$F_1\subset...\subset F_{a_i-2}\subset F_{a_i+1}\subset...\subset F_n=V,$$
where $\dim(F_j)=j$. This partial flag differ from a complete flag in that there is no flag element of dimension $a_i-1$ and $a_i$. Take a degree $m$ smooth plane curve $C$ in $\mathbb{P}(F_{a_i+1})$, whose span is disjoint from $\mathbb{P}(F_{a_i-2})$. Let $Z$ be the cone over $C$ with vertex $\mathbb{P}(F_{a_i-2})$. Let $X$ be the following locus
$$X:=\{(\Lambda_1,...,\Lambda_k)\in F(d_1,...,d_k;n)|\dim(F_{a_i'}\cap \Lambda_j)\geq \mu_{i',j}, \text{ for }i'\neq i,\dim(Z\cap\mathbb{P}(\Lambda_j))\geq \mu_{i,j}-1\}.$$
Then $X$ is irreducible with class $m\sigma_{a^\alpha}$.

Now assume $a_{i-1}=a_i-1$. If $a_{i+1}-a_i=2$ and $\alpha_i\geq\alpha_{i+1}$, then under the duality $F(d_1,...,d_k;n)\cong F(n-d_k,...,n-d_1;n)$, the Schubert class $\sigma_{a^\alpha}$ is taken to the Schubert class $\sigma_{b^\beta}$, where $b_j-b_{j-1}\neq 1$ for some $j$. Therefore this reduces to the previous case.

If $a_{i+1}-a_i=1$ and $\alpha_{i-1}\geq \alpha_{i+1}$, then the same construction as in the first case would work.
\end{proof}

As a corollary, we classify the multi-rigid Schubert classes in partial flag varieties. One can expect a Schubert class is multi-rigid if all essential sub-indices are multi-rigid and the corresponding linear subspaces form a partial flag. Recall that the following condition ensures the compatibility:
\begin{Lem}\cite{YL3}\label{coincide}
Let $\sigma_{a}$ be a Schubert class in $G(k,n)$. Let $X$ be a representative of $\sigma_a$. Assume there are two linear subspaces $F_{a_i}$ and $F_{a_j}$ of dimension $a_i$ and $a_j$ respectively, such that $\dim(\Lambda\cap F_{a_i})\geq i$, $\dim(\Lambda\cap F_{a_j})\geq j$ for all $\Lambda\in X$. If $i<j$ and $a_j$ is essential, then $F_{a_i}\subset F_{a_j}$.
\end{Lem}
This leads to the following definition.
\begin{Def}\label{link}
Let $\sigma_{a^\alpha}$ be a Schubert class in $F(d_1,...,d_k;n)$. We define a relation `$\rightarrow$' between two sub-indices: $a_i\rightarrow a_j$ if $i<j$ and $a_j$ is essential in $(\pi_t)_*(\sigma_{a^\alpha})$ for some $t\geq \min(\alpha_i,\alpha_j)$. This relation extends to a strict partial order (which we also denote by `$\rightarrow$') on the set $A=\{a_i\}_{i=1}^{d_k}$ by transitivity. If $a_i\rightarrow a_j$, then we say $a_i$ is linked to $a_j$.
\end{Def}
\begin{Cor}\label{rigid in f}
A Schubert class $\sigma_{a^\alpha}$ in the partial flag variety $F(d_1,...,d_k;n))$ is multi-rigid if and only if all essential sub-indices are multi-rigid and the set of all essential sub-indices is strict totally ordered under the relation `$\rightarrow$' defined in Definition \ref{link}.
\end{Cor}
\begin{proof}
The sufficiency follows from Theorem \ref{rigid3} and Lemma \ref{coincide}. 

For necessity, if one of the essential sub-indices is not multi-rigid, then the construction in the proof of Theorem \ref{rigid3} gives a counter-example. If the set of all essential sub-indices is not strict totally ordered, then the Schubert class $\sigma_{a^\alpha}$ is not even rigid. If there are two essential sub-indices $a_i$ and $a_j$, $i<j$ such that $a_i$ is not linked to $a_j$, assume there is no other essential $a_{r}$ between $i$ and $j$. Take a partial flag
$$F_{a_1}^{\alpha_1}\subset...\subset F_{a_{d_k}}^{\alpha_{d_k}}.$$
Replace $F_{a_j}$ with another linear subspace $G_{a_j}$ of the same dimension such that $$F_{a_{i-1}}\subset G_{a_j}\subset F_{a_{j+1}},$$
and $\dim(F_{a_{i}}\cap G_{a_j})=a_{i}-1$.

Define 
$$X:=\{\Lambda\in F(d_1,...,d_k)|\dim(G_{a_j}\cap \Lambda_t)\geq \mu_{j,t},\dim(F_{a_s}\cap \Lambda_t)\geq \mu_{s,t}, \text{ for }s\notin (i,j]\}.$$
Then $X$ represents the Schubert class $\sigma_{a^\alpha}$ but is not a Schubert variety since it is not defined by a partial flag.
\end{proof}

\section{Orthogonal partial flag varieties}\label{Type BD}
In this section, we study the multi-rigidity problem in orthogonal partial flag varieties. In particular, we characterize the multi-rigid Schubert classes in orthogonal Grassmannians and orthogonal flag varieties.

\subsection{Compatibility between different components}In this sub-section, we discuss the compatibility of essential sub-indices between different components. In particular, we prove that an essential sub-index is multi-rigid if it is multi-rigid in some component. We start with the definition of essential sub-indices:
\begin{Def}\label{def1}
Let $\sigma_{a;b}$ be a Schubert class for $OG(k,n)$. A sub-index $a_i$ is called {\em essential} if one of the following conditions holds:
\begin{itemize}
\item $i<s$ and $a_{i}<a_{i+1}-1$; or
\item $n$ is odd and $i=s$; or
\item $n$ is even, $i=s$ and $a_s+b_{k-s}\neq n-2$.
\end{itemize}

A sub-index $b_j$ is called {\em essential} if either $j=1$ or $b_{j}\neq b_{j-1}+1$. 
\end{Def}
\begin{Def}\label{defofrigid}
Let $\sigma_{a^\alpha,b^\beta}$ be a Schubert index for $OF(d_1,...,d_k;n)$. A sub-index $a_i$ or $b_j$ is called {\em essential} if it is essential with respect to the push-forward class $(\pi_t)_*(\sigma_{a^\alpha;b^\beta})$ for some $1\leq t\leq k$.
\end{Def}
Let $\pi_t:OF(d_1,...,d_k;n)\rightarrow OG(d_t,n)$ be the natural projections, $1\leq t\leq k$. By specializing to Schubert varieties, one obtain the following lemmas for calculating the pushforward classes and pull-back classes under the natural projections $\pi_t$: 

\begin{Lem}\label{class of pushforward in of}
Let $X$ be a subvariety in $OF(d_1,...,d_k;n)$ representing the Schubert class $m\sigma_{a^\alpha;b^\beta}$, $m\in\mathbb{Z}^+$. Then the class of $\pi_t(X)$ is given by $[\pi_t(X)]=\sigma_{a'^{\alpha'};b'^{\beta'}}$, where $m'\in\mathbb{Z}^+$ and $(a'^{\alpha'},b'^{\beta'})$ is obtained from $(a^\alpha;b^\beta)$ by removing the elements whose upper index is greater than $t$.
\end{Lem}
\begin{proof}
To compute the class of $\pi_t(X)$, we intersect it with the Schubert varieties in $OG(d_t;n)$ with complementary dimension, which corresponds to intersect $X$ with the pullback of those Schubert varieties. Since a coefficient vanishes or not depends only on the intersection product of the corresponding classes, we can presume that $X$ is a union of Schubert varieties, and the statement then follows from the definition of Schubert varieties.
\end{proof}
\begin{Lem}\label{class of pullback in of}
Let $X$ be a representative of $m\sigma_{a^\alpha;b^\beta}$, $m\in\mathbb{Z}^+$.  Set
$$\mu_{i,t}:=\#\{p|p\leq i,\alpha_p\leq t\},$$
$$\nu_{j,t}:=\#\{p|p\leq s,\alpha_p\leq t\}+\#\{q|q\geq j,\beta_q\leq t\},$$ 
$$x_{j,t}:=\#\{p|a_p\leq b_j,\alpha_p\leq t\},$$
$$h_{i,t}:=\#\{q|b_q\geq a_i,\beta_q\leq t\}.$$
Then the class of a general fiber of $\pi_t|_X$, $1\leq t\leq k$ is obtained as follows:
\begin{itemize}
\item If $\alpha_i\leq t$, replace $a_i$ with $\mu_{i,t}$ and put it together with the upper index $\alpha_i$ in a set $A$; 
\item If $\alpha_i> t$ and either $a_i+\mu_{s,t}-\mu_{i,t}+h_{i,t}\neq\frac{n}{2}$ or $\#\{q|\beta_q\leq t\}$ is even, replace $a_i$ with $a_i+\mu_{s,t}-\mu_{i,t}+h_{i,t}$ and put it together with the upper index $\alpha_i$ in $A$; 
\item If $\alpha_i> t$ and $a_i=a_i+\mu_{s,t}-\mu_{i,t}+h_{i,t}=\frac{n}{2}$ and $\#\{q|\beta_q\leq t\}$ is odd, replace $a_i$ with $\frac{n}{2}-1$ and put it together with the upper index $\alpha_i$ in a set $B$; 
\item If $\beta_j\leq t$, replace $b_j$ with $\nu_{j,t}$, and put it together with the upper index $\beta_j$ in $A$; 
\item If $\beta_j>t$ and $b_j+\nu_{j,t}-x_{j,t}\neq\frac{n}{2}-1$, replace $b_j$ with $b_j+\nu_{j,t}-x_{j,t}$, and put it together with the upper index $\beta_j$ in $B$. 
\item If $\beta_j>t$ and $b_j=\frac{n}{2}-1$ and $\#\{q|\beta_q\leq t\}$ is even, put it together with the upper index $\beta_j$ in $B$. 
\item If $\beta_j>t$ and $b_j=\frac{n}{2}-1$ and $\#\{q|\beta_q\leq t\}$ is odd, replace $b_j$ with $\frac{n}{2}$, and put it together with the upper index $\beta_j$ in $A$. 
\item If $\beta_j>t$ and $b_j\neq\frac{n}{2}-1$ and $b_j+\nu_{j,t}-x_{j,t}=\frac{n}{2}-1$, after all other sub-indices have been modified and put in either $A$ or $B$, let $A'$ and $B'$ be a copy of $A$ and $B$ respectively. Put $\frac{n}{2}$ together with the upper index $\beta_j$ in $A'$ and put $\frac{n}{2}-1$ together with the upper index $\beta_j$ in $B'$.
\end{itemize}
Let $A_\bullet, B_\bullet, A'_\bullet,B'_\bullet$ be the sequence obtained by rearranging the elements in $A,B,A',B'$ respectively so that the lower indices are in an increasing order. 

{\bf Case I.} If there is no $j$ such that $b_j\neq \frac{n}{2}$ and $b_j+\nu_{j,t}-x_{j,t}=\frac{n}{2}-1$, then the class of a general fiber of $\pi_t|_X$ is given by $m'\sigma_{A;B}$ for some $m'\in\mathbb{Z}^+$. 

{\bf Case II.} If $n=2d_k$ and there exists $j$ such that $b_j\neq \frac{n}{2}$ and $b_j+\nu_{j,t}-x_{j,t}=\frac{n}{2}-1$, then the class of a general fiber of $\pi_t|_X$ is given by $m'\sigma_{A';B}+m''\sigma_{A;B'}$ where $m'=0$ if $|A'|\equiv s\ (\text{mod }2)$ and $m''=0$ if $|A|\equiv s\ (\text{mod }2)$

{\bf Case III.} If there exists $j$ such that $b_j\neq \frac{n}{2}$ and $b_j+\nu_{j,t}-x_{j,t}=\frac{n}{2}-1$, then 
the class of a general fiber of $\pi_t|_X$ is given by $m'\sigma_{A';B}+m''\sigma_{A;B'}$ for some $0\neq m',m''\in\mathbb{Z}^+$. 
\end{Lem}
\begin{proof}
We specialize $X$ to a Schubert variety defined by a partial flag $F_\bullet$. Then the fiber of $\pi_t|_X$ over $\Lambda_t\in \pi_t(X)$ is the Schubert variety defined by the partial flag $G_\bullet:=\Lambda_{t\bullet}\subset F'_\bullet$, where $\Lambda_{t\bullet}$ is a complete flag in $\Lambda_t$ and $F'_\bullet$ is a partial flag obtained from $F_\bullet$ by replacing $F_{a_i}$ with $\text{span}(F_{a_i}\cap \Lambda_t^\perp,\Lambda_t)$ and replacing $F_{b_j}^\perp$ with $(\text{span}(F_{b_j}\cap \Lambda_t^\perp,\Lambda_t))^\perp$, and then remove redundant elements with no jumps of dimension. It is then straightforward to check
\begin{eqnarray}
\dim(\text{span}(F_{a_i}\cap \Lambda_t^\perp,\Lambda_t))&=&\dim(F_{a_i}\cap\Lambda_t^\perp)+\dim(\Lambda_t)-\dim(F_{a_i}\cap\Lambda_t)\nonumber\\
&=&\dim(F_{a_i})+\dim(F_{a_i}^\perp\cap\Lambda_t)-\dim(F_{a_i}\cap \Lambda_t)\nonumber\\
&=&a_i+\mu_{s,t}-\mu_{i,t}+h_{i,t}\nonumber
\end{eqnarray}
If $a_i+\mu_{s,t}-\mu_{i,t}+h_{i,t}=\frac{n}{2}$, then one needs to specify the irreducible component of $\text{span}(F_{a_i}\cap \Lambda_t^\perp,\Lambda_t)$ depending on the parity of $\#\{q|\beta_q\leq t\}$.

Similarly, \begin{eqnarray}
\dim(\text{span}(F_{b_j}\cap \Lambda_t^\perp,\Lambda_t))&=&\dim(F_{b_j})+\dim(F_{b_j}^\perp\cap\Lambda_t)-\dim(F_{b_j}\cap \Lambda_t)\nonumber\\
&=&b_j+\nu_{j,t}-x_{j,t}\nonumber
\end{eqnarray}
If $b_j=\frac{n}{2}-1$, then $\text{span}(F_{b_j}\cap \Lambda_t^\perp,\Lambda_t))$ is also a maximal isotropic subspace and one needs to specify the irreducible component depending on the parity of $\#\{q|\beta_q\leq t\}$.

If $b_j<\frac{n}{2}-1$ and $b_j+\nu_{j,t}-x_{j,t}=\frac{n}{2}-1$, then the orthogonal component of $\text{span}(F_{b_j}\cap \Lambda_t^\perp,\Lambda_t))$ intersects $Q$ into a union of two maximal isotropic subspaces belonging to different components, and therefore the class splits into a sum of two Schubert classes. When $n=2d_k$, only one of them is effective due to the requirement that $d_k\equiv s$ (mod 2).
\end{proof}

Also recall the following observation in orthogonal Grassmannians, which is useful in the later proof:
\begin{Lem}\label{atob}
Let $\sigma_{a;b}$ be a Schubert class for $OG(k,n)$. Let $X$ be a representative of $m\sigma_{a;b}$, $m\in\mathbb{Z}^+$. Assume $a_i=b_j$ for some $i$ and $j$. If there exists an isotropic subspace $F_{a_i}$ of dimension $a_i< \frac{n}{2}-1$ such that $\dim(F_{a_i}\cap \Lambda)\geq i$ for all $\Lambda\in X$, then 
$$\dim(F_{a_i}^\perp\cap \Lambda)\geq k-j+1,\forall \Lambda\in X.$$
Conversely, if there exists an isotropic subspace $F_{b_j}$ of dimension $b_j$ such that $\dim(F_{b_j}^\perp\cap \Lambda)\geq k-j+1$ for all $\Lambda\in X$, then 
$$\dim(F_{b_j}\cap \Lambda)\geq i,\forall \Lambda\in X.$$
\end{Lem}
\begin{proof}
If the statement is true for all irreducible components of $X$, then it is true for $X$. From now on, we assume $X$ is irreducible.

{\bf Case (A).} Assume there exists an isotropic subspace $F_{a_i}$ of dimension $a_i< \frac{n}{2}-1$ such that $\dim(F_{a_i}\cap \Lambda)\geq i$ for all $\Lambda\in X$. We use induction on $h_i$ by a decreasing order. 

If $a_i=i$, then $F_i\subset \Lambda$ implies $\Lambda\subset F_i^\perp$ for all $\Lambda\in X$. If $a_i-i>0$, let $\Phi_X$ be the variety swept out by $\mathbb{P}^{k-1}$ parametrized by $X$. 
\begin{Lem}\label{dim of q}
Let $X$ be a subvariety of $OG(k,n)$ with class $m\sigma_{a;b}$, $m\in\mathbb{Z}^+$. Let $\Phi_X$ be the variety swept out by the projective linear spaces $\mathbb{P}^{k-1}$ parametrized by $X$. If $b$ is non-empty, then $\dim(\Phi_X)=n-b_1-2$. If $b$ is empty, i.e. $s=k$, then $\dim(\Phi_X)=a_k-1$.
\end{Lem}
\begin{proof}
First assume $b$ is non-empty and $b_1\neq \frac{n}{2}-1$. Let $G_{b_1}$ be a general isotropic subspace of dimension $b_1$. Let $\Sigma$ be the Schubert variety that parametrizes all isotropic subspaces meeting $G_{b_1}$ in dimension at least 1. Let $\sigma=[\Sigma]$. Then from $\sigma\cdot \sigma_{a;b}=0$, we obtain that $\Sigma\cap X=\emptyset$. Therefore the dimension of $\Phi_X$ is at most $n-b_1-2$.

When $X$ is a Schubert variety, $\Phi_X$ is the intersection of $Q$ and $\mathbb{P}(F_{b_1}^\perp)$, in which cases $\Phi_X$ has dimension $n-b_1-2$. By semi-continuity, we conclude that $\dim(\Phi_X)=n-b_1-2$ for any $X$ of the class $m\sigma_{a;b}$.

Now assume $b$ is empty. Let $i:OG(k,n)\rightarrow G(k,n)$ be the natural inclusion. Then $i_*(\sigma_a)=\sigma_a$ where for the rightside we consider $a$ as a Schubert index for $G(k,n)$. Let $G_{n-a_s}$ and $G_{n-a_s+1}$ be general linear spaces in $V$ of dimension $n-a_s$ and $n-a_s+1$ respectively. Let $\Sigma'$ and $\Sigma''$ be the Schubert variety in $G(k,n)$ parametrizing all subspaces that meets $G_{n-a_s}$ and $G_{n-a_s+1}$ in dimension at least 1 respectively. From intersection product, we conclude that $\Sigma'\cap X=\emptyset$ while $\Sigma''\cap X\neq\emptyset$. Therefore $\dim(\Phi_X)=a_k-1$.
\end{proof}
By Lemma \ref{dim of q}, $\dim(\Phi_X)=n-b_1-2$. We prove $\Phi_X\subset \mathbb{P}(F^\perp_{a_i})$ by showing $\Phi_X\subset \mathbb{P}(F^\perp_{i})$ for any subspaces $\Lambda_i$ in $F_{a_i}$ of dimension $i$. Consider the following locus in $OF(i,k;n)$:
$$I:=\{(\Lambda_i,\Lambda)|\Lambda_i\subset(\Lambda\cap F_{a_i}),\Lambda\in X\}.$$
By specializing $X$ to a union of Schubert varieties, one can see the class of $I$ is given by $m'\sigma_{a^\alpha;b^\beta}$ where $m'\in\mathbb{Z}^+$, $\alpha_\mu=1$ if $\mu\leq i$, $\alpha_\mu=2$ if $\mu>i$ and $\beta_\nu=2$ for all $1\leq \nu\leq k-s$. Let $\Lambda_i$ be a general point in $\pi_1(I)$. By Lemma \ref{class of pushforward in of} and Lemma \ref{class of pullback in of}, the class of $(\pi_2\circ \pi_1^{-1}|_I)(\Lambda_i)$ is given by $m'\sigma_{a';b'}$ where $m'\in\mathbb{Z}^+$, $a'_\mu=\mu$ if $\mu\leq i$ and $b'_1=b_1$. Let $\Phi_i$ be the variety swept out by the projective linear spaces $\mathbb{P}^{k-1}$ parametrized by $(\pi_2\circ \pi_1^{-1}|_I)(\Lambda_i)$. It is then clear that $\Phi_i\subset\mathbb{P}(F_{i}^\perp)$. Apply Lemma \ref{dim of q} again, $\dim(\Phi_i)=\dim(\Phi_X)=n-b_1-2$. Since $X$ is irreducible, by the fiber dimension theorem, $\Phi_X$ is also irreducible. Therefore $\Phi_X=\Phi_i\subset \mathbb{P}(F_{i}^\perp)$ for a general $F_i$ in $\pi_1(I)$. By the semi-continuity of dimension, it holds for every subspace $F_i$ in $F_{a_i}$. We conclude that $\Lambda\subset F_{a_1}^\perp$ for all $\Lambda\in X$.

{\bf Case (B).} Conversely, assume $\Lambda\subset F_i^\perp$ for all $\Lambda\in X$. We use induction on $a_i-i$. For every irreducible component $X^r$ of $X$, consider the sub locus $X^r_S:=\{\Lambda\in X^r|F_i\subset\Lambda\}$. By specializing $X$ to a union of Schubert varieties, we obtain $\dim(X^r_S)=\dim(X^r)$, which implies $X^r_S=X^r$. Therefore $F_i\subset \Lambda$ for all $\Lambda\in X$. 

If $a_i-i>0$, consider the incidence correspondence
$$I:=\{(\Lambda,H)|\Lambda\in X, \Lambda\subset H,H\text{ is a hyperplane}\}.$$
Let $H$ be a general point in $\pi_2(I)$. Define $X_H:=\{\Lambda\in X|\Lambda\subset H\}$. Let $\Phi_H$ be the variety swept out by the projective linear spaces $\mathbb{P}^{k-1}$ parametrized by $X_H$. Then $\Phi_H\subset\mathbb{P}(F_{a_i}^\perp)\cap Q\cap H:=Q_H$, where the right side is a sub-quadric of $Q$ of dimension $n-a_i-3$ with corank $a_i-1$ and thus contained in a smooth quadric $Q'$of dimension $n-4$. Therefore $X_H$ can be viewed as a sub-variety of $OG(k,n-2)$. The class of $X_H$ in $OG(k,n-2)$ is given by $m'\sigma_{a',b'}$, where $m'\in\mathbb{Z}^+$, $a'_\mu=a_\mu$ if $a_\mu=\mu$, $a'_\mu=a_\mu-1$ if $a_\mu>\mu$ and $b'_\nu=b_\nu-1$ for all $1\leq \nu\leq k-s$. In particular, $a'_i=a_i-1=b_1-1=b'_1$. By Lemma \ref{dim of q}, $\dim(\Phi_H)=n-2-(b_1-1)-2=n-a_i-3=\dim(Q_H)$. Since the sub-quadric $Q_H$ is irreducible, $\Phi_H=Q_H$. Set $F_H=F_{a_i}\cap H$. $\dim(F_H)=a_i-1$. Notice that $F_H$ is contained in the singular locus of $Q_H$ and the maximal possible corank of a sub-quadric of dimension $n-a_i-3$ in $Q'$ is $a_i-1$. By induction, $\dim(F_{H}\cap \Lambda)\geq i$ for all $\Lambda\in X_H$. As varying $H$, $X_H$ cover $X$ and therefore $\dim(F_{a_i}\cap \Lambda)\geq i$ for all $\Lambda\in X$.

Now assume $j>1$. We use induction on $j$. Let $p$ be a general point in $\Phi_X$. The class of $X_p$ is given by $m'\sigma_{a';b'}$ where $a'_i=a_i=b'_{j-1}$. Set $F_{a_i}^p:=\text{span}(p,p^\perp\cap F_{a_i})$. $\dim(F_{a_i}^p)=a_i$. If $\dim(F_{a_i}\cap \Lambda)\geq i, \forall \Lambda\in X$, then for every $\Lambda\in X_p$, 
\begin{eqnarray}
\dim(\Lambda\cap F_{a_i}^p)&=&\dim(\Lambda\cap\text{span}(p,p^\perp\cap F_{a_i}))\nonumber\\
&=&1+\dim(\Lambda\cap F_{a_i})\nonumber\\
&\geq&i+1\nonumber
\end{eqnarray}
By induction, $\dim(\Lambda \cap (F_{a_i}^p)^\perp)\geq k-(j-1)+1=k-j+2$. Notice that
$$F_{a_i}^\perp\cap (F_{a_i}^p)^\perp=F_{a_i}^\perp\cap p^\perp\cap (p^\perp\cap F_{a_i})^\perp=F_{a_i}^\perp\cap p^\perp,$$
which is a codimension one subspace in $(F_{a_i}^p)^\perp$. Therefore 
\begin{eqnarray}
\dim(\Lambda\cap F_{a_i}^\perp)&\geq&\dim(\Lambda\cap (F_{a_i}^p)^\perp)-1\nonumber\\
 &\geq&n-j+1.\nonumber
\end{eqnarray}
Conversely, if $\dim(F_{a_i}^\perp\cap \Lambda)\geq k-j+1,\forall \Lambda\in X$, then 
\begin{eqnarray}
\dim(\Lambda\cap (F_{a_i}^p)^\perp)&=&\dim(\Lambda\cap p^\perp\cap(p^\perp\cap F_{a_i})^\perp)\nonumber\\
&=&\dim(\Lambda\cap (p^\perp\cap F_{a_i})^\perp)\nonumber\\
&\geq&\dim(\Lambda\cap p)+\dim(\Lambda\cap F_{a_i}^\perp)\nonumber\\
&\geq&k-j+2.\nonumber
\end{eqnarray}
By induction, $\dim(F_{a_i}^p\cap \Lambda)\geq i+1,\forall \Lambda\in X_p$. Since $$F_{a_i}\cap F_{a_i}^p=F_{a_i}\cap\text{span}(p,p^\perp\cap F_{a_i})=p^\perp\cap F_{a_i},$$
$\dim(F_{a_i}\cap \Lambda)\geq \dim(F_{a_i}^p\cap \Lambda)-1\geq i,\forall \Lambda\in X_p$. As varying $p$, we conclude that $\dim(F_{a_i}\cap \Lambda)\geq i,\forall \Lambda\in X$.
\end{proof}
\begin{Rem}
If $a_s=b_{k-s}=\frac{n}{2}-1$, then $F_{a_s}$ is a codimension 1 subspace of $F_{a_s}^\perp$. Therefore $\dim(F_{a_s}^\perp\cap\Lambda)\geq k-(k-s)+1=s+1$ implies $\dim(F_{a_s}\cap\Lambda)\geq s$.
\end{Rem}

Now we are able to prove the first main theorem of this section:
\begin{Thm}\label{orthogonal rigid index}
Let $X$ be a variety in $OF(d_1,...,d_k;n)$ with class $m\sigma_{a^\alpha;b^\beta}$, $m\in\mathbb{Z}^+$. Suppose $a_i$ (or $b_j$ resp.) is essential with respect to the Schubert class $(\pi_t)_*(\sigma_{a^\alpha;b^\beta})$ for some $t$ and there exists an isotropic subspace $F_{a_i}$ (or $F_{b_j}$ resp.) such that $\dim(F_{a_i}\cap\Lambda_t)\geq \mu_{i,t}$ (or $\dim(F^\perp_{b_j}\cap\Lambda_t)\geq \nu_{j,t}$ resp.) for all $\Lambda_t\in \pi_t(X)$, then the same inequality holds for all $1\leq t\leq k$.
\end{Thm}
\begin{proof}
By passing to irreducible components if necessary, one can assume $X$ is irreducible. 

{\bf Case (B).} First consider the case of $b_j$. Suppose, for a contradiction, that $\dim(F_{b_j}^\perp\cap\Lambda_r)<\nu_{j,r}$ for some $r\neq t$ and $\Lambda_r\in\pi_r(X)$. Since $X$ is irreducible, the image $\pi_r(X)$ is also irreducible. By the semi-continuity of dimension, $\dim(F_{b_j}^\perp\cap \Lambda_r)<\nu_{j,r}$ for a general $\Lambda_r\in \pi_r(X)$. Consider $\Lambda_r$ as a general point in $\pi_r(X)$. Let $\sum m'\sigma_{a';b'}$ be the class of $\pi_t(\pi_r|_X^{-1}(\Lambda_r))$ in $OG(k_t,n)$.

{\bf Case (B1).} If $t<r$, then by Lemma \ref{class of pushforward in of} and Lemma \ref{class of pullback in of}, $a'_{\nu_{j,t}}=\nu_{j,r}$. On the other hand, let $U:= F_{b_j}^\perp\cap \Lambda_r$. Clearly $U$ is isotropic since $\Lambda_r$ is. By assumption, $\dim(U)<\nu_{j,r}$. For all $\Lambda_t\in \pi_t(\pi_r|_X^{-1}(\Lambda_r))$,
\begin{eqnarray}
\dim(W\cap \Lambda_t)&=&\dim(F_{b_j}^\perp\cap\Lambda_r)\cap\Lambda_t\nonumber\\
&=&\dim(F_{b_j}^\perp\cap\Lambda_t )\nonumber\\
&\geq&\nu_{j,t}\nonumber
\end{eqnarray}
This contradicts the relation $a'_{\nu_{j,t}}=\nu_{j,r}$.

{\bf Case (B2).} Assume $r<t$. Apply {\bf Case (B1)} if necessary, we assume $t=k$. 

{\bf Case (B2-1).} If $d_k-d_r=\nu_{j,k}-\nu_{j,r}$, then since $\Lambda_r$ is of codimension $d_k-d_r$ in $\Lambda_k$,
\begin{eqnarray}
\dim(F_{b_j}^\perp\cap\Lambda_k)\geq \nu_{j,k}\Rightarrow\dim(F_{b_j}^\perp\cap\Lambda_r)&\geq&\nu_{j,k}-(d_k-d_r)\nonumber\\
&=&\nu_{j,r}\nonumber
\end{eqnarray}

{\bf Case (B2-2).} If $d_k-d_r\neq\nu_{j,k}-\nu_{j,r}$, let $h:=\max\{q|q<j, \beta_q>r\}.$ Set $$\phi:=h-\#\{q|q<h,\beta_q\leq r\}=h-(d_r-\nu_{h,r}).$$ From the construction, $\nu_{h,r}-\nu_{j,r}=j-h-1$ and therefore
$$\phi=-d_r+\nu_{j,r}+j-1$$

{\bf Case (B2-2-1).} If $b_h+\nu_{h,r}-x_{h,r}\neq \frac{n}{2}-1$, by Lemma \ref{class of pushforward in of} and Lemma \ref{class of pullback in of}, 
$b'_{\phi}=b_h+\nu_{h,r}-x_{h,r}.$ We show a contradiction that there exists either a maximal isotropic subspace or an isotropic subspace of dimension greater than $b'_\phi$ whose orthogonal complement intersects every $k$-plane in $({\pi_k}\circ \pi_r|_X^{-1})(\Lambda_r)$ in dimension at least $d_k-\phi+1$. The later case equivalent to find an isotropic subspace $W$ such that $\dim(W^\perp\cap\Lambda_k)\geq d_k-\phi+1$, $\forall\Lambda_k\in ({\pi_k}\circ \pi_r|_X^{-1})(\Lambda_r)$ and
$$b'_\phi+d_k-\phi+1-[\dim(W)+\dim(W^\perp\cap \Lambda_k)]<0.$$ 

{\bf Case (B2-2-1-1).} If $b_j=\frac{n}{2}-1$, then let $W:=\text{span}(F^\perp_{b_j}\cap \Lambda_r^\perp,\Lambda_r)$. Since $F^\perp_{n/2-1}$ is a maximal isotropic subspace, so is $W$. 
For a general $\Lambda_k\in ({\pi_k}\circ \pi_r|_X^{-1})(\Lambda_r)$,
\begin{eqnarray}
\dim(W\cap \Lambda_k)&=&\dim(\text{span}(F^\perp_{b_j}\cap \Lambda_r^\perp,\Lambda_r)\cap \Lambda_k)\nonumber\\
&=&d_r+\dim(F_{b_j}^\perp\cap \Lambda_k)-\dim(F^\perp_{b_j}\cap\Lambda_r)\nonumber\\
&\geq& d_r+d_k-j+1-(\nu_{j,r}-1)=d_k-\phi+1\nonumber
\end{eqnarray}

{\bf Case (B2-2-1-2).} If $b_j<\frac{n}{2}-1$, let $W:=\text{span}(F_{b_j}\cap \Lambda_r^\perp,\Lambda_r).$  Then $$\dim(W)=d_r+\dim(F_{b_j}\cap\Lambda_r^\perp)-\dim(F_{b_j}\cap \Lambda_r),$$
and for a general $\Lambda_k\in ({\pi_k}\circ \pi_r|_X^{-1})(\Lambda_r)$,
\begin{eqnarray}
\dim(W^\perp\cap \Lambda_k)&=&\dim((F_{b_j}\cap\Lambda_r^\perp)^\perp\cap \Lambda_r^\perp\cap\Lambda_k)\nonumber\\
&=&d_k-\dim(F_{b_j}\cap \Lambda_r^\perp)+\dim(F_{b_j}\cap\Lambda_r^\perp\cap\Lambda_k^\perp)\nonumber\\
&=&d_k-(b_j-d_r+\dim(F_{b_j}^\perp\cap\Lambda_r))+(b_j-d_k+\dim(F_{b_j}^\perp\cap\Lambda_k))\nonumber\\
&=&d_r-\dim(F_{b_j}^\perp\cap\Lambda_r)+\dim(F_{b_j}^\perp\cap\Lambda_k)\nonumber\\
&\geq& d_r-\nu_{j,r}+1+d_k-j+1=d_k-\phi+1\nonumber
\end{eqnarray}
If $\dim(W)=\frac{n}{2}$, then $\dim(W\cap\Lambda_k)=\dim(W^\perp\cap \Lambda_k)\geq d_k-\phi+1$. If $\dim(W)\leq\frac{n}{2}-1$, then
\begin{eqnarray}
\dim(W)+\dim(W^\perp\cap\Lambda_k)&=&d_k+d_r+(b_j-d_k+\dim(F_{b_j}^\perp\cap\Lambda_k))-\dim(F_{b_j}\cap\Lambda_r)\nonumber\\
&=&d_r+b_j+d_k-j+1-\dim(F_{b_j}\cap\Lambda_r)\nonumber\\
&\geq&d_r+b_j+d_k-j+1-x_{j,r}
\end{eqnarray}
Note that by definition $a_i-b_j\neq 1$ for all $i$ and $j$, which implies 
\begin{eqnarray}
x_{j,r}-x_{h,r}&=&\#\{p|b_h< a_p\leq b_j,\alpha_p\leq r\}\nonumber\\
&\leq& b_j-b_h-(j-h).
\end{eqnarray}
Therefore 
\begin{eqnarray}
b'_\phi+d_k-\phi+1-[\dim(W)+\dim(W^\perp\cap \Lambda_k)]&\leq&b_h+\nu_{h,r}-x_{h,r}+d_k-(h-(d_r-\nu_{h,r}))+1\nonumber\\
& &-(d_r+b_j+d_k-j+1-x_{j,r})\nonumber\\
&=&b_h-b_j+x_{j,r}-x_{h,r}+j-h\leq0
\end{eqnarray}
We claim that the equalities in (1) and (2) can not be reached at the same time and therefore the inequality in (3) is strict. Assume $x_{j,r}-x_{h,r}=b_j-b_h-(j-h)$, then it is necessary that $b_j=a_i$ and $\alpha_i\leq r$ for some $1\leq i\leq s$. If the equality in (1) holds, i.e. $\dim(F_{b_j}\cap\Lambda_r)=x_{j,r}$ for a general $\Lambda_r$, then by the semi-continuity of dimension, $\dim(F_{b_j}\cap\Lambda_r)\geq x_{j,r}$ for all $\Lambda_r\in\pi_r(X)$. By Lemma \ref{atob}, $\dim(F_{b_j}^\perp\cap\Lambda_r)\geq\nu_{j,r}$ for all $\Lambda_r\in \pi_r(X)$. We reach a contradiction. 

{\bf Case (B2-2-2).} If $b_h+\nu_{h,r}-x_{h,r}=\frac{n}{2}-1$, then for every $\epsilon\in (b_h,\frac{n}{2}-1]$, either $\epsilon\in b_\bullet$ or $\epsilon+1\in a_\bullet$ with upper index $\leq r$. Since $b_j$ is essential, it is necessary that $b_j=a_i$ for some $i$ and $b_{j-1}\neq b_j+1$. Let $\pi:OF(d_1,...,d_k;n)\rightarrow OF(d_r,d_k;n)$ be the canonical projection. Let $X':=\pi(X)\subset OF(d_r,d_k;n)$ be the image of $X$. Consider the following locus in $OF(d_0,d_r,d_k;n)$, $d_0:=d_r-\nu_{j,r}+1$,
$$I:=\{(\Lambda_0\subset \Lambda_r\subset\Lambda_k)|\Lambda_0\subset \Lambda_r,\dim(\Lambda_0)=d_0, (\Lambda_r,\Lambda_k)\in X'\}$$
$I$ is the preimage of $X'$ under the projection $OF(d_0,d_r,d_k;n)\rightarrow OF(d_r,d_k;n)$. The class of $X'$ in $OF(d_r,d_k;n)$ is given by $m_{r,k}\sigma_{a^{\alpha^{r,k}};b^{\beta^{r,k}}}$, where $\alpha^{r,k}$ and $\beta^{r,k}$ are obtained from $\alpha$ and $\beta$ by changing the numbers less than $r$ to $r$ and changing the numbers greater than $r$ to $k$, respectively. The class of $I$ in $OF(d_0,d_r,d_k;n)$ is then given by $m_0\sigma_{a^{\alpha^0};b^{\beta^0}}$, where $\alpha^0=\alpha^{r,k}$, $\beta^0_{j'}=\beta^{r,k}_{j'}$ if $j'>j$ or $\beta^{r,k}_{j'}\neq r$, and $\beta^0_{j'}=0$ if $j'\leq j$ and $\beta^{r,k}_{j'}=r$. Consider the projection $OF(d_0,d_r,d_k;n)\rightarrow OG(d_0,n)$. Since $\alpha^0_i=r$, we obtain $b_h+\nu_{h,0}-x_{h,0}\neq \frac{n}{2}-1$. Apply {\bf Case (B2-2-1)}, we obtain $\dim(F_{b_j}^\perp\cap\Lambda_0)\geq\nu_{j,0}$ for all $\Lambda_0\in \pi_0(I)$. Since $d_0<d_r$, we apply {\bf Case (B1)} to conclude $\dim(F_{b_j}^\perp\cap\Lambda_r)\geq \nu_{j,r}$ for all $\Lambda_r\in\pi_r(X)$. 

{\bf Case (A).} Now we consider the case of $a_i$. Suppose for a contradiction that $\dim(F_{a_i}\cap \Lambda_r)<\mu_{i,r}$ for some $r$ and $\Lambda_r\in \pi_r(X)$. Since $X$ is irreducible, the image $\pi_r(X)$ is also irreducible. By the semi-continuity of dimension, $\dim(F_{a_i}\cap \Lambda_r)<\mu_{i,r}$ for a general $\Lambda_r\in \pi_r(X)$. Let $\sum m'\sigma_{a';b'}$ be the class of $\pi_t(\pi_r|_X^{-1}(\Lambda_r))$ in $OG(k_t,n)$.

{\bf Case (A1).} If $t<r$, by Lemma \ref{class of pushforward in of} and Lemma \ref{class of pullback in of}, $a'_{\mu_{i,t}}=\mu_{i,r}$. On the other hand, set $W=F_{a_i}\cap\Lambda_r$. By assumption, $\dim(W)<\mu_{i,r}$. Since $\dim(F_{a_i}\cap\Lambda_t)\geq \mu_{i,t}$ for all $\Lambda_t\in\pi_t(X)$, we obtain for all $\Lambda_t\in \pi_t(\pi_r|_X^{-1}(\Lambda_r))$,
$$\dim(W\cap\Lambda_t)=\dim(F_{a_i}\cap(\Lambda_r\cap\Lambda_t))= \dim(F_{a_i}\cap\Lambda_t)\geq \mu_{i,t}.$$
This contradicts the condition that $a'_{\mu_{i,t}}=\mu_{i,r}$.

{\bf Case (A2).} If $r<t$, by {\bf Case (A1)} we may assume $t=k$. Since $a_i$ is essential in the $k$-th component, it is necessary that $a_{i}\neq a_{i+1}-1$ and $a_i+b_{k-s}\neq n-2$. 

{\bf Case (A2-1).} If $d_k-d_r=\mu_{i,k}-\mu_{i,r}$, then since $\Lambda_r$ is of codimension $d_k-d_r$ in $\Lambda_k$, we must have
$$\dim(F_{a_i}\cap \Lambda_r)\geq \dim(F_{a_i}\cap\Lambda_k)-(d_k-d_r)\geq \mu_{i,k}-(\mu_{i,k}-\mu_{i,r})=\mu_{i,r}.$$

{\bf Case (A2-2).} If $d_k-d_r\neq \mu_{i,k}-\mu_{i,r}$, then there exists a sub-index after $a_i$ with the upper index greater than $r$.

{\bf Case (A2-2-1).} Assume such sub-indices exist in the sequence $a_\bullet$. Let $l$ be the minimal number such that $a_l>a_i$ and $\alpha_l>r$. Set $\gamma=:d_r+l-\mu_{l,r}$. By the construction, we have $\mu_{l,r}-\mu_{i,r}=l-i-1$ and therefore
$$\gamma=d_r+l-\mu_{l,r}=d_r+i-\mu_{i,r}+1.$$

{\bf Case (A2-2-1-1).} If either $a_l+\mu_{s,r}-\mu_{l,r}+h_{l,r}\neq\frac{n}{2}$ or $\#\{q|\beta_q\leq r\}$ is even, then by Lemma \ref{class of pushforward in of} and Lemma \ref{class of pullback in of}, $$a'_\gamma=a_l+\mu_{s,r}-\mu_{l,r}+h_{l,r}.$$
We show a contradiction by showing there exists an isotropic subspace of dimension less than $a'_\gamma$ that intersects every $k$-plane in $({\pi_k}\circ \pi_r|_X^{-1})(\Lambda_r)$ in dimension at least $\gamma$, or equivalently an isotropic subspace $W$ such that $\dim(W\cap\Lambda_k)\geq \gamma$, $\forall\Lambda_k\in ({\pi_k}\circ \pi_r|_X^{-1})(\Lambda_r)$ and
$$a'_\gamma-\gamma-[\dim(W)-\dim(W\cap \Lambda_k)]>0.$$
Let $W:=$span$(F_{a_i}\cap \Lambda_r^\perp,\Lambda_r)$. Then 
$$\dim(W)=d_r+\dim(F_{a_i}\cap\Lambda_r^\perp)-\dim(F_{a_i}\cap \Lambda_r),$$
and for a general $\Lambda_k\in ({\pi_k}\circ \pi_r|_X^{-1})(\Lambda_r)$,
\begin{eqnarray}
\dim(W\cap \Lambda_k)&=&d_r+\dim(F_{a_i}\cap\Lambda_k)-\dim(F_{a_i}\cap\Lambda_r)\nonumber\\
&\geq& d_r+i-\mu_{i,r}+1=\gamma\nonumber
\end{eqnarray}
Therefore 
\begin{eqnarray}
\dim(W)-\dim(W\cap\Lambda_k)&=&\dim(F_{a_i}\cap \Lambda_r^\perp)-\dim(F_{a_i}\cap\Lambda_k)\nonumber\\
&=&a_i-d_r+\dim(F_{a_i}^\perp\cap\Lambda_r)-i\nonumber\\
&\leq&a_i-d_r+\mu_{s,r}+h_{i,r}-i
\end{eqnarray}
Note that by definition $a_i-b_j\neq 1$ for all $i$ and $j$, which implies 
\begin{eqnarray}
h_{i,r}-h_{l,r}&=&\#\{q|a_i\leq b_q<a_l,\beta_q\leq r\}\nonumber\\
&\leq& a_l-a_i-(l-i).
\end{eqnarray}
Therefore 
\begin{eqnarray}
a'_\gamma-\gamma-[\dim(W)-\dim(W\cap \Lambda_k)]&\geq&a_l+\mu_{s,r}-\mu_{l,r}+h_{l,r}-(d_r+l-\mu_{l,r})\nonumber\\
& &-(a_i-d_r+\mu_{s,r}+h_{i,r}-i)\nonumber\\
&=&a_l-a_i-(l-i)+h_{l,r}-h_{i,r}\geq0
\end{eqnarray}
We claim that the equalities in (4) and (5) can not be reached at the same time and therefore the inequality in (6) is strict. Assume $h_{i,r}-h_{l,r}=a_l-a_i-(l-i)$. Since $a_{i+1}\neq a_i+1$, it is necessary that $a_i=b_j$ and $\beta_j\leq r$ for some $1\leq j\leq k-s$. If the equality in (4) holds, i.e. $\dim(F_{a_i}^\perp\cap\Lambda_r)=\mu_{s,r}+h_{i,r}$ for a general $\Lambda_r$, then by the semi-continuity of dimension, $\dim(F_{a_i}^\perp\cap\Lambda_r)\geq\mu_{s,r}+h_{i,r}$ for all $\Lambda_r\in\pi_r(X)$. By Lemma \ref{atob}, $\dim(F_{a_i}\cap\Lambda_r)\geq\mu_{i,r}$ for all $\Lambda_r\in \pi_r(X)$. We reach a contradiction. 

{\bf Case (A2-2-1-2).} If $a_l+\mu_{s,r}-\mu_{l,r}+h_{l,r}=\frac{n}{2}$ and $\#\{q|\beta_q\leq r\}$ is odd, then $b'_{k-\gamma+1}=\frac{n}{2}-1$ which corresponding to a maximal isotropic subspace (of dimension $\frac{n}{2}$) in the other component. Replace $a'_\gamma-\gamma$ with $\frac{n}{2}-\gamma$. Then it follows from an identical argument in {\bf Case (A2-2-1-1).}

{\bf Case (A2-2-2).} If there is no $l$ such that $a_l>a_i$ and $\alpha_l>r$, let $h:=\max\{q|\beta_q>r\}$. Notice that $d_k-h-i=\nu_{h,r}-\mu_{i,r}$. Set $$\phi:=h-\#\{q|q<h,\beta_q\leq r\}=h-(d_r-\nu_{h,r}).$$ Let $\sigma_{a';b'}$ and $W$ defined as before. 

{\bf Case (A2-2-2-1).} If $b_h+\nu_{h,r}-x_{h,r}\neq \frac{n}{2}-1$, by Lemma \ref{class of pushforward in of} and Lemma \ref{class of pullback in of}, $$s'=(|B|\text{ or }|B'|)\leq d_k-\phi$$ On the other hand, for every $\Lambda_k\in ({\pi_k}\circ \pi_r|_X^{-1})(\Lambda_r)$,
\begin{eqnarray}
\dim(W\cap \Lambda_k)&=&d_r+\dim(F_{a_i}\cap\Lambda_k)-\dim(F_{a_i}\cap\Lambda_r)\nonumber\\
&\geq& d_r+i-\mu_{i,r}+1\nonumber\\
&=&d_k-h+d_r-\nu_{h,r}+1=d_k-\phi+1\nonumber
\end{eqnarray}
Since $W$ is isotropic, this contradicts the condition $s'\leq d_k-\phi.$ 

{\bf Case (A2-2-2-2).} If $b_h+\nu_{h,r}-x_{h,r}=\frac{n}{2}-1$ and $a_i<\frac{n}{2}$, then since $a_i$ is essential, it is necessary that $a_i=b_j$ for some $j$ with $\beta_j\leq r$. By assumption and Lemma \ref{atob}, $\dim(F_{a_i}^\perp\cap\Lambda_k)\geq\nu_{j,k}$ for all $\Lambda_k\in\pi_k(X)$. Apply {\bf Case (B2-2-2)} to $b_j=a_i$, we obtain $\dim(F_{a_i}^\perp\cap\Lambda_r)\geq\nu_{j,r}$ for all $\Lambda_r\in\pi_r(X)$. Apply Lemma \ref{atob} again, we conclude $\dim(F_{a_i}\cap\Lambda_r)\geq\mu_{i,r}$ for all $\Lambda_r\in\pi_r(X)$.

{\bf Case (A2-2-2-3).} If $b_h+\nu_{h,r}-x_{h,r}=\frac{n}{2}-1$ and $a_i= \frac{n}{2}$, then since $a_i$ essential, $b_{k-s}\neq \frac{n}{2}-1$ and therefore it is necessary that $a_{i-1}=a_i+1$ and $\alpha_{i-1}\leq r$. We play a same trick as in {\bf Case (B2-2-2).} Let $\pi:OF(d_1,...,d_k;n)\rightarrow OF(d_r,d_k;n)$ be the canonical projection. Let $X':=\pi(X)\subset OF(d_r,d_k;n)$ be the image of $X$. Let $I$ is the preimage of $X'$ under the projection $OF(d_0,d_r,d_k;n)\rightarrow OF(d_r,d_k;n)$, where $d_0:=d_r-\mu_{i,r}+1$. The class of $X'$ in $OF(d_r,d_k;n)$ is given by $m_{r,k}\sigma_{a^{\alpha^{r,k}};b^{\beta^{r,k}}}$, where $\alpha^{r,k}$ and $\beta^{r,k}$ are obtained from $\alpha$ and $\beta$ by changing the numbers less than $r$ to $r$ and changing the numbers greater than $r$ to $k$, respectively. The class of $I$ in $OF(d_0,d_r,d_k;n)$ is then given by $m_0\sigma_{a^{\alpha^0};b^{\beta^0}}$, where $\beta^0=\beta^{r,k}$, $\alpha^0_{i'}=\alpha^{r,k}_{i'}$ if $i'>j$ or $\alpha^{r,k}_{j'}\neq r$, and $\alpha^0_{i'}=0$ if $i'\leq j$ and $\alpha^{r,k}_{i'}=r$. Consider the projection $OF(d_0,d_r,d_k;n)\rightarrow OG(d_0,n)$. Since $\alpha^0_{i-1}=r$, we obtain $b_h+\nu_{h,0}-x_{h,0}\neq \frac{n}{2}-1$. Apply {\bf Case (A2-2-2-1)}, we obtain $\dim(F_{a_i}\cap\Lambda_0)\geq\mu_{j,0}$ for all $\Lambda_0\in \pi_0(I)$. Since $d_0<d_r$, we apply {\bf Case (A1)} to conclude $\dim(F_{a_i}\cap\Lambda_r)\geq \mu_{i,r}$ for all $\Lambda_r\in\pi_r(X)$. 
\end{proof}

\subsection{Restriction varieties}
In the next subsection, we investigate the multi-rigidity of sub-indices in orthogonal Grassmannians. It turns out to be useful to consider the restriction varieties which generalizes Schubert varieties. In this subsection, we review some basic facts regarding restriction varieties. The bilinear form $q$ geometrically defines a quadric hypersurface $Q$ in $\mathbb{P}(V)$. Therefore in the definition of Schubert varieties, one may replace $F_{b_j}^\perp$ with the intersection $\mathbb{P}(F_{b_j}^\perp)\cap Q$ , which is a sub-quadric of $Q$ of corank $b_j$. One may want to release the restriction on the coranks of those sub-quadrics. The resulting locus of isotropic subspaces are so-called restriction varieties, as they can also be considered as a restriction of a Schubert variety in Grassmannians.

Let $Q_d^r$ denote a subquadric of $Q$ of corank $r$, which is obtained by restricting $Q$ to a $d$-dimensional linear subspace. It turns out that there are some constraints that we need to impose so that the restriction varieties are in good shape (e.g. irreducibility or refined partial flag).

\begin{Def}
\label{sequence}
Given a sequence consisting of isotropic subspaces $F_{a_i}$ of $V$ and subquadrics $Q_{d_j}^{r_j}$ :
$$F_{a_1}\subsetneq...\subsetneq F_{a_s}\subsetneq Q_{d_{k-s}}^{r_{k-s}}\subsetneq...\subsetneq Q^{r_1}_{d_1}$$ such that

$(1)$ For every $1\leq j\leq k-s-1$, the singular locus of $Q_{d_j}^{r_j}$ is contained in the singular locus of $Q_{d_{j+1}}^{r_{j+1}}$;

$(2)$ For every pair $(F_{a_i},Q_{d_j}^{r_j})$, $\dim(F_{a_i}\cap \text{Sing}(Q_{d_j}^{r_j}))=\min\{a_i,r_j\}$;

$(3)$ Either $r_i=r_1=n_{r_1}$ or $r_t-r_i\geq t-i-1$ for every $t>i$. Moreover, if $r_t=r_{t-1}>r_1$ for some $t$, then $d_i-d_{i+1}=r_{i+1}-r_i$ for every $i\geq t$ and $d_{t-1}-d_t=1$;

$(A1)$ $r_{k-s}\leq d_{k-s}-3$;

$(A2)$ $a_i-r_j\neq 1$ for all $1\leq i\leq s$ and $1\leq j\leq k-s$;

$(A3)$ Let $x_j=\#\{i|a_i\leq r_j\}$. For every $1\leq j\leq k-s$, $$x_j\geq k-j+1-\left[\frac{d_j-r_j}{2}\right].$$
We define the associated restriction variety as the following locus
$$\Gamma(F_\bullet,Q_\bullet):=\{\Lambda\in OG(k,n)|\dim(\Lambda\cap F_{a_i})\geq i,\dim(\Lambda\cap Q_{d_j}^{r_j})\geq k-j+1,1\leq i\leq s, 1\leq j\leq k-s\}.$$
\end{Def}
Since the Schubert classes form a free abelian basis of the Chow ring of $OG(k,n)$, it is possible to write the class of a restriction variety as a linear combination of Schubert classes. One may find the coefficients using specialization.

\subsection{Multi-rigidity of sub-indices} In this subsection, we study the multi-rigidity of essential sub-indices. In particular, we characterize all multi-rigid sub-indices and multi-rigid Schubert classes in orthogonal Grassmannians. We start with the definition of multi-rigid sub-indices:
\begin{Def}
Let $\sigma_{a;b}$ be a Schubert class for $OG(k,n)$. An essential sub-index $a_i$ is called {\em multi-rigid} if for every irreducible representative $X$ of the class $m\sigma_{a;b}$, $m\in\mathbb{Z}^+$, there exists an isotropic subspace $F_{a_i}$ of dimension $a_i$ such that for every $k$-plane $\Lambda$ parametrized by $X$,
$$\dim(\Lambda\cap F_{a_i})\geq i.$$
Similarly, an essential sub-index $b_j$ is called {\em multi-rigid} if for every irreducible representative of the class $m\sigma_{a;b}$, $m\in\mathbb{Z}^+$, there exists an isotropic subspace $F_{b_j}$ of dimension $b_j$ such that for every $k$-plane $\Lambda$ parametrized by $X$, $$\dim(\Lambda\cap F_{b_j}^\perp)\geq k-j+1.$$
\end{Def}

\begin{Def}
Let $\sigma_{a^\alpha;b^\beta}$ be a Schubert class for $OF(d_1,...,d_k;n)$. An essential sub-index $a_i$ is called {\em multi-rigid} if for every irreducible representative $X$ of $m\sigma_{a^\alpha;b^\beta}$, $m\in\mathbb{Z}^+$, there exists an isotropic subspace $F_{a_i}$ of dimension $a_i$ such that 
$$\dim(F_{a_i}\cap\Lambda_t)\geq \mu_{i,t}, \forall (\Lambda_1,...,\Lambda_k)\in X,\alpha_i\leq t\leq k.$$
An essential sub-index $b_j$ is called {\em multi-rigid} if for every irreducible representative $X$, there exists an isotropic subspace $F_{b_j}$ of dimension $b_j$ such that
$$\dim(F_{b_j}^\perp\cap\Lambda_t)\geq \nu_{j,t}, \forall (\Lambda_1,...,\Lambda_k)\in X,\beta_j\leq t\leq k.$$
\end{Def}
The following result follows directly from the definition and Theorem \ref{orthogonal rigid index}:
\begin{Cor}
Let $\sigma_{a^\alpha;b^\beta}$ be a Schubert class for $OF(d_1,...,d_k;n)$. An essential $a_i$ (or $b_j$ resp.) is multi-rigid if $a_i$ (or $b_j$ resp.) is multi-rigid with respect to the Schubert class $(\pi_i)_*(\sigma_{a^\alpha;b^\beta})$ for some $i$.
\end{Cor}

In \cite{YL3}, we have proved the multi-rigidity of $a_i$ under certain conditions:
\begin{Thm}\cite[Theorem 5.10]{YL3}\label{arigid1}
Let $\sigma_{a;b}$ be a Schubert class in $OG(k,n)$. Let $a_i$ be an essential sub-index. Set $x_j:=\#\{i|a_i\leq b_j\}$. Assume for all $j$ such that $b_j<a_{i+1}$, $$x_j\geq k-j+1-\left[\frac{n-b_j-(a_{x_j+1}-1)}{2}\right].$$
Then $a_i$ is multi-rigid if one of the following conditions hold:
\begin{itemize}
\item $i<s$ and $a_{i-1}+1=a_i\leq a_{i+1}-3$;
\item $i=s$, $a_s=a_{s-1}+1$ and $a_s\leq n-b_{k-s}-(s-x_j)-4$.
\end{itemize}
\end{Thm}

In particular, this implies the following results:
\begin{Cor}\label{an}
Let $\sigma_{a;b}=\sigma_{1,...,b,b+2,...,k;b}$ be a Schubert class for $OG(k,n)$. Assume $n\geq 2k+3$. The sub-index $a_{s}=k$ is multi-rigid if and only if $k\geq b+3$.
\end{Cor}
\begin{proof}
If $k\geq b+3$, then by Theorem \ref{arigid1}, the sub-index $a_s=k$ is multi-rigid.

Conversely, assume $k=b+2$. We construct an irreducible subvariety $X$ with class $m\sigma_{a;b}$, $m\in\mathbb{Z}^+$ which is not a Schubert variety. Pick an isotropic partial flag:
$$F_{1}\subset...\subset F_b\subset F_{b+3}.$$
It is possible since $n\geq 2k+3=2(b+3)$. Let $C$ be a smooth plane curve of degree $m$ in $\mathbb{P}(F_{b+3})$ which is disjoint from $\mathbb{P}(F_b)$. Let $Z$ be the cone over $C$ with vertex $\mathbb{P}(F_{b})$. Let $X$ be the Zariski closure of the following locus:
$$X^o:=\{\Lambda\in OG(k,n)|F_b\subset\Lambda, \dim(\mathbb{P}(\Lambda)\cap Z)\geq k-2\}.$$
Then $X$ is irreducible with class $m\sigma_{a;b}$ but is not a Schubert variety unless $m=1$.
\end{proof}

We prove Corollary \ref{an} is also true for $n\leq 2k+2$. Recall for the Spinor varieties $OG(k,2k)$, the classification of multi-rigid Schubert classes is obtained by Robles and The:
\begin{Thm}\cite[Theorem 4.1]{RT}\label{RTresult}
Let $\sigma_{a;b}$ be a Schubert class for $OG(k,2k)$. If $a_s=k$, let $a'=(a_1,...,a_{s-1})$ be a subsequence of $a$. Otherwise, let $a'=a$. The Schubert class $\sigma_{a;b}$ is multi-rigid if and only if all essential sub-indices of $a'$, considered as a Schubert index for Grassmannians, satisfy the condition
$$a'_{i-1}+1=a'_i\leq a'_{i+1}-3.$$
\end{Thm}
\begin{Prop}\label{4.6}
Let $\sigma_{a;b}$ be a Schubert class for $OG(k,2k)$. 
\begin{itemize}
\item If $a_s=b_{k-s}=k-1$ and $a_{s-1}=k-2$, then the sub-index $b_{k-s}=k-1$ is multi-rigid.
\item If $a_s=k=a_{s-1}+1=a_{s-2}+2$, then the sub-index $a_s=k$ is multi-rigid.
\end{itemize}
\end{Prop}
\begin{proof}
Let $X$ be an irreducible representative of $m\sigma_{a;b}$, $m\in\mathbb{Z}^+$. If $a_s=k$, set $\mu_a:=\#\{i|a_{s-i}<k-i\}$. If $a_s=k-1$, set $\mu_a:=\#\{i|a_{s-i}<k-i-1\}$. We use induction on the $\mu_a$. If $\mu_a=0$, then $\sigma_{a;b}$ is multi-rigid by Theorem \ref{RTresult}.

If $\mu_a\geq 1$, consider the orthogonal partial flag variety $OF(k-1,k;2k)$. Let $\pi_1:OF(k-1,k;2k)\rightarrow OG(k-1,2k)$ and $\pi_2:OF(k-1,k;2k)\rightarrow OG(k,2k)$ be the two projections. Then the pullback of $\sigma_{a;b}$ under $\pi_2$ is the Schubert class $\sigma_{a^\alpha;b^\beta}$, where $\alpha_1=2$, $\alpha_i=1$ for $1<i<s$, $\beta_j=j$ for $1\leq j\leq k-s$. Therefore $Y:=(\pi_1\circ\pi_2^{-1})(X)$ has class $m'\sigma_{a';b'}$ where $m'\in\mathbb{Z}^+$, $a'$ is obtained from $a$ by removing $a_1$ and $b'=b$. 

First assume $a_s=b_{k-s}=k-1$. Let $Y'$ be the image of $Y$ under an involution of the quadric that interchanges the two irreducible components of the space of maximal isotropic subspaces. The class $m'\sigma_{a'';b''}$ of $Y'$ is obtained from $Y$ by adding $k$ to the sequence $a$ and deleting $k-1$ from the sequence $b$. Let $Z:=(\pi_2\circ\pi_1^{-1})(Y')\subset OG(k,2k)$. Then $Z$ has class $m'''\sigma_{a''';b'''}$ where $m'''\in\mathbb{Z}^+$, $a'''=a''$, $b'''_1=0$, $b'''_{j+1}=b_j$ for $1\leq j\leq k-s$. Notice that $\mu_{a'''}=\mu_a-1$. By induction, the sub-index $a'''_s=k$ is multi-rigid for the class $\sigma_{a''';b'''}$. Apply Theorem to \ref{orthogonal rigid index} to the class of $\pi_1^{-1}(Y')$, we obtain the sub-index $a''_s=k$ is multi-rigid for the Schubert class $\sigma_{a'';b''}$, which in turn shows that $b'_{k-s}=k-1$ is multi-rigid with respect to the class $\sigma_{a';b'}$. Apply Theorem to \ref{orthogonal rigid index} again to the class $\sigma_{a^\alpha;b^\beta}$, the sub-index $b_{k-s}=k-1$ is multi-rigid for the class $\sigma_{a;b}$. 

If $a_s=k$, then $a''$ is obtained from $a$ by deleting $k$ and $b''$ is obtained from $b$ by adding $k-1$. Then the same argument as in the previous case follows.
\end{proof}

\begin{Prop}\label{2k+1}
Consider the Schubert class $\sigma_{a;b}=\sigma_{1,...,b,b+2,...,k;b}$ for $OG(k,n)$, $n=2k+1$ or $2k+2$. The sub-index $a_{s}=k$ is multi-rigid if $k\geq b+3$.
\end{Prop}
\begin{proof}
Let $X$ be an irreducible representative of $m\sigma_{a;b}$, $m\in\mathbb{Z}^+$. 

{\bf Case I.} First assume $n=2k+1$. Consider the incidence correspondence:
$$I:=\{(\Lambda,H)|\Lambda\in X,\Lambda\subset H,H\subset G(n-1,n)\}.$$
Let $\pi_1$ and $\pi_2$ be the two projections to $X$ and $G(n-1,n)$ respectively. For a general point $H\in\pi_2(I)$, consider the fiber $X_H:=\{\Lambda\in X|\Lambda\subset H\}$. For every $\Lambda\in X_H$, $\Lambda\subset \mathbb{P}(H)\cap Q:=Q'$, where $Q'$ is a codimension 1 smooth sub-quadric of $Q$. Let $OG=OG(k,2k)\cup OG(k,2k)'$ be the space of $k$-dimensional isotropic subspaces contained in $Q'$ which has two irreducible components. By specializing $X$ to a union of Schubert varieties, one see $X_H$ has dimension $0$, which consists of $m'$ pair of $k$-planes where two $k$-planes in each pair belongs to different components and intersect in a subspace $\Lambda_{k-1}$ of dimension $k-1$. Let $Y^o$ be the locus of such $\Lambda_{k-1}$ as we varying $H$ in $\pi_2(I)$ and let $Y$ be the Zariski closure of $Y^o$ in $OG(k-1,2k+1)$. Then $Y$ has class $m''\sigma_{1,...,b,b+2,...,k}$ for some $m''\in\mathbb{Z}^+$. Notice that $\sigma_{1,...,b,b+2,...,k}$ is of Grassmannian type, $i_*(\sigma_{1,...,b,b+2,...,k})=\sigma_{1,...,b,b+2,...,k}$ where $i$ is the inclusion $OG(k-1,2k+1)\hookrightarrow G(k-1,2k+1)$. By Theorem \ref{rigid in g}, the sub-index $k$ is multi-rigid for the class $\sigma_{1,...,b,b+2,...,k}$, i.e. there exists an isotropic subspace $F_{k}$ such that $\Lambda_{k-1}\subset F_k$ for all $\Lambda_{k-1}\in Y$. Since $X_H$ covers $X$ as we varying $H$, for every $\Lambda\in X$, $\Lambda\supset \Lambda_{k-1}$ for some $\Lambda_{k-1}\in Y$. Therefore 
$$\dim(F_k\cap\Lambda)\geq \dim(F_k\cap\Lambda_{k-1})=k-1.$$
We conclude that $a_s=k$ is multi-rigid.

{\bf Case II.} Now assume $n=2k+2$. Notice that the orthogonal complement of $\Lambda\in X$ intersects $Q$ into two maximal isotropic linear spaces belonging to different components. Let $$Y:=\{\Lambda_{k+1}|\mathbb{P}({\Lambda_{k+1}})\subset Q, \Lambda_{k+1}\supset \Lambda\text{ for some }\Lambda\in X\}=Y_1\cup Y_2$$
be a decomposition into irreducible components where $Y_1\subset OG(k+1,2k+2)$ and $Y'\subset OG(k+1,2k+2)'$. The class of $Y_1$ in $OG(k+1,2k+2)$ is given by $m\sigma_{1,...,b,b+2,...,k;k,b}$. By Proposition \ref{4.6}, the sub-index $k$ in the sequence $b_\bullet$ is multi-rigid, i.e. there exists a maximal isotropic subspace $F_{k}^\perp\in OG(k+1,2k+2)'$ such that $\dim(F_{k}^\perp\cap \Lambda_{k+1})\geq k$ for all $\Lambda_{k+1}\in Y_1$. By an involution of $Q$ that interchanges $OG(k+1,2k+2)$ and $OG(k+1,2k+2)'$, a same argument shows the existence of a maximal isotropic subspace $F_{k+1}\in OG(k+1,2k+2)$ such that $\dim(F_{k+1}\cap \Lambda'_{k+1})\geq k$ for all $\Lambda'_{k+1}\in Y_2$. Let $F_k=F_{k+1}\cap F^\perp_k$. By the construction of $Y$, $F_k$ has dimension $k$ and meets every $\Lambda\in X$ in dimension at least $k-1$. We conclude that $a_s=k$ is multi-rigid.
\end{proof}

Now we are able to prove the multi-rigidity of sub-indices in $b_\bullet$. We start with $j=1$:
\begin{Prop}\label{multirigid b1}
Let $\sigma_{a;b}$ be a Schubert class in $OG(k,n)$. Assume $b_1=b_2-1<\frac{n}{2}-2$. If the sub-index $b_1$ is rigid, then it is also multi-rigid.
\end{Prop}
\begin{proof}
Let $X$ be an irreducible representative of $m\sigma_{a;b}$, $m\in\mathbb{Z}^+$. Let $\Phi_X$ be the variety swept out by the projective linear spaces $\mathbb{P}^{k-1}$ parametrized by $X$. By lemma \ref{dim of q}, $\dim(\Phi_X)=n-b_1-2$. We claim that $\Phi_X$ is a sub-quadric of the maximal possible corank $b_1$.

First, $\Phi_X$ is irreducible. Consider the incidence correspondence
$$I:=\{(p,\Lambda)|p\in\Lambda,\Lambda\in X\}\subset \mathbb{P}(V)\times OG(k,n).$$
Let $\pi_i$ be the canonical projections to the $i$-th component, $i=1,2$. Then $\Phi_X=\pi_1(I)$ and $X=\pi_2(I)$. Since $X$ is irreducible and the fibers of $\pi_2|_X$ are isomorphic to $\mathbb{P}^{k-1}$, by the fiber dimension theorem, $I$ is irreducible. Being the image of an irreducible variety, $\Phi_X$ is also irreducible.

Let $p$ be a general point in $\Phi_X$ and define $$X_p:=\{\Lambda\in X|p\in\mathbb{P}(\Lambda)\}.$$ Let $T_p\Phi_X$ be the tangent space of $\Phi_X$ at $p$. Then $$\dim(T_p\Phi_X)=\dim(\Phi_X)=n-b_1-2.$$ Varying $p$ in a dense open subset of $\Phi_X$, we get a collection of tangent spaces. Let $\Gamma$ be the closure of it in $G(n-b_1-1,n)$. We claim that $\Phi_X=\Phi_\Gamma\cap Q$. The class of $X_p$ is given by $m'\sigma_{a';b'}$, where $m'\in\mathbb{Z}^+$, $a'_1=1$, $a'_{i'+1}=a_i'+1$ if $a_{i'}\leq b_1$, $a'_{i'+1}=a'_{i'}$ if $b_1<a_{i'}<\frac{n}{2}$, $b'_{j'}=b_{j'+1}$ if $b_{j'+1}\neq \frac{n}{2}-1$, $1\leq i'\leq s$, $1\leq j'\leq k-s-1$. In particular, $b'_1=b_2=b_1+1$. Notice that for all $\Lambda\in X_p$, $\Lambda$ is contianed in the tangent space $T_p\Phi_X$ at $p$. Therefore $\Phi_{X_p}\subset T_p\Phi_X\cap Q$. Use Lemma \ref{dim of q} again, $$\dim(\Phi_{X_p})=n-b_2-2=n-b_1-3=\dim( T_p\Phi_X\cap Q).$$ Since $T_p\Phi_X\cap Q$ irreducible, we get $\Phi_{X_p}= T_p\Phi_X\cap Q$ and therefore $\Phi_{X}= \Phi_\Gamma\cap Q$. Since $\Phi_X$ is irreducible, $\Phi_\Gamma$ is also irreducible.

Now we show that $\Phi_\Gamma$ is a linear space and therefore $\Phi_X$ is a sub-quadric of $Q$. We compute the class $[\Gamma]=m_\gamma\sigma_\gamma$ in $G(n-b_1-1,n)$. Notice that $\Phi_{X_p}= T_p\Phi_X\cap Q$ is a sub-quadric of $Q$. We claim $\Phi_{X_p}$ has the maximal possible corank $b_1+1$. Recall the following characterization of rigid sub-index $b_1$ is rigid:
\begin{Lem}\cite[Proposition 5.11]{YL}\label{oldrigid}
Let $\sigma_{a;b}$ be a Schubert class for $OG(k,n)$. The sub-index $b_1$ is rigid if and only if either $b_1=0$ or there exists $i$ and $j$ such that $a_i=b_j$ and $x_j>k-j+b_j-\frac{n-1}{2}$, where $x_j:=\#\{i|a_i\leq b_j\}$.
\end{Lem} 
If $b_1=0$ and $b_2=1$, then clearly $\mathbb{P}(\Lambda)\subset p^\perp$ for all $\Lambda\in X_p$ and therefore $\Phi_{X_p}=p^\perp\cap Q$ has corank 1. If $b_1> 0$, then by Lemma \ref{oldrigid}, there exists $i$ and $j$ such that $a_i=b_j$ and $x_j>k-j+b_j-\frac{n-1}{2}$. Notice that the same inequality holds for $\sigma_{a';b'}$ and therefore $b'_1=b_2$ is rigid with respect to the Schubert class $\sigma_{a';b'}$. Suppose for a contradiction that $\Phi_{X_p}$ is not of the maximal corank $b_1+1$, then $\Phi_{X_p}$ is contained in some smooth quadric of dimension less than $Q$. Therefore $X_p$ can be view as a subvariety in $OG(k,n-1)$ and the class $\sigma_{a';b'}$ is the pushforward of some Schubert classes in $OG(k,n-1)$ via the inclusion $OG(k,n-1)\rightarrow OG(k,n)$, which implie $b_2$ is not rigid with respect to the Schubert class $\sigma_{a';b'}$. We reach a contradiction. We conclude that $\Phi_{X_p}$ has corank $b_1+1$. 

Let $S$ be the singular locus of $\Phi_{X_p}$. Let $q$ be another general point. The tangent space $T_{p}\Phi_X$ coincides with $T_{q}\Phi_X$ if and only if $q$ is contained in $S$. Therefore $$\dim(\Gamma)=\dim(\Phi_X)-b_1=n-2b_1-2.$$ 
Also, by the equality $\Phi_{X}= \Phi_\Gamma\cap Q$, we get 
$$\gamma_{n-b_1-1}-1=\dim(\Phi_\Gamma)=\dim(\Phi_X)+1=n-b_1-1.$$
These two conditions force the class of $\Gamma$ to be $m_\gamma\sigma_{1,2,...,b_1,b_1+1,b_1+3,...,n-b_1}$ with some $m_\gamma\in\mathbb{Z}^+$. Since $n\geq 2b_1+3$, $n-b_1-1>b_1+1$ and therefore the last sub-index $n-b_1$ is multi-rigid by Theorem \ref{rigid in g}. Since $\Phi_\Gamma$ is irreducible, we get $\Phi_\Gamma\cong \mathbb{P}^{n-b_1-1}$. Therefore $\Phi_X=\Phi_\Gamma\cap Q$ is a sub-quadric of $Q$. Since $b_1$ is rigid, this sub-quadric must have the maximal possible corank. Let $F_{b_1}$ be the singular locus of $\Phi_X$. Then for every $\Lambda\in X$, $\Lambda\subset F_{b_1}^\perp$.
\end{proof}

More generally we have

\begin{Thm}\label{multirigid of b}
Let $\sigma_{a;b}$ be a Schubert class in $OG(k,n)$. An essential sub-index $b_j$ is multi-rigid if it is rigid and $b_{j-1}+3\leq b_j=b_{j+1}-1<\frac{n}{2}-2$.
\end{Thm}
\begin{proof}
We use induction on $j$. The case of $j=1$ reduces to Proposition \ref{multirigid b1}.

Assume $j\geq2$. Let $X$ be an irreducible representative of $m\sigma_{a;b}$, $m\in\mathbb{Z}^+$. Let $\Phi_X$, $X_p$, $\Phi_{X_P}$ be defined as in Proposition \ref{multirigid b1}. The class of $X_p$ is given by $m'\sigma_{a';b'}$, where $m'\in\mathbb{Z}^+$, $b'_{j'}=b_{j'+1}$ if $b_{j'+1}\neq \frac{n}{2}-1$, $1\leq j'\leq k-s-1$. It is easy to check that the same conditions in the statement also hold for $b'_{j-1}=b_j$ with respect to the Schubert class $\sigma_{a';b'}$. By induction, $b'_{j-1}$ is multi-rigid, i.e. there exists a unique isotropic subspace $F^p_{b_j}$ of dimension $b_j=b'_{j-1}$ such that for all $\Lambda\in X_p$, $\dim(\Lambda\cap (F_{b_j}^p)^\perp)\geq n-b_j+2$. Let $I^\circ$ be the collection of all possible pairs $(p,F_{b_j}^p)$ and let $I$ be the Zariski closure of $I^\circ$ in $\mathbb{P}(V)\times OG(b_j,n)$. Let $\pi_1:I\rightarrow \mathbb{P}(V)$, $\pi_2:I\rightarrow OG(b_j,n)$ be the two projections. $\pi_1(I)=\Phi_X$ and therefore $\dim(\pi_1(I))=n-b_1-2$. A general fiber of $\pi_1|_{\Phi_X}$ has dimension $0$. Hence $\dim(I)=n-b_1-2$. Let $q$ be another general point in $\Phi_X$. $F_{b_j}^q=F_{b_j}^p$ if and only if $q\in F_{b_j}^p$. Therefore a general fiber of $\pi_2|_{\pi_2(I)}$ has dimension $b_j-1$. Thus 
$$\dim(\pi_2(I))=n-b_1-b_j-1.$$
Notice the $\Phi_{\pi_2(I)}=\Phi_X$ and therefore $\dim(\Phi_{\pi_2(I)})=n-b_1-2$. We obtain that $[\pi_2(I)]=m_I\sigma_{1,...,b_1,b_1+2,...,b_j;b_1}$ for some $m_I\in\mathbb{Z}^+$. By Lemma \ref{an}, the sub-index $b_j$ is multi-rigid with respect to the class $[\pi_2(I)]$, i.e. there exists an isotropic subspace $F_{b_j}$ such that $\dim(F_{b_j}^p\cap F_{b_j})\geq b_j-1$ for all $F_{b_j}^p\in\pi_2(I)$. Now it is easy to check that for all $\Lambda\in X$, $\Lambda\in X_p$ for some $p\in\Phi_X$, and 
$$\dim(\Lambda\cap(F^p_{b_j})^\perp)\geq k-b_j+2\Rightarrow \dim(\Lambda\cap F_{b_j}^\perp)\geq k-b_j+1.$$
\end{proof}

\begin{Rem}\label{non-rigidex}
Let $\sigma_{a;b}$ be a Schubert class in $OG(k,n)$. If $b_j\neq a_i$ for all $1\leq i\leq s$ and either $b_j=b_{j-1}+2$ or $b_{j+1}\neq b_j+1$, then $b_j$ is not multi-rigid. We construct a counter-example when one of the conditions fails.

Consider the Schubert class $\sigma_{a_1,...,a_s,n-b_{k-s},...,n-b_1}$ for $G(k,n)$. Then by Theorem \ref{rigid in g}, the sub-index $n-b_j$ is not multi-rigid. Assume $b_j\neq b_{j+1}-1$ (The case when $b_{j-1}=b_j+2$ can be obtained from duality). Take a partial flag of subspaces:
$$F_{a_1}\subset...\subset F_{a_s}\subset F_{n-b_{k-s}}\subset...\subset F_{n-b_{j}-2}\subset F_{n-b_{j}+1}\subset...\subset F_{n-b_1}$$
such that
\begin{enumerate}
\item $F_{a_i}$ are isotropic, $1\leq i\leq s$;
\item $\dim(S_{\gamma})=\gamma-1$, $b_1\leq \gamma\leq b_{k-s}$, $\gamma\neq b_j,b_{j}+1$, where $S_j:=$ singular locus of $\mathbb{P}(F_{n-\gamma})\cap Q$;
\item $\dim(S_{\gamma}\cap\mathbb{P}(F_{a_i}))=\min\{\gamma-1,a_i-1\}$, $b_1\leq \gamma\leq b_{k-s}$, $\gamma\neq b_j,b_{j}+1$.
\end{enumerate}
Take a smooth plane curve $C$ of degree $m$ tangent to $Q$ in $\mathbb{P}(F_{n-b_j+1})$ whose span is disjoint from $\mathbb{P}(F_{n-b_j-2})$. Let $Y$ be the cone over $C$ with vertex $\mathbb{P}(F_{n-b_j-2})$. Let $X$ be the Zariski closure of the following locus
\begin{eqnarray}
X^o:=\{\Lambda\in OG(k,n)&|&\dim(\Lambda\cap F_{a_i})\geq i, 1\leq i\leq s,\nonumber\\
& &\dim(\Lambda\cap F_{n-b_j'})\geq k-j'+1, j\neq j\nonumber\\
& &\dim(\mathbb{P}(\Lambda)\cap Y))\geq k-j \}\nonumber
\end{eqnarray}
Then $X$ is irreducible and has class $m\sigma_{a;b}$ but is not a Schubert variety when $m\geq 2$. 
\end{Rem}

\begin{Ex}
Consider the Schubert class $\sigma_{;1}$ with $s=0$ in $OG(1,n)$, $n\geq 5$. Let $S$ be a smooth hypersurface in $\mathbb{P}(V)$ of degree $m$ that is tangent to $Q$ at a point $p$. Let $T$ be the intersection of $S$ and $Q$. Then the locus of points $\{q\in OG(1,n)|q\in T\cap S\}$ has class $m\sigma_{;1}$.
\end{Ex}

Recall also that Hong \cite{Ho1} proves when $n$ is even, the class of a maximal isotropic subspace on a quadric hypersurface is multi-rigid. With Theorem \ref{orthogonal rigid index}, we prove that when $n$ is even, the sub-index corresponding to a maximal isotropic subspace is multi-rigid if it is essential.
\begin{Prop}\label{rigid of m}
Let $\sigma_{a;b}$ be a Schubert class for $OG(k,n)$, $n$ is even. If $a_s=\frac{n}{2}$ and $b_{k-s}\leq \frac{n}{2}-4$, then $a_s$ is multi-rigid. If $b_{k-s}=\frac{n}{2}-1$ and $b_{k-s}\neq \frac{n}{2}-4$, then $b_{k-s}$ is multi-rigid. 
\end{Prop}
\begin{proof}
Let $X$ be an irreducible representative of $m\sigma_{a;b}$, $m\in\mathbb{Z}^+$. We use induction on 
$$\gamma:=\#\{j|b_j\neq \frac{n}{2}-1\}.$$

{\bf Case A-0}. If $k-s=0$ and $a_s=\frac{n}{2}$, consider the incidence correspondence
$$I:=\{(\Lambda_1,\Lambda)|\Lambda_1\subset\Lambda, \Lambda\in X\}\subset OF(1,k;n).$$
By the fiber dimension theorem, $I$ is irreducible. The class of $I$ in $OF(1,k;n)$ is given by $m\sigma_{a^\alpha;}$, where $\alpha=(2,...,2,1)$. Let $\pi_1$ and $\pi_2$ be the two projections to $OG(1,n)$ and $OG(k,n)$ respectively. Then $\pi_1(I)$ has class $m'\sigma_{\frac{n}{2}}$. By \cite{Ho1}, the Schubert class $\sigma_{\frac{n}{2}}$ is multi-rigid, i.e. there exists an isotropic subspace $F_{\frac{n}{2}}$ such that $\Lambda_1\subset F_{\frac{n}{2}}$ for all $\Lambda_1\subset \pi_1(I)$. Apply Theorem \ref{orthogonal rigid index}, we obtain $\dim(F_{\frac{n}{2}}\cap\Lambda)\geq s$ for all $\Lambda\in X$. Therefore $a_s=\frac{n}{2}$ is multi-rigid. 

{\bf Case B-0}. If $k-s=1$ and $b_{k-s}=\frac{n}{2}-1$, then an identical argument as in {\bf Case A-0} follows, except the class of $I$ in $OF(1,k;n)$ is given by $m\sigma_{a^\alpha;b^\beta}$, where $\alpha=(1,...,1,1)$ and $\beta=(2)$, and the class of $\pi_1(I)$ is $m'\sigma_{;\frac{n}{2}-1}$. By \cite{Ho1}, the Schubert class $\sigma_{;\frac{n}{2}-1}$ is also multi-rigid and the statement then follows from Theorem \ref{orthogonal rigid index}.

{\bf Induction-A}. If $\gamma>0$ and $a_s=\frac{n}{2}$, let $\Phi_X$ be the variety swept out by the projective linear spaces parametrized by $X$ and let $p$ be a general point in $\Phi_X$. Consider the locus $X_p:=\{\Lambda\in X|p\in\mathbb{P}(\Lambda)\}$. Then $X_p$ has class $m'\sigma_{a';b'}$, where $s'=s$, $a'_1=1$, $a'_i=a_{i-1}$ for $2\leq i\leq s$, $b'_j=b_{j+1}$ for $1\leq j\leq k-s-1$ and $b'_{k-s}=\frac{n}{2}-1$. In particular, $\gamma'=k-s-1<\gamma$. By induction, the sub-index $b'_{k-s}=\frac{n}{2}-1$ is multi-rigid. Let $F^{'p}_{\frac{n}{2}}$ be the corresponding maximal isotropic subspace that contained in the different component thatn $OG(\frac{n}{2},n)$ such that $\dim(F^{'p}_{\frac{n}{2}}\cap \Lambda)\geq s+1$ for all $\Lambda\in X_p$. As we vary $p\in \Phi_X$, let $I$ be the Zariski closure of $\{F^{'p}_{\frac{n}{2}}\}$ in $OG(\frac{n}{2},n)'$. Take an involution of the quadric that interchanges $OG(\frac{n}{2},n)$ and $OG(\frac{n}{2},n)'$. Let $I'$ be the image of $I$ under the involution. Then $I'$ has class $m''\sigma_{1,...,b_1,b_1+2,...,\frac{n}{2}-1;\frac{n}{2}-1,b_1}$ for some $m''\in\mathbb{Z}^+$. By assumption, $\frac{n}{2}-1-(b_1+2)\geq 1$. By Theorem \ref{4.6}, the sub-index $\frac{n}{2}-1$ in sequence $b_\bullet$ is multi-rigid with respect to the class $\sigma_{1,...,b_1,b_1+2,...,\frac{n}{2}-1;\frac{n}{2}-1,b_1}$. Therefore there exists $F_{\frac{n}{2}}$, contained in $OG(\frac{n}{2},n)$ before the involution, such that for every $F^{'p}_{\frac{n}{2}}\in I$, $\dim(F_{\frac{n}{2}}\cap F^{'p}_{\frac{n}{2}})\geq \frac{n}{2}-1$. Now for every $\Lambda\in X$, $\Lambda\in X_p$ for some $p\in\Phi$ and therefore
$$\dim(F_{\frac{n}{2}\cap\Lambda})\geq \dim(F^{'p}_{\frac{n}{2}})\cap\Lambda)-1\geq s$$

{\bf Induction-B}. If $\gamma>0$ and $b_{k-s}=\frac{n}{2}-1$, then an identical argument as in {\bf Induction-A} follows, except in this case $X_p$ has class $m'\sigma_{a';b'}$, where $s'=s+2$ and $a'_{s+2}=\frac{n}{2}$, and without taking involutions of the quadric, $I$ has class $m''\sigma_{1,...,b_1,b_1+2,...,\frac{n}{2}-1;\frac{n}{2}-1,b_1}$.
\end{proof}

Now we consider the multi-rigidity of sub-indices in sequence $a_\bullet$. Note that Theorem \ref{multirigid of b} and Theorem \ref{atob} directly imply the following results for maximal orthogonal Grassmannians:
\begin{Cor}\label{spinor}
Let $\sigma_{a;b}$ be a Schubert class for $OG(k,n)$, $n=2k$ or $2k+1$. An essential sub-index $a_i$ is multi-rigid if one of the following conditions hold:
\begin{enumerate}
\item $i<s$ and $a_{i-1}+1=a_i\leq a_{i+1}-3$;
\item $n=2k$, $i=s$, $a_s\leq \frac{n}{2}-3$ and $a_s=a_{s-1}+1$;
\item $n=2k+1$, $i=s$, $a_s\leq \left[\frac{n}{2}\right]-2$ and $a_s=a_{s-1}+1$;
\end{enumerate}
\end{Cor}
\begin{proof}
Notice that when $n=2k$ or $2k+1$, if $a_i< \frac{n}{2}-1$ is essential, then it is necessary that $a_i=b_j$ for some $1\leq j\leq k-s$. Note also that for all the three cases, the conditions imply $\frac{n}{2}-2\neq b_{j+1}-1=b_j\geq b_{j-1}+3$. The statement then follows from Lemma \ref{atob}, Theorem \ref{multirigid of b}.
\end{proof}
\begin{Rem}
When $n=2k$ and $a_s=\frac{n}{2}$, then by Proposition \ref{4.6}, $a_s$ is multi-rigid if $a_{s-1}=\frac{n}{2}-1$ and $a_{s_2}=\frac{n}{2}-2$. 
\end{Rem}

With Theorem \ref{orthogonal rigid index} and Corollary \ref{spinor}, we prove the rigidity of $a_i$ for general $n$:
\begin{Thm}\label{aaa}
Let $\sigma_{a;b}$ be a Schubert class for $OG(k,n)$. An essential sub-index $a_i$ is multi-rigid if one of the following conditions hold:
\begin{enumerate}
\item $i<s$ and $a_{i-1}+1=a_i\leq a_{i+1}-3$;
\item $n$ is even, $i=s$, $a_s\leq \frac{n}{2}-3$ and $a_s=a_{s-1}+1$;
\item $n$ is even, $i=s$, $a_s= \frac{n}{2}-1$, $a_s=a_{s-1}+1$ and $b_{k-s}\neq\frac{n}{2}-1$;
\item $n$ is even, $i=s$, $a_s= \frac{n}{2}-2$, $a_s=a_{s-1}+1$ and $b_{k-s}\neq \frac{n}{2}-1$;
\item $n$ is even, $i=s$, $a_s=\frac{n}{2}$ and $b_{k-s}\leq \frac{n}{2}-4$;
\item $n$ is odd, $i=s$, $a_s\neq\left[\frac{n}{2}\right]-1$ and $a_s=a_{s-1}+1$;
\item $n$ is odd, $i=s$, $a_s=\left[\frac{n}{2}\right]-1$, $a_s=a_{s-1}+1$ and $b_{k-s}\neq\left[\frac{n}{2}\right]-1$.
\end{enumerate}
\end{Thm}
\begin{proof}
Let $X$ be an irreducible representative of $m\sigma_{a;b}$, $m\in\mathbb{Z}^+$.

{\bf Case (1)}. If $n=2k$ or $n=2k+1$, then it reduces to Corollary \ref{spinor}. Assume $n\neq 2k,2k+1$. Consider the orthogonal partial flag variety $OF(k,\left[\frac{n}{2}\right];n)$. Let $\pi_1$ and $\pi_2$ be the two projections to $OG(k,n)$ and $OG(\left[\frac{n}{2}\right],n)$ respectively. Let $Y:=\pi_1^{-1}(X)$ and $X':=\pi_2(\pi_1^{-1}(X))$. Then the class $m'\sigma_{a';b'}$ of $X'$ is obtained from $\sigma_{a;b}$ by completing the sequence $b_\bullet$ to the maximal admissible sequence such that $1\notin a_\bullet-b_\bullet$ and then interchange $\frac{n}{2}$ in $a_\bullet$ or $\frac{n}{2}-1$ in $b_\bullet$ according to the parity of $s$. In particular, $a_i=a'_i$ for $i<s$. By Corollary \ref{spinor}, the sub-index $a'_i$ is multi-rigid with respect to the class $\sigma_{a';b'}$, i.e. there exists an isotropic subspace $F_{a_i}$ such that $\dim(F_{a_i}\cap\Lambda')\geq i$ for all $\Lambda'\in X'$. Apply Theorem \ref{orthogonal rigid index}, $\dim(F_{a_i}\cap\Lambda)\geq i$ for all $\Lambda\in X$. Therefore $a_i$ is multi-rigid.

{\bf Case (2)}. Follows from an identical argument as in {\bf Case (1)}. Notice that in this case either $a'_\bullet=(a_1,...,a_s)$ or $a'_\bullet=(a_1,...,a_s,\frac{n}{2})$ depending on the parity of $s$. Since $\frac{n}{2}-a_s\geq 3$, both cases satisfy the conditions in Corollary \ref{spinor}.

{\bf Case (3)}. We use induction on $k-s$. If $k-s=0$, then $a_s$ is multi-rigid by Theorem \ref{rigid in g}. Assume $j\geq 1$. Let $\Phi_X$ be the variety swept out by the projective linear spaces $\mathbb{P}^{k-1}$ parametrized by $X$. Let $p$ be a general point in $\Phi_X$. Let $X_p:=\{\Lambda\in X|p\in\mathbb{P}(\Lambda)\}$. Notice that by assumption $b_j\leq \frac{n}{2}-4=a_s-3$ for all $1\leq j\leq k-s$. The class of $X_p$ is then given by $m'\sigma_{a';b'}$ where $a'_1=1$, $a'_{i'+1}=a_{i'}+1$ if $a_{i'}\leq b_1$, $a'_{i'+1}=a_{i'}$ if $a_{i'}>b_1$ and $b'_{j}=b_{j+1}$ for $1\leq j\leq k-s-1$. In particular, $a'_{s+1}=a_s$ and $a'_{s}=a_{s-1}$. By induction, for every irreducible component $X_p'$ of $X_p$, there exists an isotropic subspace $F^p_{a_s}$ such that $\dim(F^p_{a_s}\cap\Lambda)\geq s+1$ for all $\Lambda\in X_p'$. Let $I$ be the Zariski closure in $OG(a_s,n)$ of the locus of all possible $F_{a_s}^p$ as varying $p\in\Phi_X$. The class of $I$ is then given by $m''\sigma_{1,...,b_1,b_1+2,...,a_s;b_1}$. Since $n=2a_s+2$ and $a_s-b_1\geq 3$, by Proposition \ref{2k+1}, the sub-index $a_s$ is multi-rigid for the class $\sigma_{1,...,b_1,b_1+2,...,a_s;b_1}$. Therefore there exists an isotropic subspace $F_{a_s}$ such that $\dim(F_{a_s}\cap F_{a_s}^p)\geq a_s-1$ for all $F_{a_s}^p\in I$. It is then straight to check for every $\Lambda\in X$, $\Lambda\in X_p$ for some $p\in\Phi_X$ and
$$\dim(F_{a_s}\cap \Lambda)\geq \dim(F^p_{a_s}\cap\Lambda)-1\geq s.$$

{\bf Case (4)}. An identical argument as in {\bf Case (3)} shows $a_s$ is multi-rigid, unless when $b_1=a_s= \frac{n}{2}-2$, the class of $X_p$ is given by $m'\sigma_{a';b'}$, where $a'_{s}=a_{s-1}+1$ and $a'_{s+1}=a_s+1$, and the class of $I$ is given by $m''\sigma_{1,...,a_s;a_s}$ in $OG(a_s+1,n)$. Notice that $i_*(\sigma_{1,...,a_s;a_s})=\sigma_{1,...,a_s,a_s+3}$ where $i:OG(a_s+1,n)\hookrightarrow G(a_s+1,n)$ is the natural inclusion. By Theorem \ref{rigid in g}, the sub-index $a_s$ is multi-rigid with respect to the class $\sigma_{1,...,a_s,a_s+3}$. Therefore there exists an isotropic subspace $F_{a_s}$ such that $F_{a_s}\subset F_{a_s+1}^p$ for all $F_{a_s+1}^p\in I$. It is then straight to check for every $\Lambda\in X$, $\Lambda\in X_p$ for some $p\in\Phi_X$ and
$$\dim(F_{a_s}\cap \Lambda)\geq \dim(F^p_{a_s+1}\cap\Lambda)-1\geq s.$$

{\bf Case (5)}. Follows from Proposition \ref{rigid of m}.

{\bf Case (6)}. If $a_s\leq \left[\frac{n}{2}\right]-2$, then it follows from an identical argument as in {\bf Case (1)}. If $a_s= \left[\frac{n}{2}\right]$, then $b_j\leq \left[\frac{n}{2}\right]-3$ for all $1\leq j\leq k-s$ and it follows from an identical argument as in {\bf Case (3)}.

{\bf Case (7)}. If $a_i=\left[\frac{n}{2}\right]-1$ and $b_{k-s}\neq\left[\frac{n}{2}\right]-1$, then $b_j\leq \left[\frac{n}{2}\right]-4$ for all $1\leq j\leq k-s$ and it follows from an identical argument as in {\bf Case (3)}.
\end{proof}

Unfortunately Theorem \ref{aaa} is not sharp. This can be seen from the following example:
\begin{Ex}
Consider the Schubert class $\sigma_{2;2}$ in $OG(2,n)$, $n\geq 8$. We claim that the sub-index $2$ is multi-rigid. Let $X$ be an irreducible representative of $m\sigma_{2;2}$, $m\in\mathbb{P}$. Let $i:OG(2,n)\hookrightarrow G(2,n)$ be the inclusion morphism. Then the image $i(X)$ has class $2m\sigma_{2,n-3}$ in $G(2,n)$. By Lemma \ref{degree}, there exists a projective curve $C$ such that $C\cap \mathbb{P}(\Lambda)\neq\emptyset$ for all $\Lambda\in X$. Let $W$ be the projective linear space spanned by $C$. Take a general point $p\in C$, and let $X_p:=\{\Lambda\in X|p\in\mathbb{P}(\Lambda)\}$. Then $X_p$ has class $m'\sigma_{1;2}$ for some $m'\in\mathbb{Z}$. Let $Z$ and $Z_p$ be the variety swept out by projectivel linear spaces parametrized by $X$ and $X_p$ respectively. Then $\dim(Z)=\dim(Z_p)=n-4$. This implies $Z=Z_p$ and $Z_p\subset p^\perp$. Therefore $Z\subset W^\perp$. If $C$ is not linear, then $\dim(W)\geq 2$ and therefore $\dim(W^\perp)\leq n-4= \dim(Z)$. Since $W$ can not be contained in $Q$, we reach a contradiction. We conclude that $C$ is linear and therefore the sub-index $2$ is multi-rigid. 
\end{Ex}

To obtain a complete classification, we first generalize Lemma \ref{degree}. In \cite{YL3}, we obtained the following result:
\begin{Thm}\cite[Theorem 5.6]{YL3}\label{d}
Let $X$ be an irreducible subvariety with class $[X]=\sum_{a\in A}c_a\sigma_a$, $c_a\in \mathbb{Z}^+$. Set $m_i=\max\limits_{a\in A}\{a_i\}$. Define $$A_1:=\{a\in A|a_1=m_1\}$$ and inductively
$$A_i:=\{a\in A_{i-1}|a_i=\max\limits_{a'\in A_{i-1}}\{a_i'\}\},\  2\leq i\leq k.$$
If $\max\limits_{a\in A_i}\{a_i\}=m_i$ and there exists an index $a\in A_i$ such that $a_{i-1}+1=m_i\leq a_{i+1}-3$, then $\gamma_i(X)=m_i$.
\end{Thm}
We release the conditions in Theorem \ref{d} and prove a similar statement as Lemma \ref{degree}.
\begin{Lem}\label{degreegeneral}
Let $X$ be an irreducible subvariety of $G(k,n)$ with class $[X]=\sum_{a\in A}c_a\sigma_a$, $c_a\in \mathbb{Z}^+$. Set $m_i=\max\limits_{a\in A}\{a_i\}$. Define $$A_1:=\{a\in A|a_1=m_1\}$$ and inductively
$$A_i:=\{a\in A_{i-1}|a_i=\max\limits_{a'\in A_{i-1}}\{a_i'\}\},\  2\leq i\leq k.$$
If $\max\limits_{a\in A_i}\{a_i\}=m_i$ and $\max\limits_{a\in A_i}\{a_{i+1}\}\geq m_i+3$, then there exists a projective subvariety $Y$ of $\mathbb{P}(V)$ of dimension $m_i-1$ such that for all $\Lambda\in X$,
$$\dim(Y\cap\mathbb{P}(\Lambda))\geq i-1.$$
\end{Lem}
\begin{proof}
The proof is similar to the proof of Lemma \ref{degree}. We use induction on $m_i-i$. If $m_i=i$, then it reduces to Theorem \ref{d}. 

If $m_i\geq i+1$, then consider the incidence correspondence
$$I:=\{(\Lambda,H)|\Lambda\subset H,\Lambda\in X, H\in G(n-1,n)\}.$$
Set $i_0:=\min\{i'|m_{i'}>i'\}$ and $A':=\{a\in A|a_{i_0}>i_0\}$. For a general hyperplane $H\in\pi_2(I)$, the locus $X_H:=\pi_1(\pi_2^{-1}(H))$ has class $\sum\limits_{a\in A'}c_a\sigma_{a'}$ where $a'_j=a_j$ if $a_j=j$ and $a'_j=a_j-1$ if $a_j\geq j+1$. By induction, there exists an irreducible projective variety $Y_H$ of dimension $m_i-2$ such that $\dim(\mathbb{P}(\Lambda)\cap Y_H)\geq i-1$ for all $\Lambda\in X_H$. As we vary $H$, let $Y$ be the projective variety swept out by $Y_H$. We claim that $\dim(Y)=m_i-1$. 

Suppose, for a contradiction, that $\dim{Y}=m\geq m_i$. Let $G_\bullet$ be a general complete flag. Then $\mathbb{P}(G_{n-m_i})$ will meet $Y$ in finitely many points. By the construction of $Y$, there exists a hyperplane $H$ such that $$\mathbb{P}(G_{n-m_i})\cap Y_H\neq \emptyset.$$ Since $G_{n+2-m_i}$ is a general linear space of dimension $n+2-m_i$ containing $G_{n-m_i}$, we may assume
$$\mathbb{P}(G_{n-m_i})\cap Y_H=\mathbb{P}(G_{n+2-m_i})\cap Y_H.$$ 
Let $a^o\in A_{i+1}$ and let $\sigma_b$ be the dual Schubert class of $\sigma_{({a^{o}})'}$. Let $\Sigma_b(G_\bullet)$ be the corresponding Schubert variety. Since $\sigma_b\cdot [X_H]\neq0$, the intersection of $X_H$ and $\Sigma_b(G_\bullet)$ is non-empty. Let $\Lambda\in X_H\cap\Sigma_b(G_\bullet)$. Then $\dim(\mathbb{P}(\Lambda)\cap Y_H)=i-1$, $\dim(\Lambda\cap G_{n+2-m_{i+1}})=k-i$ and $\dim(\Lambda\cap G_{n+2-m_{i}})=k-i+1$, and therefore $$\mathbb{P}(\Lambda)\cap \mathbb{P}(G_{n+2-m_i})\cap Y_H\neq\emptyset.$$
Since $m_{i+1}\geq a_i+3$, $\mathbb{P}(G_{n+2-m_{i+1}})$ does not meet $Y_H$, 
$$\dim(\Lambda\cap G_{n-m_i})\geq \dim(\Lambda\cap Y_H\cap G_{n+2-m_i})+\dim(\Lambda\cap G_{n+2-m_{i+1}})=k-i+1.$$
Hence we proved that for a general $(n-m_i)$-dimensional vector space $G_{n-m_i}$, there exists a $\Lambda\in X$ that meets $G_{n-m_i}$ in a $(k-i+1)$-dimensional subspace, which is a contradiction since
$$\sum_{a\in A}c_a\sigma_a\cdot \sigma_c=0,$$
where $c_j=n-m_i+j-(k-i+1)$ for $1\leq j\leq k-i+1$ and $c_j=n+j-k$ for $k-i+2\leq j\leq k$. We conclude that $\dim(Y)=m_i-1$ as desired.

\end{proof}

\begin{Prop}\label{-3}
Let $\sigma_{a;b}$ be a Schubert class for $OG(k,n)$. If $a_i\leq a_{i+1}-3$, $a_i=b_j$ for some $j$ and $\frac{n}{2}-2>b_j\geq b_{j-1}+3$, then the sub-index $a_i$ is multi-rigid.
\end{Prop}
\begin{proof}
Let $X$ be an irreducible representative of $m\sigma_{a;b}$, $m\in\mathbb{Z}$. We use induction on $j$. 

First assume $j=1$. If $b_1=b_2-1$, then by Theorem \ref{multirigid b1}, the sub-index $b_1$ is multi-rigid. By Theorem \ref{atob}, the sub-index $a_i=b_1$ is also multi-rigid. Assume $b_2\neq b_1+1$. Let $i:OG(k,n)\hookrightarrow G(k,n)$ be the natural inclusion. Then $i(X)$ has class $\sum_{\gamma\in \Gamma}c_\gamma\sigma_\gamma$ in $G(k,n)$, where $c_\gamma\in\mathbb{Z}$, $m_t=a_t$ for $1\leq t\leq i$ and $(a_1,...,a_i,a_{i+1}',...)\in\Gamma$ where $a_{i+1}'-a_i\geq 3.$. By Lemma \ref{degreegeneral}, there exists an irreducible projective subvariety $Y$ of dimension $a_i-1$ such that $\dim(Y\cap\mathbb{P}(\Lambda))\geq i-1$ for every $\Lambda\in X$. We claim that $Y$ is linear.

Let $\Phi_X$ be the variety swept out by the projective linear spaces parametrized by $X$. By Lemma \ref{dim of q}, $\dim(\Phi_X)=n-b_1-2=n-a_i-2$. Let $p$ be a general point in $Y$. Let $X_p:=\{\Lambda\in X|p\in\mathbb{P}(\Lambda)\}$. Then $X_p$ has class $m'\sigma_{a';b'}$, where $a'_1=1$, $a'_t=a_{t-1}+1$ for $2\leq t<i$, $a'_t=a_t$ for $t>i$, $b'=b$. Let $\Phi_p$ be the variety swept out by the projective linear spaces parametrized by $X_p$. Use Lemma \ref{dim of q} again, $\dim(\Phi_p)=n-a_i-2=\dim(\Phi_X)$. Since $\Phi_X$ is irreducible and $\Phi_p\subset \Phi_X$, we conclude that $\Phi_X=\Phi_p$. Let $T_pQ=p^\perp$ be the tangent space of $Q$ at $p$. Since for every $\Lambda\in X_p$, $\mathbb{P}(\Lambda)\subset T_pQ$, we get $\Phi_X=\Phi_p\subset T_pQ$. On the other hand, let $W$ be the projective linear space spaned by $Y$. From the construction, it is easy to see that $W$ is isotropic. If $Y$ is not linear, then $\dim(W)>\dim(Y)=a_i-1$. Notice that $$\Phi_X\subset \cap_{p\in Y}T_pQ\cap Q=W^\perp\cap Q,$$ while $\dim(\Phi_X)=n-a_i-2\geq n-\dim(W)-2=\dim(W^\perp)>\dim(W^\perp\cap Q)$, we reach a contradiction. We conclude that $Y$ is linear.

Now assume $j\geq 2$. Let $q$ be a general point in $\Phi_X$. Consider the locus $X_q:=\{\Lambda\in X|q\in\mathbb{P}(\Lambda)\}$. $X_p$ has class $\sigma_{a'';b''}$, where $a''_{i+1}=a_i=b_j=b''_{j-1}$ and $a''_{i+2}=a_{i+1}\geq a_i+3$. By induction, the sub-index $a''_{i+1}$ is multi-rigid. Let $F^q_{a''_{i+1}}$ be the corresponding isotropic subspace such that $\dim(F^q_{a''_{i+1}}\cap \Lambda)\geq i+1$ for all $\Lambda\in X_q$. As we vary $q\in Z$, let $I$ be the Zariski closure of $\{F^q_{a''_{i+1}}\}$ in $OG(a_{i},n)$. Then $I$ has class $c\sigma_{1,...,b_1,b_1+2,...,a_i;b_1}$ for some $c\in\mathbb{Z}^+$. By assumption, $a_i-(b_1+2)\geq 1$. By Theorem \ref{an}, the sub-index $a_i$ is multi-rigid. Let $G_{a_i}$ be the corresponding linear space. Then it is easy to check for every $\Lambda\in X$, $\dim(G_{a_i}\cap \Lambda)\geq i$.
\end{proof}
\begin{Cor}\label{cooo}
Let $\sigma_{a;b}$ be a Schubert class for $OG(k,n)$. If $\frac{n}{2}-2>b_j\geq b_{j-1}+3$ and $b_j=a_i$ for some $i$ such that $a_i\leq a_{i+1}-3$, then the sub-index $b_j$ is multi-rigid.
\end{Cor}
\begin{proof}
Follows directly from Proposition \ref{-3} and Lemma \ref{atob}.
\end{proof}

We thus obtain the following characterization of multi-rigid sub-indices in $OG(k,n)$:
\begin{Thm}\label{rigidb}
Let $\sigma_{a;b}$ be a Schubert class for $OG(k,n)$. An essential sub-index $b_j$ is multi-rigid if $b_j$ is rigid and one of the following conditions hold:
\begin{enumerate}
\item $b_j<\frac{n}{2}-2$, $b_j\neq a_i$ for all $1\leq i\leq s$ and $b_{j+1}-1=b_j\leq b_{j-1}+3$;
\item $b_j<\frac{n}{2}-2$, $b_j=a_i$ for some $1\leq i\leq s$, $a_i\leq a_{i+1}-3$ and $b_j\leq b_{j-1}+3$;
\item $n$ is even, $j=k-s$, $b_{k-s}=a_s=\frac{n}{2}-2$ and $a_{s-1}=\frac{n}{2}-3$;
\item $n$ is even, $j=k-s$, $b_{k-s}=\frac{n}{2}-1$ and $b_{k-s-1}\leq\frac{n}{2}-4$;
\end{enumerate}
\end{Thm}
\begin{proof}
If condition (1) holds, then the statement follows from Theorem \ref{multirigid of b}. If condition (2) holds, then the statement follows from Corollary \ref{cooo}. If condition (3) holds, then the statement follows from Theorem \ref{aaa} and Lemma \ref{atob}. If condition (4) holds, then the statement follows from Proposition \ref{rigid of m}.
\end{proof}
\begin{Rem}\label{sharp}
If $n$ is odd, then Theorem \ref{rigidb} is also sharp.
Conversely, if $b_j$ is not rigid, then it can not be multi-rigid by definition. 

{\bf Case I}. If $b_j<\frac{n}{2}-2$, $b_j\neq a_i$ for all $1\leq i\leq s$ and either $b_{j+1}-b_j\geq 2$ or $b_j\leq b_{j-1}+3$, then the construction in Remark \ref{non-rigidex} gives a counter-example.

{\bf Case II-1}. If $b_j<\frac{n}{2}-2$, $b_j=a_i$ for some $1\leq i\leq s$ and $b_j=b_{j-1}+2$, then it is necessary that $a_{i-1}\leq a_i-2$. Take a partial flag of isotropic subspaces:
$$F_1\subset ... F_{a_i-2}\subset F_{a_i+1}\subset...\subset F_{\left[\frac{n}{2}\right]}.$$
Take a smooth plane curve $C$ of degree $m$ in $\mathbb{P}(F_{a_i+1})$ whose span disjoint from $\mathbb{P}(F_{a_i-2})$. Let $Z$ be the cone over $C$ with vertex $\mathbb{P}(F_{a_i+1})$. Let $X$ be the Zariski closure of the following locus in $OG(k,n)$:
\begin{eqnarray}
X^o:=\{\Lambda\in OG(k,n)&|&\dim(\Lambda\cap F_{a_{i'}})= i', 1\leq i'\leq s, i'\neq i;\nonumber\\
& &\dim(\Lambda\cap F_{b_{j'}}^\perp)= k-j'+1, 1\leq j'\leq k-s, j'\neq j;\nonumber\\
& &\dim(\mathbb{P}(\Lambda)\cap Z)= i-1;\nonumber\\
& &\dim(\mathbb{P}(\Lambda)\cap\mathbb{P}( F_{a_i-2}))=i-2\}\nonumber
\end{eqnarray}
To compute the class of $X$, we specialize $C$ to a union of lines which specializes $Z$ to a union of projective linear spaces $\mathbb{P}(F_{a_i})$. For every $\Lambda\in X^o$, there exists a point $p\in\mathbb{P}(\Lambda)\cap Z$ outside $\mathbb{P}(F_{a_i-2})$. Let $q\in \mathbb{P}(F_{a_i})$ be a point that does not contained in the span of $\mathbb{P}(F_{a_i-2})$ and $p$. Then
$$\dim(\mathbb{P}(\Lambda)\cap \mathbb{P}(F_{a_i}^\perp))=\dim((\mathbb{P}(\Lambda)\cap \mathbb{P}(F_{a_i-2}^\perp))\cap q^\perp)=k-j+1.$$
Therefore $X$ has class $m\sigma_{a;b}$, while $X$ is not a Schubert variety if $m\geq 2$.

{\bf Case II-2}. If $b_j<\frac{n}{2}-2$, $b_j=a_i$ for some $1\leq i\leq s$ and $a_i=a_{i+1}-2$, take a partial flag of isotropic subspaces:
$$F_1\subset ... F_{a_i-1}\subset F_{a_i+2}\subset...\subset F_{\left[\frac{n}{2}\right]}.$$
Take a smooth plane curve $C$ of degree $m$ in $\mathbb{P}(F_{a_i-1}^\perp)$ which is tangent to $Q$ and disjoint from $\mathbb{P}(F_{a_i+2}^\perp)$. Let $Z$ be the cone over $C$ with vertex $\mathbb{P}(F_{a_i+2}^\perp)$. Then $Z$ is tangent to $Q$ along the intersection with $Q$. Let $X$ be the Zariski closure of the following locus in $OG(k,n)$:
\begin{eqnarray}
X^o:=\{\Lambda\in OG(k,n)&|&\dim(\Lambda\cap F_{a_{i'}})= i', 1\leq i'\leq s, i'\neq i;\nonumber\\
& &\dim(\Lambda\cap F_{b_{j'}}^\perp)= k-j'+1, 1\leq j'\leq k-s, j'\neq j;\nonumber\\
& &\dim(\mathbb{P}(\Lambda)\cap Z)= k-j;\nonumber\\
& &\dim(\mathbb{P}(\Lambda)\cap\mathbb{P}( F_{a_i+2}^\perp))=k-j-1\}\nonumber
\end{eqnarray}
By specializing $C$ to a union of lines tangent to $Q$, one obtains that $[X]=m\sigma_{a;b}$.

{\bf Case III}. If $b_j=\frac{n-1}{2}-1$, then $b_j$ is not even rigid by Lemma \ref{oldrigid}. If $a_s\neq\frac{n-1}{2}-1$, then consider the restriction variety $X$ defined by the following sequence:
$$F_{a_1}\subset...\subset F_{a_s}\subset Q_{\frac{n+3}{2}}^{\frac{n-1}{2}-1}\subset Q_{n-b_{j-1}}^{b_{j-1}}\subset...\subset Q_{n-b_1}^{b_1}.$$
Note that this sequence differs from the one defining Schubert varieties in that the quadric in the $s+1$-th position has corank 1 less than the maximal possible corank. Then $X$ has class $\sigma_{a;b}$ but is not a Schubert variety.

If $a_s=\frac{n-1}{2}-1$, then consider the restriction variety $Y$ defined by the following sequence:
$$F_{a_1}\subset...\subset F_{a_s+1}\subset Q_{\frac{n+3}{2}}^{\frac{n-1}{2}-1}\subset Q_{n-b_{j-1}}^{b_{j-1}}\subset...\subset Q_{n-b_1}^{b_1}.$$
Note that this sequence differs from the one defining Schubert varieties in that the $s$-th flag element has dimenison one greater and the quadric in the $s+1$-th position has corank 1 less. Then $X$ has class $\sigma_{a;b}$ but is not a Schubert variety.
\end{Rem}

\begin{Thm}\label{aaaaa}
Let $\sigma_{a;b}$ be a Schubert class for $OG(k,n)$. An essential sub-index $a_i$ is multi-rigid if $a_i$ is rigid and one of the following conditions hold:
\begin{enumerate}
\item $i<s$ and $a_{i-1}+1=a_i\leq a_{i+1}-3$;
\item $i<s$, $a_i=b_j$ for some $1\leq j\leq k-s$, $a_i\leq a_{i+1}-3$ and $b_j\leq b_{j-1}+3$.
\item $n$ is even, $i=s$, $a_s\leq \frac{n}{2}-3$ and $a_s=a_{s-1}+1$;
\item $n$ is even, $i=s$, $a_s= \frac{n}{2}-1$, $a_s=a_{s-1}+1$ and $b_{k-s}\neq\frac{n}{2}-1$;
\item $n$ is even, $i=s$, $a_s= \frac{n}{2}-2$, $a_s=a_{s-1}+1$ and $b_{k-s}\neq \frac{n}{2}-1$;
\item $n$ is even, $i=s$, $a_s=\frac{n}{2}$ and $b_{k-s}\leq \frac{n}{2}-4$;
\item $n$ is odd, $i=s$, $a_s\neq\left[\frac{n}{2}\right]-1$ and $a_s=a_{s-1}+1$;
\item $n$ is odd, $i=s$, $a_s=\left[\frac{n}{2}\right]-1$, $a_s=a_{s-1}+1$ and $b_{k-s}\neq\left[\frac{n}{2}\right]-1$.
\end{enumerate}
\end{Thm}
\begin{proof}
Condition (1), (3)-(8) follows from Theorem \ref{aaa}. Condition (2) follows from Proposition \ref{-3}.
\end{proof}
\begin{Rem}
When $n$ is odd, then Theorem \ref{aaaaa} is also sharp.

{\bf Case I}. If $i<s$, $a_i\neq b_j$ for all $1\leq j\leq k-s$ and either $a_{i-1}\leq a_i+2$ or $a_{i+1}-a_i=2$, take a partial flag of isotropic subspaces:
$$F_1\subset...\subset F_{a_i-1}\subset F_{a_i+1}\subset...\subset F_{\left[\frac{n}{2}\right]}.$$
Since the sub-index $a_i$ is not multi-rigid with respect to the Schubert class $\sigma_{a_1,...,a_i,a_{i+1}}$ by Theorem \ref{rigid in g}, there exists an irreducible subvariety $Y$ in $G(i+1,F_{a_{i+1}})$ with class $m\sigma_{a_1,...,a_i,a_{i+1}}$ such that $\dim(F_{a_{i'}}\cap\Lambda')\geq i'$ for all $1\leq i'<i$ and $\Lambda'\in Y$, but is not a Schubert variety. Let $X$ be the Zariski closure of the following locus:
$$X^o:=\{\Lambda\in OG(k,n)|\Lambda \cap F_{a_{i+1}}\in Y,\dim(\Lambda\cap F_{a_{i'}})=i',i'\geq i+1,\dim(\Lambda\cap F_{b_j}^\perp)\geq k-j+1\}.$$
Then $X$ has class $m\sigma_{a;b}$ but is not a Schubert variety.

{\bf Case II}. If $i<s$, $a_i=b_j$ for some $1\leq j\leq k-s$ and either $a_i=a_{i+1}-2$ or $b_j=b_{j-1}+2$, then the construction in Remark \ref{sharp} gives counter-examples.

{\bf Case III}. If $i=s$ and $a_s\leq \left[\frac{n}{2}\right]-2$ and $a_s\neq a_{s-1}+1$, then same construction as in {\bf Case I} works. If $a_s=\left[\frac{n}{2}\right]$, take an isotropic subspace $F_{\left[\frac{n}{2}\right]-1}$ of dimension $\left[\frac{n}{2}\right]-1$. Let $Q'=\mathbb{P}(F_{\left[\frac{n}{2}\right]-1}^\perp)\cap Q$ and let $Z$ be a general hyperplane section of $Q'$. By Bertini's theorem, $Z$ is irreducible if $Q'$ has dimension at least two. Let $F_{\left[\frac{n}{2}\right]-2}$ be the singular locus of $Z$. Pick a complete flag in $F_{\left[\frac{n}{2}\right]-2}$:
$$F_1\subset...\subset F_{\left[\frac{n}{2}\right]-2}$$
Let $X$ be the Zariski closure of the following locus
\begin{eqnarray}
\{\Lambda\in OG(k,n)&|&\dim(\mathbb{P}(\Lambda)\cap Z)\geq s-1, \dim(\Lambda\cap F_{a_i})\geq i,1\leq i\leq s-1,\nonumber\\
& &\dim(\Lambda\cap F^\perp_{b_j})\geq k-j+1,1\leq j\leq k-s\}\nonumber
\end{eqnarray}
Then the class of $X$ is $2\sigma_{a;b}$ but is not a Schubert variety.

{\bf Case IV}. If $i=s$, $a_s= \left[\frac{n}{2}\right]-1$ and $a_s=b_{k-s}$, then the sub-index $a_s$ is not even rigid. (See Remark \ref{sharp} Case III). If $a_s\neq b_{k-s}$, take an isotropic subspace $F_{\left[\frac{n}{2}\right]-3}$ contained in a maximal isotropic subspace $F_{\left[\frac{n}{2}\right]}$. Let $C$ be a plane curve in $\mathbb{P}(F_{\left[\frac{n}{2}\right]})$ of degree $m$ whose span is disjoint from $\mathbb{P}(F_{\left[\frac{n}{2}\right]-3})$. Let $Z$ be the cone over $C$ with vertex $\mathbb{P}(F_{\left[\frac{n}{2}\right]-3})$. Take a complete flag of isotropic subspaces contained in $F_{\left[\frac{n}{2}\right]-3}$:
$$F_1\subset...\subset F_{\left[\frac{n}{2}\right]-3}.$$
Let $X$ be the Zariski closure of the following locus
\begin{eqnarray}
\{\Lambda\in OG(k,n)&|&\dim(\mathbb{P}(\Lambda)\cap Z)\geq s-1, \dim(\Lambda\cap F_{a_i})\geq i,1\leq i\leq s-1,\nonumber\\
& &\dim(\Lambda\cap F^\perp_{b_j})\geq k-j+1,1\leq j\leq k-s\}\nonumber
\end{eqnarray}
Then the class of $X$ is $m\sigma_{a;b}$ but is not a Schubert variety.
\end{Rem}

We hereby obtain the multi-rigidity of Schubert classes in orthogonal Grassmannians:
\begin{Cor}
Let $\sigma_{a;b}$ be a Schubert class for $OG(k;n)$. The Schubert class $\sigma_{a;b}$ is multi-rigid if all essential $a_i$ satisfy one of the conditions in Theorem \ref{aaaaa} and all essential $b_j$ satisfy one of the conditions in Theorem \ref{rigidb}. If $n$ is odd, then the converse is also true.
\end{Cor}

\section{General rational homogeneous spaces}\label{general case}
In this section, we prove the rigidity of Schubert classes in general rational homogeneous spaces. As a corollary, we also prove the multi-rigidity in these cases.
\subsection{Notations}Let $G$ be a connected, simply-connected, semi-simple algebraic group over an algebraically closed field $k$. Fix a maximal torus $T$ of $G$. Let $W=N(T)/T$ be the Weyl group. Let $R$ be the root system of $G$ relative to $T$.

Let $B$ be a Borel subgroup containing $T$ and let $\Phi\subset R$ be the set of simple roots relative to $B$. Let $U$ be the unipotent radical of $B$. Then $T=B/U$. 

Let $X(T)$ be the character group of $T$, i.e. $X(T):=\text{Hom}(T,\mathbb{G}_m)$, where $\mathbb{G}_m$ is the multiplication group of $k$. Let $\lambda\in X(T)$ be a character of $T$. Then $\lambda$ induces a one-dimensional representation $k_{\lambda}$ of $B$ via the pullback of the natural projection $B\rightarrow T=B/U$. Consider $G\rightarrow G/B$ as a principal $B$-bundle and let $L(\lambda):=G\times_Bk_{-\lambda}$ be the associated line bundle on $G/B$. Then any line bundle on $G/B$ is isomorphic to $L(\lambda)$ for some $\lambda\in X(T)$. 

The space of global sections $H^0(G/B,L(\lambda))$ is nonzero if and only if $\Lambda$ is dominant. The dual module $V(\lambda):=H^0(G/B,L(\lambda))^*$ is called the Weyl module of highest weight $\lambda$.

Let $\mathfrak{g},\mathfrak{h}$ be the Lie algebra of $G$ and $T$ respectively. Then $\mathfrak{g}$ can be decomposed as 
$$\mathfrak{g}=\mathfrak{h}\oplus\oplus_{\alpha\in R}\mathfrak{g}_\alpha,$$
where $\mathfrak{g_\alpha}$ is the root space of $\alpha$. For each $\beta\in R$, there exists a unique one-dimensional unipotent subgroup $U_\beta$ of $G$ with $\mathfrak{g_\beta}$ as its Lie algebra. Such $U_\beta$ is called a root subgroup. The Borel subgroup $B$ is the semidirect product of $U$ and $T$, and 
$$U\cong\Pi_{\beta\in R^+}U_\beta.$$

Let $P$ be a parabolic subgroup containing $B$. There is a bijection between the set of parabolic subgroups containing $B$ and the power set of $\Phi$. For a subset $J\subset B$, let $P_J$ be the parabolic subgroup generated by $B$ and $\{U_{-\alpha}|\alpha=\sum_{\beta\in J}a_\beta\beta,a_\beta\geq 0\}$. Let $H_\alpha$ be the coroot of $\alpha$. The character group $X(P)$ of $P_J$ is a subgroup of $X(T)$ consisting of $\lambda$ such that $\lambda(H_\alpha)=0$ for all $\alpha\in J$. For $\lambda\in X(P)$, we can construct similarly the line bundle $L^P(\lambda):=G\times_Pk_{-\lambda}$ associated to the principal $P$-bundle $G\rightarrow G/P$, and any line bundle on $G/P$ occur in this way. Note that $H^0(G/P,L^P(\lambda))\neq0$ if and only if $\lambda$ is dominant.

\begin{Lem}\cite[Theorem 3.3.4]{BK2005}\label{lemma5.1}
Let $\pi:G/B\rightarrow G/P$ be the natural projection. Let $\lambda\in X(P)$ be a dominant weight. Then the pull-back 
$$\pi^*:H^0(G/P,L^P(\lambda))\rightarrow H^0(G/B,L(\lambda))$$
is an isomorphism.
\end{Lem}

Let $\lambda$ be a dominant weight. The space $H^0(G/B,L(\lambda))$ has $-\lambda$ as its lowest weight and the corresponding weight space is one-dimensional. Let $e_\lambda$ be a weight vector of $-\lambda$. For $w\in W$, fix a representative $n_w$ in $N(T)$. Set $p_w:=n_w\cdot e_\lambda$. Then $p_w$ has weight $-w(\lambda)$. Write $\lambda=\sum a_i\omega_i$, where $\omega_i$ are fundamental weights. Let $P_\lambda$ be the parabolic subgroup with the simple roots $S\setminus\{\alpha_i|a_i\neq0\}$. The following lemma characterizes the Schubert varieties in $G/P_\lambda$:

\begin{Lem}\cite[Lemma 2.11.14]{BL}\label{lemma5.2}
For $\omega,\eta\in W^{P_\lambda}$, 
$$p_\eta|_{\Sigma_{\omega,P_{\lambda}}}\neq 0\Leftrightarrow \omega\geq \eta$$
\end{Lem}

\subsection{The rigidity problem}
The rational homogeneous space $G/P$ with $G$ semi-simple can be decomposed into a direct product 
$$G/P\cong G_1/P_1\times...\times G_k/P_k,$$
where $G\cong G_1\times ...\times G_k$, $G_i$ are simple and $P\cong P_1\times...\times P_k$, $P_i$ are parabolic subgroups of $G_i$. Let $\pi_i:G/P\rightarrow G_i/P_i$ be the natural projections. The rigidity (resp. multi-rigidity) of Schubert classes in $G/P$ can be deduced from the rigidity (resp. multi-rigidity) of Schubert classes in $G_i/P_i$:
\begin{Thm}
A Schubert class $\sigma_{w,P}$ in $G/P$ is rigid (resp. multi-rigid) if $(\pi_i)_*(\sigma_{w,P})$ are rigid (resp. multirigid) for all $1\leq i\leq k$.
\end{Thm}
\begin{proof}
Notice that $W^P\cong W^{P_1}\times...\times W^{P_k}$. Let $\theta_i:W^P\rightarrow W^{P_i}$ be the natural projections. Set $w_i:=\theta_i(w)$. Then $(\pi_i)_*(\sigma_{w})=\sigma_{w_i}$. 

First assume $\sigma_{w_i}$ are rigid for all $1\leq i\leq k$. Let $X$ be a representative of $\sigma_{w}$. Since the Schubert classes are indecomposable, $X$ is irreducible. Let $X_i:=\pi_i(X)$ be the image of $X$ under the $i$-th projection. By assumption, $X_i$ is a translate of the Schubert variety $\Sigma_{w_i}$, i.e. $X_i=g_i\Sigma_{w_i}$ for some $g_i\in G_i$, $1\leq i\leq k$. Let $g=(g_1,...,g_k)\in G$. Notice that $\Sigma_{w,P}=\cap_{i=1}^k\pi_i^{-1}(\Sigma_{w_i})$, $$X\subset \cap_{i=1}^k\pi_i^{-1}(X_i)=\cap_{i=1}^k\pi_i^{-1}(g_i\Sigma_{w_i})=g\Sigma_{w}.$$
Since Schubert varieties are irreducible and $X$ and $g\Sigma_{w}$ share the same cohomology class, in particular $\dim(X)=\dim(g\Sigma_{w})$, we conclude that $X=g\Sigma_{w}$ and therefore the Schubert class $\sigma_{w}$ is rigid.

Now assume $\sigma_{w_i}$ are multi-rigid for all $1\leq i\leq k$. Let $X$ be an irreducible representative of $m\sigma_{w}$, $m\in\mathbb{Z}^+$. Then $X_i=\pi_i(X)$ are irreducible with class $m_i\sigma_{w_i}$ for some $m_i\in\mathbb{Z}^+$. Since $\sigma_{w_i}$ are multi-rigid, $m_i=1$ and $X_i$ are translates of the Schubert varieties $\Sigma_{w_i}$. A similar argument shows that $X$ is also Schubert and therefore we conclude that the Schubert class $\sigma_{w}$ is multi-rigid.
\end{proof}
Let $P_m$ be a maximal parabolic subgroup containing $P$. There is a natural projection $$\pi:G/P\rightarrow G/P_m.$$ Inspired by the classical cases, we investigate the multi-rigidity of Schubert classes in $G/P$ by pushing it to $G/P_m$. 

We illustrate the idea for the type A cases. Choose a basis $\{e_1,...,e_n\}$ for $V$, or equivalently choose a complete flag $F_1\subset...\subset F_n$ in $V$, $F_i=<e_1,...,e_i>$. Let $G$ be the general linear group $GL_n(\mathbb{C})$. Let $B$ be the Borel subgroup of upper triangular matrices. There is a bijection between $G/B$ and the set of complete flags in $V$: 
$$\bar{g}=(v_1...v_n)\leftrightarrow <v_1>\subset<v_1,v_2>\subset...\subset <v_1,...,v_n>.$$
For $g\in G$, the right multiplication by $B$ adds some multiple of the $i$-th column of $g$ to $j$-th column for $i\leq j$. Therefore the coset $gB$ has a canonical representative in echelon form: the lowest non-zero entry in each column is 1 and the entries to the right of each leading one are all zeros.

Let $P$ be a parabolic subgroup that consists of upper triangular block matrices of block size $(d_1,d_2-d_1,...,d_k-d_{k-1},n-d_k)$. For two indices $1\leq i\neq j\leq n$, we say that they are of the same block if $d_s+1\leq i,j\leq d_{s+1}$ for some $0\leq s\leq k$ (Here we set $d_0=0$ and $d_{k+1}=n$). For $g\in G$, the right multiplication by $P$ is a composition of column operations from $B$ and switching two columns belonging to the same block. The coset $gP$ has a canonical representative which is in echelon form and such that if $i<j$ belong to the same block then the lowest non-zero entry in $i$-th column is in the northwest of the lowest non-zero entry in $j$-th column.

\begin{Ex}
Let $g=\begin{bmatrix}
    9 & 5 & 9 & 7  \\
    6 & 2 & 4 & 0 \\
    3 & 1 & 0 & 0 \\
    0 & 0 & 2 & 2
\end{bmatrix}$. The canonical representative of $gB$ is 
$\begin{bmatrix}
    3 & 1 & 0 & 0  \\
    2 & 0 & 2 & 1 \\
    1 & 0 & 0 & 0 \\
    0 & 0 & 1 & 0
\end{bmatrix}.$ The canonical representative of $gP$, where $P$ is of block size $(2,2)$, is $\begin{bmatrix}
    1 & 0 & 0 & 0  \\
    0 & 2 & 1 & 0 \\
    0 & 1 & 0 & 0 \\
    0 & 0 & 0 & 1
\end{bmatrix}.$ 
\end{Ex}
We read the row number of the lowest non-zero entry in each column. For example, $(1 3 2 4)$ for $g=\begin{bmatrix}
    1 & 0 & 0 & 0  \\
    0 & 2 & 1 & 0 \\
    0 & 1 & 0 & 0 \\
    0 & 0 & 0 & 1
\end{bmatrix}.$ It is an element in the Weyl group. The Schubert cell $C_{1324}$ is the orbit of $g$ under the left multiplication by $B$. 
$$C_{1324}=\left\{\begin{bmatrix}
    1 & 0 & 0 & 0  \\
    0 & * & 1 & 0 \\
    0 & 1 & 0 & 0 \\
    0 & 0 & 0 & 1
\end{bmatrix}\right\}.$$
Let $\Lambda$ be the subspace spanned by the first two column vectors, then $C_{1324}$ can also be described as
$$C_{1324}=\{\Lambda\in G(2,4)|F_1\subset\Lambda\subset F_3\}.$$

Let $w=(w_1...w_n)\in W_P$. For each $1\leq i\leq n$, we define $w^i$ successively as follows: for $w_1$, if $w_1<w_i$, then replace $w_1$ by $w_i-1$; if $w_1>w_i$, then replace $w_1$ by $n$. Inductively for $w_j$, if $w_j<w_i$, then replace $w_j$ by the largest number less than $w_i$ excluding $w_1,...,w_{j-1}$. If $w_j>w_i$, then replace $w_j$ by the largest number less than $n$ excluding $w_1,...,w_{j-1}$. After replacing $w_n$, rearrange $w_1...w_n$ so that $w_i<w_j$ if $i$ and $j$ belong to the same block corresponding to $P$. We call $w_i$ rigid if for every representative $X$ of the Schubert class $\sigma_w$, $X$ is contained in the Schubert variety $\Sigma_{w^i}$, up to a general translate. One can easily check this definition coincides with the previous definition in Section 3.

Let $P^m_t$ be the parabolic subgroup consisting of upper triangular block matrices of size $(d_t,n-d_t)$, $1\leq t\leq k$. The homogeneous space $G/P^m_t$ is isomorphic to the Grassmannian $G(d_t,n)$. The subgroup $P^m_t$ is a maximal parabolic subgroup containing $P$. There is a natural projection between the homogeneous spaces
$$\pi_t:G/P\rightarrow G/P^m_t$$
and a projection between the Weyl groups
$$\eta_t:W_P\rightarrow W_{P^m_t}.$$
Let $w=(w_1...w_n)\in W_P$. Then $\eta_t(w)$ is obatined from $w$ by re-arranging $w_1...w_{d_t}$ and $w_{d_t+1}...w_n$ in increasing order. 

Assume $w_i$ is rigid for $\eta_t(w)$. Let $X$ be a representative of $\sigma_{w}$. Then there exists a basis such that $\pi_t(X)$ is contained in the Schubert variety $\Sigma_{\eta_t(w)^i}$. We claim that $X\subset \Sigma_{w^i}$ and therefore $w_i$ is rigid for $w$. Set $s:=\#\{j|w_j\leq w_i,j\leq d_t\}$. Then the first $d_t$ entries of $\eta_t(w)^i$ is given by $w_i-s+1\ ...\ w_i-1\ w_i\ n-(d_t-s)+1\ ...\ n-1\ n$. Let $\Lambda$ be a general point in $X$. Let $g$ be the canonical form of $\Lambda$ in $G/P$ and let $g_t$ be the canonical form in $G/P_t^m$. Then the entry of $g_t$ at $(w_i,s)$-place equals $1$ and the entries below $(w_i,s)$ are all zeros. This implies the entry of $g$ at $(w_i,i)$-place equals $1$ and the entries below $(w_i,i)$ are all zeros. Therefore $\Lambda\in \Sigma_{w^i}$.

To prove a general statament, we need the following lemmas:
\begin{Lem}\cite[Theorem 3.1.2]{BK2005}\label{lemma5.6}
For dominant $\lambda,\mu\in X(P)$, the product map
$$H^0(G/P,L^p(\lambda))\times H^0(G/P,L^p(\mu))\rightarrow H^0(G/P,L^p(\lambda+\mu))$$
is surjective.
\end{Lem}
Let $P$ be a parabolic subgroup and let $I_P$ be the set of simple roots relative to $P$. For $\alpha_i\in \Phi$, let $P_i$ be the maximal parabolic subgroup with the set of simple roots $\Phi\setminus\{\alpha_i\}$. Then $P_i$ contains $P$ if and only if $\alpha_i\notin I_P$. Let $\pi_i:G/P\rightarrow G/P_i$ be the natural projections. 
\begin{Cor}\label{maximalP}
Let $X$ be a representative of $\sigma_{\mu}$ in $G/P$. If $\pi_i(X)$ are $B$-stable Schubert varieties for all $i$ such that $\alpha_i\notin I_P$, then $X$ coincides with the Schubert variety $\Sigma_{\mu}$.
\end{Cor}
\begin{proof}
Let $\sigma_{\mu_i}$ be the class of $\pi_i(X)$ in $G/P_i$. Set $\lambda_i=\omega_i$ and $\lambda=\sum_{\alpha_i\notin I_P} \omega_i$, where $\omega_i$ are fundamental weights. Then $P_i=P_{\lambda_i}$ and $P=P_{\lambda}$, where $P_{\lambda_i}$ and $P_{\lambda}$ are defined as before. By Lemma \ref{lemma5.1} and Lemma \ref{lemma5.6}, the map
$$\otimes_{\lambda_i\notin S_P}H^0(G/P_i,L^{P_i}(\lambda_i))\rightarrow H^0(G/P,L^p(\lambda))$$
is surjective. Let $\nu\in W^P$, let $\nu_i$ be the image of $\nu$ under the projection $W^P\rightarrow W^{P_i}$. By Lemma \ref{lemma5.2}, we have
\begin{eqnarray}
p_\nu|_X=0&\Leftrightarrow& p_{\nu_i}|_{\pi_i(X)}=0\text{ for some }i\nonumber\\
&\Leftrightarrow&\nu_i\not\leq \mu_i\text{ for some i}\nonumber\\
&\Leftrightarrow&\nu\not\leq\mu\nonumber
\end{eqnarray}
Therefore $X$ concide with the Schubert variety $\Sigma_{\mu}$.
\end{proof}

Let $\sigma_{w}$ be a Schubert class for $G/P$. Assume $\sigma_{w_i}$ are rigid for all $i$ such that $\alpha_i\notin I_P$. Then $\pi_i(X)$ are translate of Schubert varieties, i.e. $\pi_i(X)=g_i\Sigma_{w_i}$ for some $g_i\in G$. The following lemma characterizes the stabilizer of a Schubert variety:
\begin{Lem}\cite[Proposition 2.1]{HM2}
Let $R_{P^-}:=R^-\cup\{\alpha|\alpha=\sum_{\beta\in I_P}a_\beta\beta\}$. Let $w\in W^P$. The stabilizer of the Schubert variety $\Sigma_{w}$ is the parabolic subgroup $P_{I_w}$ relative to the set of simple roots $I_w:=\Phi\cap w(R_{P^-})$.
\end{Lem}

\begin{Def}
Let $w\in W^P$. The set $I_w:=\Phi\cap w(R_{P^-})$ is called the set of simple roots associated to $w$.
\end{Def}

More generally,
\begin{Lem}\label{translate of Schubert varieties}
Assume $\Sigma_{w_1}\subset\Sigma_{w_2}$ in $G/P$, $w_1,w_2\in W^P$. Let $g\in G$. If $g\Sigma_{w_1}\subset \Sigma_{w_2}$, then $g\in P_{I_{w_1}\cup I_{w_2}}$.
\end{Lem}
\begin{proof}
By Bruhat decomposition, $g=b_1w_gb_2$ for some $b_1,b_2\in B$ and $w_g\in W$. Since the Schubert varieties are stable under the left multiplication by the Borel group $B$, we get $w_g\Sigma_{w_1}\subset \Sigma_{w_2}$. 

First assume $w_g=s$ is a simple reflection. We claim that $s\in P_1\cup P_2$. Notice that 
$$BsBw_1P=\begin{cases}
Bsw_1P \ \ \ \ \ \ \ \ \ \ \ \ \ \ \ \text{ if }l^P(sw_1)\geq l^P(w_1),\\
Bsw_1P\cup Bw_1P\ \ \ \text{ if }l^P(sw_1)<l^P(w_1).
\end{cases}$$
In either case $Bsw_1P\subset \overline{Bw_2P}$. Write $sw_1=w^Pw_P$ where $w^P\in W^P$ and $w_P\in W_P$. Then $l(sw_1)=l(w^P)+l(w_P)=l(w_1)\pm 1$. If $l(w_P)>0$ or $l(sw_1)=l(w_1)-1$, then $s\in P_1$. If $l(w_P)=0$ and $l(sw_1)=l(w_1)+1$, then $w_1<sw_1\leq w_2$ and $sw_1$ is a sub-expression of $w_2$. This implies $l(sw_2)<l(w_2)$ and therefore $s\in P_2$.

Now let $w_g=s_t...s_1$ be an irreducible expression of $w$ in terms of simple reflections. From the description of the orbit multiplication, it is easy to check
$$\overline{Bw_gB\cdot Bw_1P}=\overline{\bigcup\limits_{w'\leq w_g}Bw'w_1P}\subset\overline{Bw_2P}.$$
In particular, we have $Bs_iw_1P\subset\overline{Bw_2P}$ and $Bs_i...s_1w_1P\subset\overline{Bw_2P}$ for any $1\leq i\leq t$. By the previous argument, $s_i\in P_1\cup P_2$ for all $1\leq i\leq t$.
\end{proof}

\begin{Def}
Let $w\in W^P$. A simple root $\alpha$ is called {\em essential} relative to $w$ if the associating minimal parabolic subgroup does not stabilize the Schubert variety $\Sigma_w\in G/P$, or equivalently $w^{-1}(\alpha)\notin\Delta(P^-)$.
\end{Def}

\begin{Rem}
Consider the partial flag variety $F(d_1,....,d_k)$ identified with the homogeneous space $G/P$ as described in \S \ref{sec-prelim}. Let $\Phi=\{\beta_i|\beta_i=\epsilon_{i+1}-\epsilon_i, 1\leq i\leq n-1\}$ be the set of simple roots. Let $\sigma_{a^\alpha}$ be a Schubert class for $F(d_1,....,d_k)$ and let $w\in W^P$ be the corresponding coset in $W/W_P$. Then a sub-index $a_\gamma$ is essential with respect to $a^\alpha$ if and only if $\beta_{a_\gamma}$ is essential relative to $w$.
\end{Rem}

Let $E(w)$ be the set of essential roots relative to $w$.

\begin{Thm}\label{deduction}
Let $P$ be a parabolic subgroup with associated set of simple roots $I_P$. Let $w\in W^P$. The Schubert class $\sigma_w\in A^*(G/P)$ is rigid if
\begin{enumerate}
\item $(\pi_\alpha)_*(\sigma_w):=\sigma_{w_\alpha}$ are rigid for all $\alpha\in S:=\Phi\backslash I_P$, where $\pi_\alpha:G/P\rightarrow G/P_{\Phi\backslash\{\alpha\}}$ are natural projections; and
\item there exists an ordering on $S=\{\alpha_1,...\alpha_t\}$ such that 
$$E(w_{\alpha_i})\subset E(\pi_{\alpha_i*}\pi_{\alpha_{i+1}}^*(w_{\alpha_{i+1}})),\ \ 1\leq i\leq t-1$$
\end{enumerate}
\end{Thm}
\begin{proof}
Let $X$ be a representative of $\sigma_w$. Let $X_{\alpha_i}$ be the image of $X$ under $\pi_{\alpha_i}$, $1\leq i\leq t$. By Corollary \ref{maximalP}, it suffices to show that there exists $g\in G$ such that $gX_{\alpha_i}$ are $B$-stable Schubert varieties. We use induction on the cardinality $|S|$:

If $|S|=1$, then $\pi_\alpha$ is the identity morphism and the statement is trivial.

If $|S|=2$, by assumption, the Schubert classes $\sigma_{w_{\alpha_i}}$, $i=1,2$ are rigid, and therefore there exists $g_1,g_2\in G$ such that $X_{\alpha_i}=g_i\Sigma_{w_{\alpha_i}}$. Let $P_1$ and $P_2$ be the stabilizer of the Schubert variety $\Sigma_{w_{\alpha_1}}$ and $\Sigma_{w_{\alpha_2}}$ respectively. The statement then equivelent to find $g\in G$ such that $gg_1\in P_1$ and $gg_2\in P_2$, or equivalently $g_2^{-1}g_1\in P_2P_1$.

WLOG assume $g_2=1$ and $X_{\alpha_2}$ is already a standard Schubert variety. Let $Y:=\pi_{\alpha_1}(\pi_{\alpha_2}^{-1}(\Sigma_{w_{\alpha_2}}))$, which is also a Schubert variety. Let $P_3$ be the stabilizer of $Y$. By assumption, $P_1\supset P_3\supset P_2$. Since both $g_1\Sigma_{w_{\alpha_1}}$ and $\Sigma_{w_{\alpha_1}}$ are contained in $Y$, by Lemma \ref{translate of Schubert varieties}, $g_1\in P_1$.

If $|S|\geq 3$, assume by induction that $X_{\alpha_i}$ are Schubert varieties $\Sigma_{\alpha_i}$, $2\leq i\leq t$. Then the previous argument shows that $g_1\in P_1$ and therefore $X_{\alpha_1}$ is also a standard Schubert variety.
\end{proof}

\begin{Cor}
Let $P$ be a parabolic subgroup with associated set of simple roots $I_P$. Let $w\in W^P$. The Schubert class $\sigma_w\in A^*(G/P)$ is multi-rigid if
\begin{enumerate}
\item $(\pi_\alpha)_*(\sigma_w):=\sigma_{w_\alpha}$ are multi-rigid for all $\alpha\in S:=\Phi\backslash I_P$, where $\pi_\alpha:G/P\rightarrow G/P_{\Phi\backslash\{\alpha\}}$ are natural projections; and
\item there exists an ordering on $S=\{\alpha_1,...\alpha_t\}$ such that 
$$E(w_{\alpha_i})\subset E(\pi_{\alpha_i*}\pi_{\alpha_{i+1}}^*(w_{\alpha_{i+1}})),\ \ 1\leq i\leq t-1$$
\end{enumerate}
\end{Cor}
\begin{proof}
Let $X$ be an irreducible representative of $m\sigma_w$, $m\in\mathbb{Z}^+$. Then $X_{\alpha_i}:=\pi_{\alpha_i}(X)$ is also irreducible with class $m_i\sigma_{w_{\alpha_i}}$ for some $m_i\in\mathbb{Z}^+$. By assumption, $\sigma_{w_{\alpha_i}}$ are multi-rigid and therefore $m_i=1$ for all $1\leq i\leq t$. The statement then follows from an identical argument as in the proof of Theorem \ref{deduction}.
\end{proof}

\bibliographystyle{plain}

\end{document}